\documentclass[12pt,twoside]{amsart}
\usepackage[margin=3cm]{geometry}
\usepackage[colorlinks=false]{hyperref}
\usepackage[english]{babel}
\usepackage{graphicx,titling}
\usepackage{float}
\usepackage{amsmath,amsfonts,amssymb,amsthm}
\usepackage{lipsum}
\usepackage[T1]{fontenc}
\usepackage{fourier}
\usepackage{color}
\usepackage[latin1]{inputenc}
\usepackage{esint}
\usepackage{caption}
\usepackage{ccicons}

\makeatletter
\def\blfootnote{\xdef\@thefnmark{}\@footnotetext}
\makeatother

\newcommand\ccnote{
    \blfootnote{\copyright\,\, Alessandro Pigati and Daniel Stern}
    \blfootnote{\ccLogo\, \ccAttribution\,\, Licensed under a \href{https://creativecommons.org/licenses/by/4.0/}{Creative Commons Attribution License (CC-BY)}.}
}

\usepackage[export]{adjustbox}
\numberwithin{equation}{section}
\usepackage{setspace}
\renewcommand{\le}{\leqslant}
\renewcommand{\leq}{\leqslant}
\renewcommand{\ge}{\geqslant}
\renewcommand{\geq}{\geqslant}
\renewcommand{\mathbb}{\varmathbb}
\usepackage{fancyhdr}
\newtheorem{theorem}{Theorem}[section]
\newtheorem{lemma}[theorem]{Lemma}
\newtheorem{corollary}[theorem]{Corollary}
\newtheorem{proposition}[theorem]{Proposition}
\newtheorem{definition}[theorem]{Definition}
\newtheorem{remark}[theorem]{Remark}
\newtheorem{question}[theorem]{Question}
\usepackage[capitalise,nameinlink]{cleveref}
\crefformat{equation}{#2(#1)#3}
\newcommand{\N}{\mathbb{N}}
\newcommand{\Z}{\mathbb{Z}}
\newcommand{\R}{\mathbb{R}}
\newcommand{\C}{\mathbb{C}}
\newcommand{\de}{\partial}
\newcommand{\mz}{\frac{1}{2}}

\newcommand{\nin}{\not\in}

\newcommand{\mrestr}{\mathbin{\vrule height 1.6ex depth 0pt width
		0.13ex\vrule height 0.13ex depth 0pt width 1.3ex}}
\renewcommand{\bar}{\overline}
\DeclareMathOperator{\dist}{dist}

\newcommand{\ang}[1]{\langle #1\rangle}
\renewcommand{\epsilon}{\varepsilon}
\newcommand{\leps}{\lvert\log\epsilon\rvert}
\newcommand{\lteps}{\lvert\log\tilde\epsilon\rvert}

\address{Alessandro Pigati, New York University, Courant Institute of Mathematical Sciences, 251 Mercer Street, New York, NY 10003, United States of America.}
\email{ap6968@nyu.edu}
\medskip
\address{Daniel Stern, University of Chicago, Department of Mathematics, 5734 S University Ave, \linebreak Chicago, IL 60637, United States of America.} 
\email{dstern@uchicago.edu}

\begin{document}

\thispagestyle{empty}

\begin{minipage}{0.28\textwidth}
\begin{figure}[H]
\includegraphics[width=2.5cm,height=2.5cm,left]{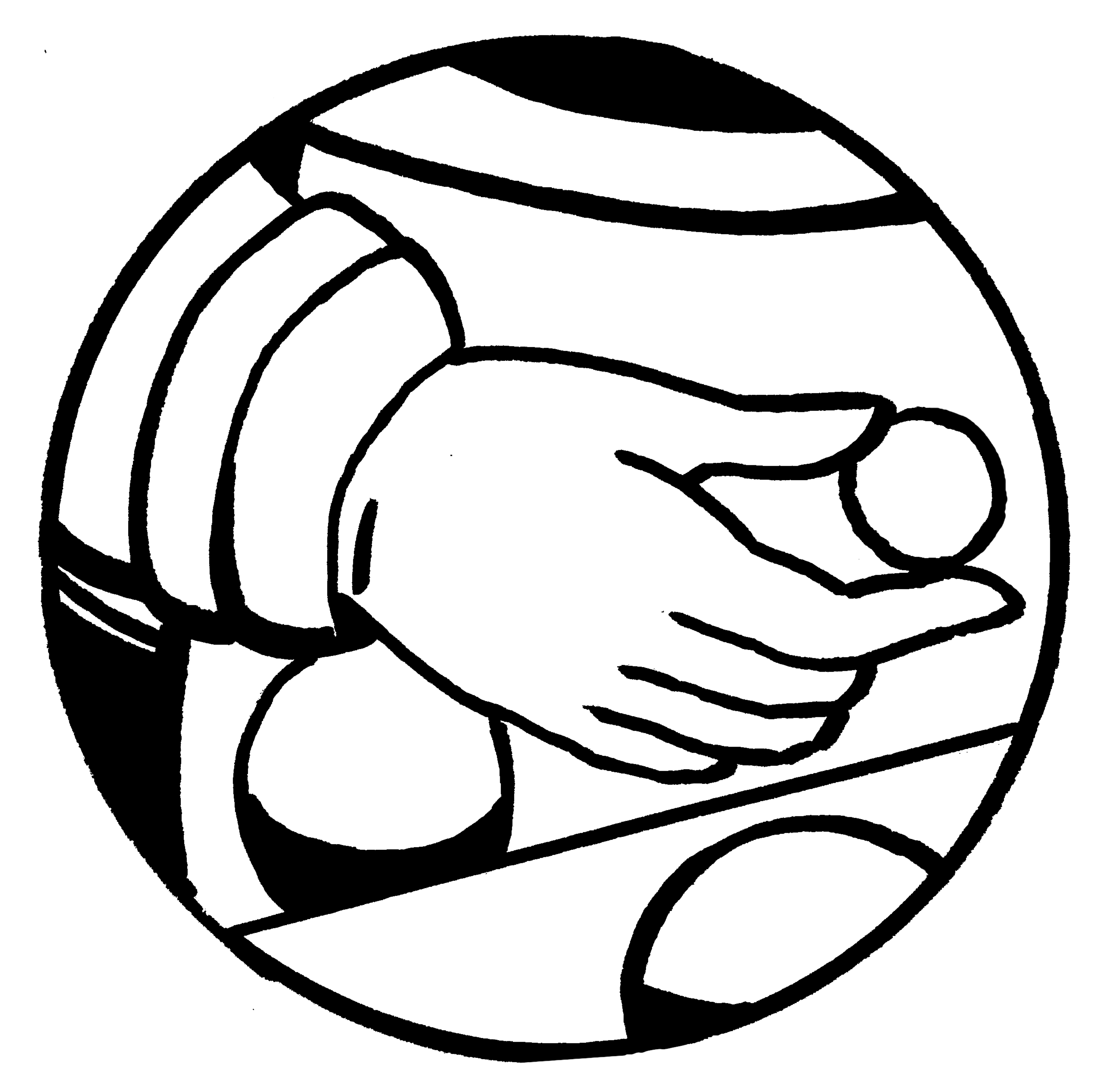}
\end{figure}
\end{minipage}
\begin{minipage}{0.7\textwidth} 
\begin{flushright}
Ars Inveniendi Analytica (2023), Paper No. 3, 55 pp.
\\
DOI 10.15781/g5bs-0m80
\\
ISSN: 2769-8505
\end{flushright}
\end{minipage}

\ccnote

\vspace{1cm}


\begin{center}
\begin{huge}
\textit{Quantization and non-quantization}

\textit{of energy for higher-dimensional}

\textit{Ginzburg--Landau vortices}

\end{huge}
\end{center}

\vspace{1cm}


\begin{minipage}[t]{.28\textwidth}
\begin{center}
{\large{\bf{Alessandro Pigati}}} \\
\vskip0.15cm
\footnotesize{Courant Institute of Mathematical Sciences}
\end{center}
\end{minipage}
\hfill
\noindent
\begin{minipage}[t]{.28\textwidth}
\begin{center}
{\large{\bf{Daniel Stern}}} \\
\vskip0.15cm
\footnotesize{University of Chicago}
\end{center}
\end{minipage}

\vspace{1cm}


\begin{center}
\noindent \em{Communicated by Peter Sternberg}
\end{center}
\vspace{1cm}


\noindent \textbf{Abstract.} \textit{Given a family of critical points
	$u_{\epsilon}:M^n\to\C$ for the complex Ginzburg--Landau energies
	\[ E_\epsilon(u)=\int_{M}\left(\frac{|du|^2}{2}+\frac{(1-|u|^2)^2}{4\epsilon^2}\right), \]
	on a manifold $M$, with natural energy growth $E_{\epsilon}(u_{\epsilon})=O(\leps )$, it is known that the vorticity sets $\{|u_\epsilon|\le\mz\}$ converge subsequentially to the support of a stationary, rectifiable $(n-2)$-varifold $V$ in the interior, characterized as the concentrated portion of the limit $\lim_{\epsilon\to 0} \frac{e_\epsilon(u_\epsilon)}{\pi\leps }$ of the normalized energy measures. 
	When $n=2$ or the solutions $u_{\epsilon}$ are energy-minimizing, it is known moreover that this varifold $V$ is \emph{integral}; i.e., the $(n-2)$-density $\Theta^{n-2}(|V|,x)$ of $|V|$ takes values in $\mathbb{N}$ at $|V|$-a.e.\ $x\in M$. In the present paper, we show that for a general family of critical points with $E_{\epsilon}(u_{\epsilon})=O(\leps )$ in dimension $n\geq 3$, this energy quantization phenomenon only holds where the density is less than $2$: namely, we prove that the density $\Theta^{n-2}(|V|,x)$ of the limit varifold takes values in $\{1\}\cup [2,\infty)$ at $|V|$-a.e.\ $x\in M$, and show that this is sharp, in the sense that for any $n\geq 3$ and $\theta\in \{1\}\cup [2,\infty)$, there exists a family of critical points $u_{\epsilon}$ for $E_{\epsilon}$ in the ball $B_1^n(0)$ with concentration varifold $V$ given by an $(n-2)$-plane with density $\theta$.}
\vskip0.3cm

\vspace{0.5cm}


\section{Introduction}

\subsection{The Ginzburg--Landau equations and the integrality question}
A complex valued map $u: M\to \mathbb{C}$ on a Riemannian manifold (or Euclidean domain) $M$ is said to satisfy the \emph{Ginzburg--Landau equations} with parameter $\epsilon>0$ if 
\begin{equation}\label{gl.eqn}
	\epsilon^2\Delta u=DW(u)=-(1-|u|^2)u,
\end{equation}
where $\Delta=-d^*d$ and $W: \mathbb{R}^2\to \mathbb{R}^2$ is the nonlinear potential $W(u)=\frac{1}{4}(1-|u|^2)^2$. The system \cref{gl.eqn} arises as the Euler--Lagrange equations for the energy functional
\begin{equation}\label{gl.en.def}
	E_{\epsilon}(u):=\int_Me_{\epsilon}(u)=\int_M\left(\frac{1}{2}|du|^2+\frac{W(u)}{\epsilon^2}\right),
\end{equation}
which combines the usual Dirichlet energy $\int \frac{1}{2}|du|^2$ with a nonlinear term $\int \frac{(1-|u|^2)^2}{4\epsilon^2}$ which penalizes the deviation of the values $u$ from the unit circle $S^1\subset \mathbb{C}$, with increasing severity as $\epsilon\to 0$.

While the study of the system \cref{gl.eqn} can be traced back to Ginzburg and Landau's work on superconductivity in the 1950s, the subject captured the attention of the geometric analysis community about thirty years ago, with the publication of the influential monograph \cite{BBH} by Bethuel--Brezis--H\'elein. The investigations of \cite{BBH, Struwe}, and others of this period focused on solutions $u_{\epsilon}:\Omega \to \mathbb{R}^2$ on simply connected planar domains $\Omega\subset \mathbb{R}^2$ obtained by minimizing $E_{\epsilon}$ with prescribed boundary data $g: \partial\Omega \to S^1$ of nonzero degree $\deg(g,\partial\Omega)\neq 0$, motivated by the search for a canonical `energy-minimizing' extension $u_*:\Omega\to S^1$ of $g$, in a setting where no finite-energy extension exists. It was shown that these maps $u_{\epsilon}$ converge as $\epsilon\to 0$ to a singular harmonic $S^1$-valued extension $u_*:\Omega\to S^1$ of $g$, whose singularities minimize a certain interaction energy between points in the plane. Moreover, these maps have energy $E_{\epsilon}(u_{\epsilon})=\pi \lvert\deg(g)\rvert\log(1/\epsilon)+O(1)$ as $\epsilon\to 0$, with the normalized energy measures $\frac{e_{\epsilon}(u_{\epsilon})}{\pi\log(1/\epsilon)}\,dx$ converging to a sum of Dirac masses at the $\lvert\deg(g)\rvert$ singular points of $u_*$. Non-minimizing critical points on two-dimensional domains were also studied, e.g., in \cite{Bar} and \cite{CM}.

Later, attention turned to solutions of \cref{gl.eqn} in dimension $n\geq 3$, with the work of Rivi\`ere \cite{Riv}, Lin--Rivi\`ere \cite{LR.jems, LR}, Jerrard--Soner \cite{JSoner}, Bethuel--Brezis--Orlandi \cite{BBO}, and others. For solutions $u_{\epsilon}:\Omega\subset\mathbb{R}^n\to \mathbb{C}$ of \cref{gl.eqn} in higher dimensional domains, satisfying the natural energy growth $E_{\epsilon}(u_{\epsilon})=O(\leps )$, it was shown that the zero sets $u_{\epsilon}^{-1}\{0\}$ converge (roughly speaking) to the support of a generalized minimal submanifold of codimension two. In particular, following the analysis of \cite{BBO} (see also \cite{Stern.thesis}), one arrives at the following asymptotic description of solutions as $\epsilon\to 0$.

\begin{theorem}\label{asymp}
	Given a manifold $M^n$ without boundary, of dimension $n\ge 3$,
	assume that we have a sequence of maps $u_\epsilon:M\to\C$ (indexed by a sequence $\epsilon\to 0$) solving \cref{gl.eqn}, with respect to a smoothly converging sequence of metrics $g_\epsilon\to g_0$, such that
	\begin{align*}
		&\limsup_{\epsilon\to 0}\frac{1}{\leps }\int_K\left(\mz|du_\epsilon|_{g_\epsilon}^2+\frac{W(u_\epsilon)}{\epsilon^2}\right) d\operatorname{vol}_{g_\epsilon}<\infty
	\end{align*}
	for all compact $K\subseteq M$. Then, up to a subsequence,
	the normalized energy densities
	\begin{align*}
		&\mu_\epsilon:=\frac{e_\epsilon(u_\epsilon)}{\pi\leps }\,\operatorname{vol}_{g_\epsilon},
		\quad\text{where } e_\epsilon(u_\epsilon):=\frac{|du_\epsilon|_{g_\epsilon}^2}{2}+\frac{W(u_\epsilon)}{\epsilon^2},
	\end{align*}
	converge to a Radon measure $\mu$ which decomposes as
	\begin{align}
		&\mu=|V|+f\operatorname{vol}_{g_0},
	\end{align}
	for a suitable smooth nonnegative function $f:M\to\R$ and a stationary rectifiable $(n-2)$-varifold with density $\Theta^{n-2}(|V|,\cdot)\ge c(n)>0$ on its support.
	Also, the measures $\frac{W(u_\epsilon)}{\epsilon^2}\,\operatorname{vol}_{g_\epsilon}$ converge to a limit measure satisfying
	\begin{align}\label{w.nu}
		&\lim_{\epsilon \to 0}\frac{W(u_\epsilon)}{\epsilon^2}\,\operatorname{vol}_{g_\epsilon}\le C(K)|V|
	\end{align}
	for all compact $K\subseteq M$.
	Finally, $\operatorname{spt}(|V|)$ is the limit of the sets $\{|u_\epsilon|\le\beta\}$ in the local Hausdorff topology, for any $\beta\in(0,1)$.
\end{theorem}

\begin{remark}
	Since the variants of \cref{asymp} appearing in \cite{BBO,Stern.thesis} are not quite stated in this form, we later include a sketch of the proof for the reader's convenience. We note also that the last statement is true even for $\beta=0$; in this case, it follows from some arguments contained in the present paper (see \cref{beta0} below). 
\end{remark}

The simple example $u_\epsilon(x)=\sqrt{1-\epsilon^2k_\epsilon^2}\,e^{ik_\epsilon x}$ on the circle $M=\R/2\pi\Z$, with $k_\epsilon\in\N$ satisfying $k_\epsilon\sim\sqrt{\leps }$,
shows that the limit measure $\mu$ can be completely diffuse. In \cite{Stern.jdg} (see also \cite{Cheng}), it was shown however that any closed Riemannian manifold $(M,g_0)$ of dimension $\geq 2$ admits a family of solutions satisfying the hypotheses of \cref{asymp} with $g_{\epsilon}=g_0$ for which the energy concentration varifold $V$ is nonzero, and it is expected that many such families exist. 

While results like \cref{asymp} reveal a strong link between solutions of \cref{gl.eqn} and minimal varieties of codimension two, the result sheds little light on the structure of the limit varifold. In particular, the weakest notion of minimal variety typically considered in geometric measure theory is not the stationary rectifiable varifold, but the slightly stronger stationary \emph{integral} varifold, which satisfies the additional condition that the density of its weight $\Theta^{n-2}(|V|,x)$ takes values in $\mathbb{N}$ for $|V|$-almost every $x$.

For some formally similar (though qualitatively rather different) families of equations like the Allen--Cahn equations or the self-dual abelian Higgs equations, results analogous to \cref{asymp} do indeed give energy concentration along stationary \emph{integral} varifolds of codimension one \cite{HT} and codimension two \cite{PS}, respectively. Moreover, the results of \cite{LR.jems} for $E_{\epsilon}$-minimizing solutions of \cref{gl.eqn} and \cite{CM} for general solutions in dimension two reveal that integrality of the limit varifold holds in these cases. All of these observations naturally lead us to the following question.

\begin{question}\label{integral.q} When $n\geq 3$, is the stationary varifold $V$ arising from a family of solutions to \cref{gl.eqn} as in \cref{asymp} necessarily integral? In other words, is the energy of $u_{\epsilon}$ quantized along the concentration set?
\end{question}

For an equivalent formulation, consider the set $\mathcal{D}\subset (0,\infty)$ of positive real numbers $\theta\in (0,\infty)$ for which there exists a family $u_{\epsilon}:B_1^n(0)\to \mathbb{R}^2$ of solutions to \cref{gl.eqn} in the unit $n$-ball whose energy concentrates along an $(n-2)$-plane $P\subset \mathbb{R}^n$ with (necessarily constant) density $\theta$, in the sense that
$$\mu_{\epsilon}\rightharpoonup^* \theta\mathcal{H}^{n-2}\mrestr P.$$
By a straightforward blow-up argument, it is easy to check that \cref{integral.q} has a positive answer if and only if $\mathcal{D}=\mathbb{N}\setminus\{0\}$.

In the present paper, we answer \cref{integral.q} in the negative, proving instead that
$$\mathcal{D}=\{1\}\cup[2,\infty)\supsetneq \mathbb{N}\setminus\{0\}.$$
In other words, we prove that the density $\Theta^{n-2}(|V|,\cdot)$ of the energy concentration varifold $V$ in \cref{asymp} takes values in $\{1\}\cup [2,\infty)$ almost everywhere, and give examples to show that this cannot be improved in general.

\subsection{Quantization and non-quantization results.}
The bulk of the paper is devoted the proof of the following theorem, showing that $\mathcal{D}\subseteq \{1\}\cup [2,\infty)$, and hence that the density of the limiting energy measure in \cref{asymp} is indeed quantized where $\Theta^{n-2}(|V|,\cdot)\leq 2$.

\begin{theorem}\label{dens.gap}
	In the setting of \cref{asymp}, assume moreover that $M=B_2^n(0)$ and $g_0$ is the Euclidean metric. If the energy densities concentrate along the plane $P=\mathbb{R}^{n-2}\times\{0\}$ with constant multiplicity $\theta\in(0,\infty)$, i.e.,
	$$\frac{e_{\epsilon}(u_\epsilon)}{\pi\leps }\,\operatorname{vol}_{g_\epsilon}\rightharpoonup \theta \mathcal{H}^{n-2}\mrestr P\cap B_2^n(0)$$
	in $C^0_c(M)^*$, 
	then $\theta \in \{1\}\cup [2,\infty)$.
\end{theorem}

In the general setting of \cref{asymp}, by applying \cref{dens.gap} to a family of rescaled solutions in balls centered at a point where the varifold $V$ has flat tangent cone, we arrive at the following corollary.

\begin{corollary}\label{dens.cor}
	Under the hypotheses of \cref{asymp}, the $(n-2)$-varifold $V$ has density
	$$\Theta^{n-2}(|V|,x)=\lim_{r\to 0}\frac{|V|(B_r(x))}{\omega_{n-2}r^{n-2}}\in \{1\}\cup [2,\infty)$$
	for $|V|$-a.e.\ $x$.
\end{corollary}

Previously, the best known lower bound for non-minimizing solutions $u_{\epsilon}$ in dimension $n\geq 3$ was the non-explicit lower bound $\Theta^{n-2}(|V|,\cdot)\geq c(n)>0$, a consequence of the following important result, obtained by different methods in \cite{LR} when $n=3$, and in \cite{BBO} for $n\ge 3$, which is the key ingredient in the proof of \cref{asymp}. In later works it was suggestively called \emph{clearing-out} for the vorticity. Here, for simplicity, we state it for the flat Euclidean metric.

\begin{theorem}\cite{LR, BBO}\label{clearing.flat}
	There exists a constant $\eta(n)>0$ such that, if
	$$E_\epsilon(u_\epsilon;B_r(x))\le\eta r^{n-2}\log(r/\epsilon),$$
	for a ball $B_r(x)$ in the domain with $r\ge\epsilon$, then $|u_\epsilon(x)|>\mz$.
\end{theorem}

By a trivial covering argument, up to changing the constant $\eta(n)$, one then obtains that $|u_\epsilon|>\mz$ on the entire smaller ball $B_{r/2}(x)$. When $u_\epsilon$ is a typical two-dimensional vortex centered at $x$, then the energy is $\sim r^{n-2}\log(r/\epsilon)$.
Thus, \cref{clearing.flat} essentially says that, if the energy is much smaller than the expected one, then indeed there cannot be any vortex on a smaller ball.

While it is possible to obtain explicit lower bounds for the energy threshold $\eta(n)$ using the arguments of \cite{BBO}, the resulting bounds are non-sharp. As a simple consequence of \cref{dens.cor}, we obtain the following sharp version of \cref{clearing.flat}.

\begin{corollary}\label{eta.sharp}
	For any $\eta<\pi\cdot \omega_{n-2}$ there exists $\delta(\eta,n)>0$ such that, if
	\begin{align*}
		&E_\epsilon(u_\epsilon;B_r(x))\le\eta r^{n-2}\log(r/\epsilon)
	\end{align*}
	and $\epsilon\le\delta r$, then $|u_\epsilon(x)|>\mz$.
\end{corollary}
For otherwise, since by scaling we can assume that $x=0$ and $r=1$, there would exist a sequence $u_\epsilon$ with $\epsilon\to 0$ and energy at most $\eta\leps $ on $B_1(0)$,
but with $|u_\epsilon(0)|\le\mz$. By \cref{asymp}, the point $0$ would belong to the support of the energy concentration varifold $V$. Since $V$ is stationary, \cref{dens.cor} and upper semicontinuity of density give $\Theta^{n-2}(|V|,0)\ge 1$, which gives $|V|(B_1(0))\ge\omega_{n-2}$ by the monotonicity formula.
However, this contradicts the fact that
$$|V|(B_1(0))\le\mu(B_1(0))\le\liminf_{\epsilon\to 0}\frac{E_\epsilon(u_\epsilon;B_1(0))}{\pi\leps }\le\frac{\eta}{\pi}<\omega_{n-2}.$$
Note that for any $\beta<1$ the same argument gives $|u_{\epsilon}(x)|>\beta$ provided that we assume $E_{\epsilon}(u_{\epsilon};B_r(x))\leq \eta r^{n-2}\log(r/\epsilon)$ for $\eta<\pi\omega_{n-2}$ and $\epsilon\leq \delta(\beta,\eta,n) r$.

\begin{remark} It seems likely that variants of Theorem \ref{dens.gap} and Corollary \ref{eta.sharp} should hold in the parabolic setting as well, yielding, e.g., a sharp version of \cite[Theorem 1]{BOSpar}.
\end{remark}

Building on these observations, one can also easily obtain sharp lower bounds on the energy of nontrivial solutions to the Ginzburg--Landau equations in several settings. For instance, one obtains the following sharp lower bound on the energy of nonconstant entire solutions, which was already shown in \cite{SS} when $n=3$ and $u$ is energy-minimizing.

\begin{corollary} For $n\geq 2$, any entire solution $u:\mathbb{R}^n\to \mathbb{C}$ of
	$$\Delta u=-(1-|u|^2)u$$
	for which
	\begin{equation}\label{growth.bd}
		\limsup_{R\to\infty}\frac{\int_{B_R(0)}(\frac{1}{2}|du|^2+W(u))}{R^{n-2}\log R}<\pi\omega_{n-2}
	\end{equation}
	must be a constant map $u\equiv e^{i\alpha}$ for some $\alpha\in [0,2\pi)$.
\end{corollary}
It is well known \cite{HH} that there exist nonconstant solutions for which equality holds. The proof is a straightforward consequence of the previous corollary: if the strict inequality \cref{growth.bd} holds, then the arguments of the preceding paragraph can be employed to show that $|u|\ge 1$ everywhere on $\mathbb{R}^n$.
Since we also have $|u|\le 1+CR^{-2}$ on $B_{R/2}$ (by rescaling the bound \cref{bound.general.critical} from the appendix), we obtain $|u|\equiv 1$.
Hence, $u$ is harmonic as a map to $\mathbb{R}^2$, which together with \cref{growth.bd} clearly shows that $u$ must be a constant map to $S^1$. 

The other main result gives a converse to \cref{dens.gap}, showing that $\{1\}\cup[2,\infty)\subseteq \mathcal{D}$, the novel observation here being that $[2,\infty)\setminus \mathbb{N}\subseteq \mathcal{D}$. 

\begin{theorem}\label{prescribe.dens}
	For any $\theta\in \{1\}\cup [2,\infty)$, there exists a family of solutions satisfying the hypotheses of \cref{dens.gap} (with $n \geq 3$ and $g_\epsilon=g_0$), with limit density $\theta$.
\end{theorem}
The examples provided by \cref{prescribe.dens} are obtained by scaling down certain entire solutions in $\R^3$ with helical symmetry, constructed in \cite{DDMR}.
In particular, we see that the conclusion of \cref{dens.gap} is sharp in dimension $\geq 3$, without additional constraints on the family of solutions.

\subsection{Proof ideas}
Unlike in the asymptotic analysis of the Allen--Cahn or $U(1)$-Higgs equations, where most of the energy concentrates at the $O(\epsilon)$ scale about the zero sets of solutions, the main contribution to the $\leps $ energy blow-up for solutions of the Ginzburg--Landau equations as in \cref{asymp} comes from the annular regions of distance $\epsilon^{1-\delta}\leq r\leq \epsilon^{\delta}$ about the zero set of a solution $u_{\epsilon}$ (for $\delta\in(0,\mz)$ small), where $u_{\epsilon}$ resembles a harmonic $S^1$-valued map. In particular, for any $\alpha\in (0,1)$, interactions between distinct components of the zero sets $u_{\epsilon}^{-1}\{0\}$ separated by a distance $\sim \epsilon^{\alpha}$ influence the leading-order behavior of the energy, and the key point in the proofs of \cref{dens.gap} and \cref{prescribe.dens} is to understand which kinds of interactions are permissible for solutions of \cref{gl.eqn} in higher dimension.

Given a family of solutions $u_{\epsilon}$ in $B_1^n(0)$ with energy concentrating on an $(n-2)$-plane $P=\mathbb{R}^{n-2}\times \{0\}$ as in \cref{dens.gap}, we show that the limiting multiplicity $\theta\in (0,\infty)$ for which
$$\frac{e_{\epsilon}(u_{\epsilon})}{\pi\log(1/\epsilon)}\,\operatorname{vol}_{g_\epsilon}\rightharpoonup^* \theta \mathcal{H}^{n-2}\mrestr P$$
can be computed via the following preliminary energy identity. After passing to a subsequence, for a generic sequence $y_{\epsilon}\in B_1^{n-2}$, the zero set $\{z\in D_1^2 : u_{\epsilon}(y_{\epsilon},z)=0\}$ of the solutions $u_{\epsilon}$ in the two-dimensional slice $\{y_{\epsilon}\}\times D_1^2$ is contained in a collection of $m$ disks $D_{C\epsilon}(z^{\epsilon}_1),\ldots, D_{C\epsilon}(z^{\epsilon}_m)$ of radius $O(\epsilon)$ with centers $z_1^{\epsilon},\ldots, z_m^{\epsilon}$. Denoting by $\kappa_j^{\epsilon}\in \mathbb{Z}$ the local degree 
$$\kappa_j^{\epsilon}:=\deg(u_{\epsilon}/|u_{\epsilon}|, \partial D_{C\epsilon}(z_j^{\epsilon}))\in \mathbb{Z},$$
we then find that (after passing to a subsequence)
\begin{equation}\label{intro.theta.comp}
	\theta=\lim_{\epsilon\to 0}\left({\textstyle\sum_{j=1}^m}(\kappa_j^{\epsilon})^2+2{\textstyle\sum_{1\leq i<j\leq m}}\kappa_i^{\epsilon}\kappa_j^{\epsilon}\frac{\lvert\log |z_i^{\epsilon}-z_j^{\epsilon}|\rvert}{\leps }\right).
\end{equation}

Note that if all of these degrees $\kappa_j^{\epsilon}$ had the same sign, it would follow from \cref{intro.theta.comp} that $\theta \geq m$, and the conclusion of \cref{dens.gap} would follow easily, since $\theta<2$ would imply that there is only $m=1$ such disk $D_{C\epsilon}(z_1^{\epsilon})$, with degree $(\kappa_1^{\epsilon})^2<2$, and therefore $\theta=(\kappa_1^{\epsilon})^2=1$. The difficulty in proving \cref{dens.gap} therefore lies in the case where the degrees $\kappa_j^{\epsilon}$ have different signs, so that the interaction terms $\kappa_i^{\epsilon}\kappa_j^{\epsilon}\frac{\lvert\log|z_i^{\epsilon}-z_j^{\epsilon}|\rvert}{\leps }$ subtract from the limiting density $\theta$. 

After some reductions, in the proof of \cref{dens.gap} we may assume that $|z_i^{\epsilon}|\leq \epsilon^{\delta}$ for some fixed $\delta>0$ for all $1\leq i,j\leq m$, and denoting by $\kappa$ the total degree
$$\kappa:={\textstyle\sum_{j=1}^m}\kappa_j,$$
we argue that (possibly after precomposing $u_{\epsilon}$ with a small translation) the energy density drop of $u_{\epsilon}$ between the scales $1$ and $\epsilon^{\delta}$ is given by
$$E_{\epsilon}(u_{\epsilon};B_1(0))-(\epsilon^{\delta})^{2-n}E_{\epsilon}(u_{\epsilon};B_{\epsilon^{\delta}}(0))=\pi\omega_{n-2}|\kappa|^2\log(1/\epsilon^{\delta})+o(\leps ).$$
By the well-known monotonicity formula for solutions of \cref{gl.eqn}, it then follows that 
$$\pi\omega_{n-2}|\kappa|^2\log(1/\epsilon^{\delta})+o(\leps )\geq \int_{\epsilon^{\delta}}^1\frac{2}{r^{n-1}}\int_{B_r(0)}\frac{W(u_{\epsilon})}{\epsilon^2},$$
and one easily concludes that there exists a sequence $r_{\epsilon}\in [\epsilon^{\delta},1]$ for which
\begin{equation}\label{w.ubd.intro}
	\frac{2}{r_{\epsilon}^{n-2}}\int_{B_{r_{\epsilon}}(0)}\frac{W(u_{\epsilon})}{\epsilon^2}\leq \pi\omega_{n-2}|\kappa|^2+o(1).
\end{equation}
Note that in the two-dimensional setting, a simple argument via a Pohozaev identity upgrades the inequality \cref{w.ubd.intro} to \emph{equality}, which forms the basis for the quantization results in \cite{CM}.

We then introduce new estimates relating the potential energy $\int_{B_r(0)}\frac{W(u_{\epsilon})}{\epsilon^2}$ to the degrees $\kappa_j^{\epsilon}$, concluding roughly that
\begin{equation}\label{w.lbd.intro}
	\frac{2}{r_{\epsilon}^{n-2}}\int_{B_{r_{\epsilon}}(0)}\frac{W(u_{\epsilon})}{\epsilon^2}\geq \frac{\pi\omega_{n-2}}{2}{\textstyle\sum_{j=1}^m}|\kappa_j^{\epsilon}|.
\end{equation}
Combining this with \cref{w.ubd.intro}, we deduce in particular that
$${\textstyle\sum_{j=1}^m}|\kappa^{\epsilon}_j|\leq 2|\kappa|^2=2\left|{\textstyle\sum_{j=1}^m}\kappa_j^{\epsilon}\right|^2.$$
On the other hand, if $\theta<2$, then the results of \cite{JSoner} imply that $\kappa=\pm 1$ or $0$, and by the preceding inequality, it follows that the only possibility is that, for some $1\leq i\leq m$, $\kappa_i^{\epsilon}=\pm 1$ and $\kappa_j^{\epsilon}=0$ for all $j\neq i$; hence, $\theta=1$ by \cref{intro.theta.comp}. This suffices for the proof of \cref{dens.gap}, showing that $\theta<2$ forces the vortex to occur with multiplicity one.

To prove \cref{prescribe.dens}, we observe that the families of entire solutions of \cref{gl.eqn} in $\mathbb{R}^3$ constructed in \cite{DDMR}, whose zero sets consist of $m\geq 2$ degree-one helices separated by a distance $\sim \frac{1}{\sqrt{\leps }}$ collapsing to the line $L=\{0\}\times \mathbb{R}$, can be blown down by a factor of $\epsilon^{\tau}$ for any fixed $\tau\in [0,1)$, to obtain a new family of solutions to \cref{gl.eqn} with parameter $\tilde{\epsilon}=\epsilon^{1+\tau}$. The zero sets of these blown-down solutions are then separated by a distance $\sim \epsilon^{\tau}\leps ^{-1/2}=\tilde{\epsilon}^{\frac{\tau}{1+\tau}}\leps ^{-1/2}$, and we can use \cref{intro.theta.comp} to deduce that the limiting energy measure on the line $L$ has density
$$\theta(m,\tau)={\textstyle\sum_{j=1}^m} 1+2{\textstyle\sum_{1\leq i<j\leq m}}1\cdot 1\cdot \frac{\tau}{\tau+1}=m+(m^2-m)\frac{\tau}{\tau+1}.$$
In particular, since
$$\{\theta(m,\tau)\mid \tau \in [0,1)\}=[m,{\textstyle\frac{1}{2}}(m^2+m) )$$
and $\bigcup_{m=2}^{\infty}[m,\frac{1}{2}(m^2+m))=[2,\infty)$, \cref{prescribe.dens} follows.

Note that the solutions constructed in \cite{DDMR} appear to be quite unstable at large scales; in particular, it should be possible to decrease the energy via a perturbation that spreads the $m$ components of the zero set farther apart. From a variational perspective, it would be very interesting to understand whether an additional assumption of \emph{stability} or \emph{bounded Morse index} of solutions allows one to refine the conclusions of \cref{dens.gap}, perhaps even giving a positive answer to \cref{integral.q} in this case. 

\section*{Acknowledgements} The authors thank Manuel Del Pino for helpful correspondence about the results of \cite{DDMR}, as well as Guido De Philippis, Aria Halavati, Fanghua Lin and Tristan Rivi\`ere, for their interest in this work and related discussions.
They are also grateful to the anonymous referees, whose careful reading and thoughtful comments helped improve the presentation.
DS acknowledges the support of the National Science Foundation under grant DMS-2002055.

\section{Preliminary estimates}

In this section we prove \cref{asymp} and, later, we obtain additional information in the special situation of \cref{dens.gap}.
While the proof of \cref{asymp} is simply a localization of some arguments from \cite{BBO,BOS,LR} and \cite{Cheng,Stern.jdg}, we summarize it here as a convenient way to fix some notation which will be used in the next sections.

\subsection{Proof of \cref{asymp}}
Since the statement is local, we can assume that $M=B_2^n(0)$, and prove that the conclusions hold on $B_{3/2}^n(0)$. In the appendix, we recall some fundamental estimates from \cite{BOS}, stating them in the case of a general metric. Thus, we are considering a family of solutions 
$$u_{\epsilon}: B_2^n(0)\to \mathbb{C}$$
to the complex Ginzburg--Landau equation
\begin{equation}\label{gl.e}
	\epsilon^2\Delta_{g_{\epsilon}} u_{\epsilon}=-(1-|u_{\epsilon}|^2)u_{\epsilon}
\end{equation}
with
$$\int_{B_2^n(0)}\left(\frac{|du_\epsilon|_{g_{\epsilon}}^2}{2}+\frac{W(u_\epsilon)}{\epsilon^2}\right)d\operatorname{vol}_{g_{\epsilon}}\le C\leps $$
for some $C>0$ independent of $\epsilon$.

First of all, from the local bounds stated in the appendix, it follows that
\begin{align}\label{u.bd}
	&|u_\epsilon|\le 1+C\epsilon^2,\quad |du_\epsilon|_{g_{\epsilon}}^2\le\frac{1-|u_\epsilon|^2}{\epsilon^2}+C
\end{align}
on the smaller ball $B_{7/4}=B_{7/4}^n(0)$, as well as
\begin{align}\label{w.bd}
	&\int_{B_{7/4}}\left(|d|u_\epsilon||_{g_{\epsilon}}^2+\frac{W(u_\epsilon)}{\epsilon^2}\right)d\operatorname{vol}_{g_{\epsilon}}\le C,
\end{align}
where, throughout the paper, $C$ denotes different constants which do not depend on $\epsilon$, but possibly on our sequence of solutions (we will, however, emphasize whether such constants depend on additional parameters introduced later on). Henceforth, we suppress the subscript $g_{\epsilon}$ in quantities depending on the metric when the meaning is clear from context, as well as the volume form.

As in the works quoted above, we introduce the one-forms
$$ju_{\epsilon}:=u_{\epsilon}^*(r^2\,d\theta)=u^1\,du^2-u^2\,du^1,$$
and observe that 
$$|du_{\epsilon}|^2=|d|u_{\epsilon}||^2+\frac{|ju_{\epsilon}|^2}{|u_{\epsilon}|^2}$$
on $\{u_\epsilon\neq 0\}$, and consequently
\begin{align}\label{du.ju}
	&||du_{\epsilon}|^2-|ju_{\epsilon}|^2|
	\leq |d|u_{\epsilon}||^2+|1-|u_{\epsilon}|^2||du_{\epsilon}|^2
	\leq |d|u_{\epsilon}||^2+\frac{(1-|u_{\epsilon}|^2)^2}{\epsilon^2}+C.
\end{align}
Hence, it follows from \cref{w.bd} that
\begin{equation}\label{ju.main}
	\int_{B_{7/4}}\left|e_{\epsilon}(u_{\epsilon})-\frac{1}{2}|ju_{\epsilon}|^2\right|\leq C.
\end{equation}

Note moreover that we have
\begin{equation}
	d^*ju_{\epsilon}=0,
\end{equation}
as an easy consequence of \cref{gl.e}.
Now, for each $u_{\epsilon}$, we denote by $\mathcal{V}(u_{\epsilon})\subseteq B_2$ the \emph{vorticity set}
$$\mathcal{V}(u_\epsilon):=\{|u_\epsilon|\leq \tfrac{1}{2}\},$$
and define a perturbed map $v_{\epsilon}:B_2\to \mathbb{C}$ by
$$v_{\epsilon}:=\chi(|u_\epsilon|)u_\epsilon,$$
where $\chi:\R\to\R$ is smooth and such that $\chi(t)=1$ on $[0,\frac{1}{4}]$ and $\chi(t)=\frac{1}{t}$ on $[\mz,\infty)$.
In particular,
$$v_{\epsilon}(x)=\frac{u_\epsilon(x)}{|u_\epsilon(x)|}\in S^1\quad\text{for }x\in B_2\setminus \mathcal{V}(u_{\epsilon})$$
and $|v_{\epsilon}|\le C|u_{\epsilon}|$ on $\mathcal{V}(u_{\epsilon})$.
As in \cite{BBO,BOS,LR}, a suitable Hodge decomposition of the one-forms
\begin{equation}\label{jv.def}
	jv_{\epsilon}=v_{\epsilon}^*(r^2\,d\theta)=\chi(|u_{\epsilon}|)^2ju_{\epsilon}
\end{equation}
plays a central role in the analysis. To obtain the exact part of the decomposition, first consider a solution $\psi_{\epsilon}\in C^{\infty}(B_{7/4})$ to the boundary value problem
\begin{align}\label{psi.def}
	\left\{
	\begin{aligned}
		d^*d\psi_{\epsilon}&=d^*(jv_{\epsilon})=d^*(jv_{\epsilon}-ju_{\epsilon})&&\quad\text{in }B_{7/4}, \\
		\psi_\epsilon&=0&&\quad\text{on }\partial B_{7/4}.
	\end{aligned}
	\right.
\end{align}
Note then that
\begin{align*}
	\int_{B_{7/4}}|d\psi_{\epsilon}|^2&\leq \int_{B_{7/4}}|jv_{\epsilon}-ju_{\epsilon}|^2\\
	&\leq \int_{B_{7/4}}|\chi(|u_{\epsilon}|)^2-1|^2|u_{\epsilon}|^2|du_{\epsilon}|^2\\
	&\leq C\int_{B_{7/4}}|1-|u_{\epsilon}|^2||du_{\epsilon}|^2\\
	\text{(by \cref{u.bd}) }&\leq C\int_{B_{7/4}}\frac{W(u_{\epsilon})}{\epsilon^2}+C,
\end{align*}
which together with \cref{w.bd} gives
\begin{equation}\label{dpsi.est}
	\|jv_\epsilon-ju_\epsilon\|_{L^2(B_{7/4})}+\|d\psi_{\epsilon}\|_{L^2(B_{7/4})}\leq C.
\end{equation}

Next, let $\varphi\in C_c^{\infty}(B_{7/4})$ be a cutoff function with $B_{3/2}\subset\subset \{\varphi = 1\}$, and consider the two-form
\begin{align}\label{xi.def}
	&\xi_{\epsilon}:=\Delta_H^{-1}(\varphi\,djv_{\epsilon})=G*(\varphi\,djv_{\epsilon}),
\end{align}
given by convolution of $djv_\epsilon=2dv_\epsilon^1\wedge dv_\epsilon^2$ (multiplied by the cutoff $\varphi$) where
$$(G*\zeta)(x):=\sum_{i\in I}\int_{p\in B_{7/4}} G_{i,p}(x)\langle\zeta(p),\omega_i(p)\rangle\,d\operatorname{vol}_{g_{\epsilon}}$$
is the local Green's operator for the Hodge Laplacian $\Delta_H=d^*d+dd^*$ with respect to the metric $g_{\epsilon}$ as described in \cref{green} (with $U:=B_2$ and $K:=\bar B_{7/4}$), so that 
$$\Delta_H\xi_{\epsilon}=d^*d\xi_{\epsilon}+dd^*\xi_{\epsilon}=\varphi\,djv_{\epsilon}.$$

It is easy to see that $\varphi\,djv_\epsilon$ is supported in $\mathcal{V}(u_{\epsilon})\cap B_{7/4}(0)$, where
\begin{align}\label{jv.w}
	&|djv_\epsilon|\leq C|du_{\epsilon}|^2\leq C\frac{W(u_{\epsilon})}{\epsilon^2}
\end{align}
(since $|du_\epsilon|\le\frac{C}{\epsilon}$ and $1-|u_{\epsilon}|^2\geq \frac{3}{4}$ on $\mathcal{V}(u_{\epsilon})$).
In particular, using \cref{green} to bound $|\nabla G_{i,p}|(x)\leq C\operatorname{dist}_{g_{\epsilon}}(x,p)^{1-n}$, we have
\begin{equation}\label{xi.conv}
	|\xi_\epsilon|(x)+|\nabla \xi_{\epsilon}|(x)\leq C\int_{\mathcal{V}(u_{\epsilon})\cap B_{7/4}}\operatorname{dist}_{g_{\epsilon}}(x,y)^{1-n}\frac{W(u_{\epsilon}(y))}{\epsilon^2}\,dy,
\end{equation}
and as an easy consequence,
\begin{equation}\label{nabxi.l1}
	\|\xi_{\epsilon}\|_{L^p(B_{7/4})}+\|\nabla\xi_{\epsilon}\|_{L^p(B_{7/4})}\leq C(p)\int_{B_{7/4}}\frac{W(u_{\epsilon})}{\epsilon^2}\leq C(p)
\end{equation}
for any $p\in[1,\frac{n}{n-1})$.

Finally, letting
\begin{equation}\label{h.def}
	h_{\epsilon}:=jv_{\epsilon}-d^*\xi_{\epsilon}-d\psi_{\epsilon},
\end{equation}
observe that $h_\epsilon$ is harmonic on the interior of $\{\varphi=1\}$, since here
\begin{align*}
	\Delta_H h_{\epsilon}&=d^*d(jv_{\epsilon}-d^*\xi_{\epsilon})+dd^*(jv_{\epsilon}-d\psi_{\epsilon})\\
	&=d^*(djv_{\epsilon}-\Delta_H\xi_{\epsilon})\\
	&=d^*(djv_{\epsilon}-\varphi\,djv_{\epsilon})\\
	&=0.
\end{align*}
In particular, elliptic estimates give
\begin{align*}
	&\|h_{\epsilon}\|_{C^1(B_{3/2})}
	\le C\|h_{\epsilon}\|_{L^1(B_{7/4})},
\end{align*}
and using \cref{dpsi.est} and \cref{nabxi.l1}, we deduce that
\begin{align}\label{h.bd}
	&\|h_{\epsilon}\|_{C^1(B_{3/2})}\le\|jv_\epsilon\|_{L^1(B_{7/4})}+C\le C\leps ^{1/2}.
\end{align}
Using again \cref{dpsi.est}, \cref{h.def}, and the trivial bound $\|ju_\epsilon\|_{L^2(B_2)}\le C\leps^{1/2}$, this also implies
\begin{align}\label{nabxi.l2}
	&\|d^*\xi_\epsilon\|_{L^2(B_{3/2})}\le \|jv_\epsilon\|_{L^2(B_{3/2})}+C+C\leps^{1/2}\le C\leps^{1/2}.
\end{align}

Now, let $S$ be the (subsequential) limit of the sets $\mathcal{V}(u_\epsilon)$, in the Hausdorff topology on $B_2$. (Note that the metrics $g_{\epsilon}$ in the family are uniformly equivalent to the Euclidean metric $\delta$, i.e., $C^{-1}g_{\epsilon}\leq \delta\leq Cg_{\epsilon}$ on the ball $B_{7/4}$, so Hausdorff convergence can be considered with respect to the Euclidean metric.)
This set will be useful in the proof of the following statement.

\begin{lemma}
	As $\epsilon\to 0$, we have
	\begin{align}\label{orth}
		&\lim_{\epsilon \to 0}\frac{1}{\leps }\int_{B_{3/2}}|d^*\xi_\epsilon||h_\epsilon|=0.
	\end{align}
\end{lemma}

\begin{proof}
	If $x\in S$ then we can find $x_\epsilon\in\mathcal{V}(u_\epsilon)$ such that $x_\epsilon\to x$,
	and by \cref{clearing} we then have
	\begin{align}\label{dens.lb.K}
		&\mu(\bar B_r(x))\ge\limsup_{\epsilon\to 0}\mu_\epsilon(B_{r-|x_\epsilon-x|}(x_\epsilon))\ge c(n)\lim_{\epsilon\to 0}(r-|x_\epsilon-x|)^{n-2}=c(n)r^{n-2}
	\end{align}
	for any $r<2-|x|$. By a simple Vitali covering argument, this implies that
	\begin{align}\label{loc.hau}
		&\mathcal{H}^{n-2}\mrestr S\le C(n)\mu,
	\end{align}
	and in particular $S$ is negligible with respect to the Lebesgue measure.
	
	Now, for any $\delta>0$, denoting by $B_{\delta}(S)$ the $\delta$-neighborhood of $S$, we can bound
	$$\int_{B_{3/2}}|d^*\xi_\epsilon||h_\epsilon|
	\le\int_{B_{3/2}\cap B_{\delta}(S)}|d^*\xi_\epsilon||h_\epsilon|
	+\|d^*\xi_\epsilon\|_{L^\infty(B_{3/2}\setminus B_\delta(S))}\|h_\epsilon\|_{L^1(B_{3/2})}.$$
	By Cauchy--Schwarz, \cref{h.bd}, and \cref{nabxi.l2}, the first term is bounded by
	$$\|d^*\xi_\epsilon\|_{L^2(B_{3/2})}|B_{\delta}(S)|^{1/2}\cdot \|h_\epsilon\|_{L^\infty(B_{3/2})}
	\le C\leps |B_\delta(S)|^{1/2}.$$
	Moreover, recalling the definition of $S$, \cref{xi.conv} gives
	\begin{align}\label{xi.van}
		&\limsup_{\epsilon\to 0}\|d^*\xi_\epsilon\|_{L^\infty(B_{3/2}\setminus B_\delta(S))}\le C(\delta)\int_{B_{7/4}}\frac{W(u_\epsilon)}{\epsilon^2}\le C(\delta),
	\end{align}
	which implies that the second term above is at most $C(\delta)\leps ^{1/2}$, and \cref{orth} follows by letting $\epsilon\to 0$ and then $\delta\to 0$.
\end{proof}

By \cref{ju.main}, \cref{dpsi.est}, and \cref{h.def}, we have
\begin{align*}
	&\mu=\lim_{\epsilon\to 0}\frac{|jv_\epsilon|^2}{2\pi\leps }\,dx
	=\lim_{\epsilon\to 0}\frac{|d\psi_\epsilon+d^*\xi_\epsilon+h_\epsilon|^2}{2\pi\leps }\,dx
	=\lim_{\epsilon\to 0}\frac{|d^*\xi_\epsilon+h_\epsilon|^2}{2\pi\leps }\,dx.
\end{align*}
The previous lemma, together with \cref{h.bd} and \cref{nabxi.l2}, implies that
\begin{align*}
	&\mu=\lim_{\epsilon\to 0}\frac{|d^*\xi_\epsilon+h_\epsilon|^2}{2\pi\leps }\,dx=\nu+\frac{1}{2\pi}|h_0|^2\,dx\quad\text{on }B_{3/2}
\end{align*}
up to subsequences,
where $h_0:=\lim_{\epsilon \to 0}\frac{h_\epsilon}{\leps ^{1/2}}$ is a harmonic one-form, while
$$\nu:=\lim_{\epsilon \to 0}\frac{|d^*\xi_\epsilon|^2}{2\pi\leps }\,dx.$$
From \cref{xi.van} it follows that $\operatorname{spt}(\nu)\subseteq S$, while \cref{dens.lb.K} and the structure of $\mu$ imply
\begin{align}\label{nu.lb}
	&\lim_{r\to 0}\frac{\nu(B_r(x))}{r^{n-2}}\ge c(n)>0,
\end{align}
which forces the reverse inclusion to hold on $B_{3/2}$. Thus,
\begin{align*}
	&S=\operatorname{spt}(\nu)\quad\text{on }B_{3/2}.
\end{align*}
Note that the previous argument also shows that $\lim_{\epsilon \to 0}\{|u_\epsilon|\le\beta\}=\operatorname{spt}(\nu)$ on $B_{3/2}$ for any $\beta\in(0,1)$
(without the need of a further subsequence, as any subsequential limit of $\{|u_\epsilon|\le\beta\}$ equals $\operatorname{spt}(\nu)$).

To prove \cref{w.nu}, define the measure $\lambda:=\lim_{\epsilon\to 0}\frac{W(u_\epsilon)}{\epsilon^2}\,dx$, which exists up to subsequences by \cref{w.bd}.
Note that, by the monotonicity formula \cref{monotonicity}, the rescaled maps $\tilde u_{\tilde\epsilon}(x):=u_\epsilon(p+rx)$ (with $\tilde\epsilon=\epsilon/r$) have energy at most $C\leps $ on $B_1(0)$,
for $p\in B_{3/2}(0)$ and $r<\frac{1}{4}$.
Applying \cref{bos2} and scaling back, it follows that
$r^{2-n}\lambda(B_{r/2}(p))\le C$. Also, from \cite[Theorem~2.1]{BOS} it easily follows that $\lambda=0$
on $B_{3/2}\setminus S$ (since $S$ is the limit of the sets $\{|u_\epsilon|\le\beta\}$ for any $\beta\in(0,1)$).
Hence,
\begin{align*}
	&\lambda\le C\mathcal{H}^{n-2}\mrestr S
\end{align*}
on $B_{3/2}$, while \cref{loc.hau} and the structure of $\mu$ imply that the right-hand side is bounded above by $C\nu$.

This proves \cref{w.nu} and the theorem, provided that we check that $\nu=\mu\mrestr S$ coincides with the weight of a stationary $(n-2)$-varifold $V$ for the limit metric $g_0=\lim_{\epsilon\to 0}g_{\epsilon}$.
On $B_{3/2}$ we introduce the \emph{stress-energy tensor}
\begin{align*}
	&T_\epsilon:=e_\epsilon(u_\epsilon)I-du_\epsilon\otimes du_\epsilon,
\end{align*}
with the implicit scalar product $du_\epsilon\otimes du_\epsilon=du_\epsilon^1\otimes du_\epsilon^1+du_\epsilon^2\otimes du_\epsilon^2$, and define the measure
\begin{align*}
	&T_0=\lim_{\epsilon \to 0}\frac{T_\epsilon}{\pi\leps }\,\operatorname{vol}_{g_{\epsilon}},
\end{align*}
taking values in symmetric bilinear forms.
The fact that $u_\epsilon$ is critical with respect to inner variations gives $\operatorname{div}_{g_{\epsilon}}T_\epsilon=0$,
which implies that $T_0$ is also divergence-free, in the sense that the pairing $\langle T_0,\nabla_{g_0} X\rangle$ vanishes for any $C^1$ vector field $X$ supported in $B_{3/2}$.

A computation similar to \cref{du.ju}, together with \cref{w.bd}, shows that
\begin{align*}
	&T_0=\lim_{\epsilon \to 0}\frac{1}{\pi\leps }\left(\frac{|ju_\epsilon|^2}{2}I-ju_\epsilon\otimes ju_\epsilon\right)\operatorname{vol}_{g_{\epsilon}}.
\end{align*}
Also, \cref{dpsi.est} and \cref{orth} give
\begin{align*}
	&T_0=V+\left(\frac{|h_0|^2}{2}I-h_0\otimes h_0\right)\operatorname{vol}_{g_0},
\end{align*}
with
\begin{align*}
	&V:=\lim_{\epsilon \to 0}\frac{1}{\pi\leps }\left[\frac{|d^*\xi_\epsilon|^2}{2}I-(d^*\xi_\epsilon)\otimes (d^*\xi_\epsilon)\right]\,\operatorname{vol}_{g_{\epsilon}}.
\end{align*}
Note that $h_0$ is strongly harmonic, meaning that $dh_0=0$ and $d_{g_0}^*h_0=0$:
indeed, we already have $d_{g_{\epsilon}}^*h_\epsilon=d_{g_{\epsilon}}^*(jv_\epsilon-d\psi_\epsilon)=0$ by \cref{psi.def};
also, \eqref{nabxi.l1} gives $\|d\xi_\epsilon\|_{L^1(B_{3/2})}\le C$, and hence
$$ dh_0=\lim_{\epsilon\to 0}\frac{dh_\epsilon}{\leps^{1/2}}=\lim_{\epsilon\to 0}\frac{djv_\epsilon-dd^*\xi_\epsilon}{\leps^{1/2}}
=\lim_{\epsilon\to 0}\frac{\Delta_H\xi_\epsilon-dd^*\xi_\epsilon}{\leps^{1/2}}=\lim_{\epsilon\to 0}\frac{d^*d\xi_\epsilon}{\leps^{1/2}}=0 $$
on $B_{3/2}$, with the limits understood distributionally (where we used \cref{h.def} in the second equality and \cref{xi.def} in the third one).
Since $h_0$ is strongly harmonic, the term $\frac{|h_0|^2}{2}I-h_0\otimes h_0$ is divergence-free, and thus $\operatorname{div}V=0$, as well.

Since $\operatorname{tr}(V)\ge(n-2)\nu$ and $|\langle V w,w\rangle|\le|w|^2\nu$ for any $w\in\R^n$, the measure $V$ can be identified with a generalized stationary $(n-2)$-varifold with weight $\nu$, according to the definition from \cite[Section~A.2]{Stern.jdg}.
Since it has positive density on its support by \cref{nu.lb}, it now follows from the classical result by Ambrosio--Soner \cite[Theorem~3.8]{AS} that $V$ is actually a rectifiable varifold.

\subsection{Additional bounds in the situation of \cref{dens.gap}}
Suppose now that we are in the setting of \cref{dens.gap}. Henceforth, we will assume for simplicity of notation that $g_\epsilon$ is in fact equal to the flat Euclidean metric; it is an easy exercise to extend the arguments to metrics converging smoothly to the Euclidean metric, and we will comment occasionally on the necessary modifications for this case. Thus, we are considering a family of solutions $u_{\epsilon}: B_2^n(0)\to \mathbb{C}$ to \cref{gl.e},
for which the normalized energy measures
$$\mu_{\epsilon}:=\frac{e_{\epsilon}(u_{\epsilon})}{\pi\leps }\,dx$$
converge weakly in $C^0_c(B_2)^*$ to a multiple of the $(n-2)$-plane $P=\mathbb{R}^{n-2}\times \{0\}$
$$\mu_{\epsilon}\rightharpoonup \theta \mathcal{H}^{n-2}\mrestr P$$
as $\epsilon\to 0$. 

On any domain compactly contained in $B_2=B_2(0)$, such as $B_{7/4}$, the following is a simple consequence of the last assumption.

\begin{lemma}\label{par.van.lem}
	Writing $|du_{\epsilon}(P)|^2:={\textstyle\sum_{i=1}^{n-2}}|du_{\epsilon}(e_i)|^2$, we have
	\begin{equation}\label{par.van}
		\lim_{\epsilon\to 0}\frac{1}{\leps }\int_{B_{7/4}}|du_{\epsilon}(P)|^2=0.
	\end{equation}
\end{lemma}

\begin{proof}
	Since the stationary varifold $V$ from \cref{asymp} coincides with a multiple of $P$, viewing $V$ as a matrix-valued measure we can write
	\begin{align}\label{V.matrix}
		&V=\theta M\,\mathcal{H}^{n-2}\mrestr P,
	\end{align}
	where $M\in\R^{n\times n}$ is the orthogonal projection onto $P$. As seen in the proof of \cref{asymp}, the normalized stress-energy tensors $\frac{T_\epsilon}{\pi\leps }$ converge to $V$, and by \cref{V.matrix} this implies that
	$$\lim_{\epsilon \to 0}\int_{B_2}\chi\frac{\langle T_\epsilon w,w\rangle}{\pi\leps }\,dx=\int_{B_2}\chi\,d\langle Vw,w\rangle=\int_{B_2}\chi\,d|V|$$
	for any $\chi\in C^0_c(B_2)$ and any unit vector $w\in P$. Recalling the definition of $T_\epsilon$ and the fact that
	\begin{align*}
		&\lim_{\epsilon \to 0}\int_{B_2}\chi\frac{e_\epsilon(u_\epsilon)}{\pi\leps }=\int_{B_2}\chi\,d|V|,
	\end{align*}
	we deduce that $\lim_{\epsilon \to 0}\int_{B_2}\chi\frac{|du_\epsilon(w)|^2}{\pi\leps }=\lim_{\epsilon\to 0}\int_{B_2}\chi\frac{e_{\epsilon}(u_{\epsilon})-\langle T_{\epsilon}w,w\rangle}{\pi \leps}=0$, as desired.
\end{proof}

Using the preceding bounds, we can prove the following key estimates, showing that the limiting energy density can be computed in terms of the one-form $d^*\xi_{\epsilon}$. 

\begin{lemma}\label{coex.dom} As $\epsilon\to 0$, we have
	$$\lim_{\epsilon\to 0} \frac{1}{\leps }\int_{B_{3/2}}|ju_{\epsilon}-d^*\xi_{\epsilon}|^2=0.$$
	In particular, combining with \cref{ju.main} gives
	\begin{equation}\label{en.is.coex}
		\lim_{\epsilon\to 0}\frac{1}{\leps }\int_{B_{3/2}}\left|e_{\epsilon}(u_{\epsilon})-\frac{1}{2}|d^*\xi_{\epsilon}|^2\right|=0.
	\end{equation}
\end{lemma}

\begin{proof} In view of \cref{dpsi.est}, it suffices to show that
	\begin{align}\label{jv.xi}
		&\|jv_{\epsilon}-d^*\xi_{\epsilon}\|_{L^2(B_{3/2})}^2=o(\leps )
	\end{align}
	as $\epsilon\to 0$. Note that
	$$jv_{\epsilon}-d^*\xi_{\epsilon}=h_{\epsilon}+d\psi_{\epsilon},$$
	and we know from \cref{dpsi.est} that $\|d\psi_{\epsilon}\|_{L^2(B_{7/4})}^2\leq C$, so all that remains is to show that
	$$\|h_{\epsilon}\|_{L^2(B_{3/2})}^2=o(\leps ).$$
	Since the energy concentrates along $P$, note that, for any fixed $\delta>0$,
	\begin{align*}
		\limsup_{\epsilon\to 0}\frac{\|du_{\epsilon}\|_{L^1(B_{7/4})}}{\leps ^{1/2}}
		&=\limsup_{\epsilon\to 0} \frac{\|du_{\epsilon}\|_{L^1(B_{7/4}\cap B_{\delta}(P))}}{\leps ^{1/2}}\\
		&\leq \limsup_{\epsilon\to 0}\frac{\|du_{\epsilon}\|_{L^2(B_{7/4}\cap B_{\delta}(P))}}{\leps ^{1/2}}|B_{7/4}\cap B_{\delta}(P)|^{1/2}\\
		&\leq C\delta,
	\end{align*}
	so that $\|du_{\epsilon}\|_{L^1(B_{7/4})}^2=o(\leps )$. Using \cref{h.bd}, we arrive at
	$$\|h_{\epsilon}\|_{L^2(B_{3/2})}\le\|ju_\epsilon\|_{L^1(B_{7/4})}+C=o(\leps ^{1/2}),$$
	and the claim follows.
\end{proof}

Now, denote by $Q$ the cylinder
$$Q:=B_1^{n-2}(0)\times D_1^2(0)\subset B_{3/2}^n(0),$$
and fix a large constant $K>0$, which will be specified in the final part of the proof. In what follows, we identify a distinguished family of two-dimensional slices perpendicular to the $(n-2)$-plane $P$ of concentration, such that the energy density $\theta$ can be computed in terms of the behavior of the solutions $u_{\epsilon}$ along these slices. (Cf., e.g., \cite{LR.hm} or \cite{Riv.ym} for similar arguments in the harmonic map or Yang--Mills settings.)

\begin{definition}
	Given $y\in B_{1/2}^{n-2}(0)$, we say that $P_y^{\perp}:=\{y\}\times D_1^2(0)$ (or simply $y$) is a \emph{$\delta$-good slice} for $u_\epsilon$ if 
	$$\sup_{0<r<1/2}\left|r^{2-n}\int_{B_r^{n-2}(y)\times D_1^2}e_\epsilon(u_\epsilon)-\pi\omega_{n-2}\theta\leps \right|<\delta\leps ,$$
	$$\sup_{0<r<1/2}r^{2-n}\int_{B_r^{n-2}(y)\times D_1^2}\left(|du_{\epsilon}(P)|^2+|jv_\epsilon-d^*\xi_\epsilon|^2+\left|e_\epsilon(u_\epsilon)-\mz|d^*\xi_\epsilon|^2\right|\right)<\delta\leps ,$$
	and
	$$\sup_{0<r<1/2}r^{2-n}\int_{B_r^{n-2}(y)\times D_1^2}\left(\frac{W(u_{\epsilon})}{\epsilon^2}+|\xi_\epsilon|\right)<K.$$
\end{definition}
The first and second conditions require uniform $L^2$ vanishing of $du_{\epsilon}(P)$ on the cylinders $B_r^{n-2}(y)\times D_1^2$ centered at $y$ at all small scales $r$, and assert that the density $\theta$ can be computed by integrating any one of $e_{\epsilon}(u_{\epsilon})$, $\frac{1}{2}|d^*\xi_{\epsilon}|^2$, or $\frac{1}{2}|jv_{\epsilon}|^2$ along the cylinders $B_r^{n-2}(y)\times D_1^2$, or (letting $r\to 0$) along the slice $P_y^{\perp}=\{y\}\times D_1^2$, while the third condition enforces uniform $L^1$ control on the potential $\frac{W(u)}{\epsilon^2}$ and $|\xi_{\epsilon}|$ over the same cylinders at all scales.

\begin{lemma}\label{many.good.slices}
	Denote by $\mathcal{G}_{\epsilon,\delta}\subseteq B_{1/2}^{n-2}(0)$ the collection of $\delta$-good slices for $u_{\epsilon}$. Then,
	for any $\delta>0$, we have
	$$\limsup_{\epsilon\to 0}|B_{1/2}^{n-2}\setminus \mathcal{G}_{\epsilon,\delta}|\le \frac{C}{K}.$$
\end{lemma}

\begin{proof}
	Let $F_\epsilon^1$, $F_\epsilon^2$, and $F_\epsilon^3$ be the sets of slices where the first, second, and third requirements fail, respectively.
	In view of \cref{w.bd} and \cref{nabxi.l1}, we have
	\begin{align*}
		&\int_{B_1^{n-2}}\int_{D_1^2}\left(\frac{W(u_\epsilon)}{\epsilon^2}+|\xi_\epsilon|\right)(y,z)\,dz\,dy
		=\int_Q\left(\frac{W(u_\epsilon)}{\epsilon^2}+|\xi_\epsilon|\right)
		\le C,
	\end{align*}
	for some constant $C$ independent of $\epsilon$. The Hardy--Littlewood weak $(1,1)$ estimate for the maximal function of
	\begin{align*}
		&y\mapsto\int_{D_1^2}\left(\frac{W(u_\epsilon)}{\epsilon^2}+|\xi_\epsilon|\right)(y,z)\,dz
	\end{align*}
	then implies that
	\begin{align*}
		&|F_\epsilon^3|\le\frac{C}{K}.
	\end{align*}
	Similarly, from \cref{par.van}, \cref{en.is.coex}, \cref{jv.xi}, and the maximal inequality, it follows that
	\begin{align*}
		&|F_\epsilon^2|
		\le \frac{C}{\delta\leps }\int_Q\left(|du_{\epsilon}(P)|^2+|jv_\epsilon-d^*\xi_\epsilon|^2+\left|e_\epsilon(u_\epsilon)-\mz|d^*\xi_\epsilon|^2\right|\right)
		\to 0.
	\end{align*}
	
	In order to bound the measure of $F_\epsilon^{1}$, we observe that $F_\epsilon^1\subseteq F_\epsilon^{1a}\cup F_\epsilon^{1b}$, where we denote by $F_{\epsilon}^{1a}$ and $F_{\epsilon}^{1b}$ the sets of slices $y\in B_{1/2}^{n-2}$ such that
	\begin{align*}
		&\sup_{0<r<1/2}r^{2-n}\int_{B_r^{n-2}(y)\times[D_1^2\setminus D_{1/2}^2]}e_\epsilon(u_\epsilon)\ge\frac{\delta}{2}\leps 
	\end{align*}
	and
	\begin{align*}
		&\sup_{0<r<1/2}\left|r^{2-n}\int_{B_r^{n-2}(y)\times D_1^2}\chi e_\epsilon(u_\epsilon)-\pi\omega_{n-2}\theta\leps \right|\ge\frac{\delta}{2}\leps ,
	\end{align*}
	respectively, where $\chi=\chi(z)$ is a cutoff function supported in $D_1^2$, with $\chi\equiv 1$ on $D_{1/2}^2$ and $0\le\chi\le 1$.
	
	Since the energy concentrates along $P$, we have
	$$\frac{1}{\leps }\int_{B_1^{n-2}\times[D_1^2\setminus D_{1/2}^2]}e_\epsilon(u_\epsilon)\to 0,$$
	which implies that $|F_\epsilon^{1a}|\to 0$, again by the Hardy--Littlewood maximal inequality.
	Finally, for $y\in B_{5/4}^{n-2}(0)$, let
	\begin{align*}
		&f_\epsilon(y):=\int_{\{y\}\times D_1^2}\chi e_\epsilon(u_\epsilon)\,dz.
	\end{align*}
	Recall that the stress-energy tensor $T_\epsilon=e_\epsilon(u_\epsilon)I-du_\epsilon^*du_\epsilon$ has zero divergence.
	Hence, testing against the vector field $\psi(y)\chi(z)e_k$ gives
	\begin{align*}
		&\int_{B_2^n}e_\epsilon(u_\epsilon)\de_k(\psi\chi)
		=\int_{B_2^n}{\textstyle\sum_{j=1}^n}\de_j u_\epsilon\de_k u_\epsilon\de_j(\psi\chi),
	\end{align*}
	for $\psi\in C^1_c(B_{5/4}^{n-2})$. In particular, for $k=1,\dots,n-2$, \cref{par.van} and Cauchy--Schwarz imply that
	\begin{align*}
		&\left|\int_{B_{5/4}^{n-2}}f_\epsilon\de_k\psi\right|
		\le C\|du_\epsilon\|_{L^2(B_{7/4})}\|\de_k u_\epsilon\|_{L^2(B_{7/4})}\|\psi\|_{C^1}
		\le o(\leps )\|d\psi\|_{C^0}
	\end{align*}
	(as $B_{5/4}^{n-2}\times D_1^2\subseteq B_{7/4}^n$ and $\|\psi\|_{C^1}\le C\|d\psi\|_{C^0}$). In particular, applying the Hahn--Banach extension theorem to the functionals
	$$\nabla\psi\mapsto \frac{1}{\leps}\int_{B_{5/4}^{n-2}}f_{\epsilon}\partial_k\psi$$
	on the subspace
	$$\{\nabla \psi\mid \psi\in C_c^1(B_{5/4}^{n-2})\}\subset C_0(B_{5/4}^{n-2},\mathbb{R}^{n-2}),$$
	where $C_0(B_{5/4}^{n-2},\mathbb{R}^{n-2})=\overline{C_c^{\infty}(B_{5/4}^{n-2},\mathbb{R}^{n-2})}^{C^0}$, it follows that there exist vector-valued Radon measures $X_k^{\epsilon}\in C_0(B_{5/4}^{n-2},\mathbb{R}^{n-2})^*$ such that
	$$\lim_{\epsilon\to 0}\| X_k^{\epsilon}\|=0$$
	and
	$$\langle X_k^{\epsilon},\nabla \psi\rangle=\frac{1}{\leps}\int f_{\epsilon}\partial_k\psi,$$
	so that $\operatorname{div}(X_k^{\epsilon})=\partial_k\left(\frac{f_{\epsilon}}{\leps}\right)$ distributionally.
	
	We can then apply Allard's strong constancy lemma \cite[Theorem~1.(4)]{Allard.constancy} and deduce that
	\begin{align*}
		&\frac{1}{\leps }\|f_\epsilon-\alpha_\epsilon\|_{L^1(B_1^{n-2})}\to 0,
	\end{align*}
	for suitable constants $\alpha_\epsilon$. In fact, since $\frac{1}{\leps }\int_{B_1^{n-2}}f_\epsilon\to\pi\omega_{n-2}\theta$, the same conclusion must hold with $\alpha_\epsilon=\pi\theta\leps $. This implies that $|F_\epsilon^{1b}|\to 0$, which completes the proof.
\end{proof}

Next, we record the following lemma about the small-scale behavior of the two-form $djv_{\epsilon}$ near a good slice, which we will use repeatedly in subsequent sections to refine our characterization of the limiting energy measure.

\begin{lemma}\label{lim.jac.lem}
	For any $\alpha\in (0,1)$ and $\gamma>0$, there exists $\delta_1(\alpha,\gamma)\in(0,1)$ such that if $B_{2r}(x)\subset Q$ is a ball with $r\ge\epsilon^{\alpha}$ for which
	\begin{equation}\label{lim.jac.lem.h1}
		r^{2-n}\int_{B_{2r}(x)}|du_{\epsilon}(P)|^2\le\delta_1\leps ,
	\end{equation}
	then for $\epsilon$ small enough (depending on $\alpha$ and $\gamma$)
	\begin{equation}\label{par.small}
		\left|r^{2-n}\int_{B_r(x)}(djv_{\epsilon}(x'))_{ab}\frac{x-x'}{|x-x'|}\,dx'\right|<\gamma\quad\text{for }(a,b)\neq (n-1,n).
	\end{equation}
	Moreover,
	if $\mathcal{V}(u_{\epsilon})\cap B_{2r}(x)\cap P^{\perp}_x\subseteq B_{\delta_1r}(x)$, with $\kappa:=\deg(v_{\epsilon},x+\{0\}\times rS^1)$ we have
	\begin{equation}\label{vort.char}
		\left|r^{2-n}\int_{B_r(x)}(djv_{\epsilon}(x'))_{n-1,n} -2\pi \kappa\omega_{n-2}\right|<\gamma.
	\end{equation}
\end{lemma}

\begin{proof}
	We first prove \cref{par.small}, via a contradiction argument. If the statement fails, then (passing to a subsequence) there exist balls $B_{2r_{\epsilon}}(x_{\epsilon})\subset Q$ with $r_{\epsilon}\ge\epsilon^{\alpha}$ such that
	$$\lim_{\epsilon\to 0}r_{\epsilon}^{2-n}\frac{1}{\leps }\int_{B_{2r_{\epsilon}}(x_\epsilon)}|du_{\epsilon}(P)|^2=0$$
	but 
	\begin{equation}\label{non.van}
		\left|r_{\epsilon}^{2-n}\int_{B_{r_{\epsilon}}(x_\epsilon)}(djv_{\epsilon}(x'))_{ab}\frac{x_\epsilon-x'}{|x_\epsilon-x'|}\,dx'\right|
		\geq \gamma
	\end{equation}
	for some $(a,b)\neq (n-1,n)$. Rescaling $B_{2r_{\epsilon}}(x_{\epsilon})$ to $B_2(0)$, we obtain a sequence of solutions $\tilde{u}_{\tilde{\epsilon}}$ of the Ginzburg--Landau equation on $B_2(0)$, with $\tilde\epsilon={\epsilon}/{r_\epsilon}\le\epsilon^{1-\alpha}$ and
	$$E_{\tilde{\epsilon}}(\tilde{u}_{\tilde{\epsilon}})\le C\leps \le \frac{C}{1-\alpha}\lteps ,$$
	as well as
	$$\frac{1}{\lteps }\int_{B_2(0)}|d\tilde{u}_{\tilde{\epsilon}}(P)|^2\to 0$$
	as $\tilde{\epsilon}\to 0$.
	
	By \cref{asymp}, the limit of the normalized energy densities of $\tilde{u}_{\tilde\epsilon}$ has a concentrated part $|\tilde V|$, for a rectifiable stationary varifold $\tilde V$,
	and by reversing the proof of \cref{par.van.lem}, we know that the tangent plane to $\tilde V$ at $x$ coincides with $P$, for $|\tilde V|$-a.e.\ $x$. Together with stationarity, this easily implies (testing stationarity against vector fields of the form $X=\phi v$ for fixed vectors $v\in P$) that $\tilde V$ is invariant under translations by vectors in $P$, and therefore coincides on $B_2(0)$ with a locally finite union of planes $P+x_j$ parallel to $P$ (with multiplicity at least $c(n)>0$).
	
	Moreover, as in \cref{jv.w},
	\begin{equation}\label{jac.ptwise}
		|dj\tilde{v}_{\tilde\epsilon}|\leq C\frac{W(\tilde{u}_{\tilde\epsilon})}{\tilde{\epsilon}^2},
	\end{equation}
	and by \cref{w.nu} we deduce that
	\begin{equation}\label{meas.cvg}
		|dj\tilde{v}_{\tilde{\epsilon}}|\rightharpoonup {\textstyle\sum_j} f_j \mathcal{H}^{n-2}\mrestr(P+x_j)
	\end{equation}
	weakly in $C^0_c(B_2)^*$, with $f_j$ locally bounded.
	
	Hence, the rescaled $(n-2)$-cycles $*dj\tilde{v}_{\tilde{\epsilon}}$ converge weakly as currents in $B_2$ to a cycle supported on $\bigcup_j(P+x_j)$.
	By the constancy theorem for cycles, it follows that
	\begin{equation}\label{curr.cvg}
		*(dj\tilde{v}_{\tilde{\epsilon}})\rightharpoonup {\textstyle\sum_j} 2\pi\kappa_j(P+x_j)\quad\text{on }B_2,
	\end{equation}
	for suitable constants $\kappa_j\in\R$ (actually $\kappa_j\in\Z$, by \cite[Theorem~5.2]{JSoner}, or by a slicing argument similar to the proof of \cref{kappa.at.most.1} below,
	which reveals that $\kappa_j$ is the degree of $\tilde v_{\tilde\epsilon}$ around $P+x_j$).
	
	If \cref{non.van} holds, then rescaling gives
	\begin{equation}\label{non.van.rescale}
		\left|\int_{B_1}(dj\tilde{v}_{\tilde{\epsilon}}(x'))_{ab}\frac{x'}{|x'|}\,dx'\right|
		\geq \gamma,
	\end{equation}
	but fixing $\chi\in C_c^{\infty}(B_1)$, $0\le\chi\le 1$ be some test function supported away from $0$ such that 
	$${\textstyle\sum_j}\int_{(P+x_j)\cap B_1}(1-\chi)f_j\,d\mathcal{H}^{n-2}<\frac{\gamma}{2},$$
	it follows from the distributional convergence \cref{curr.cvg} that, for this couple $(a,b)\neq(n-1,n)$, $(dj\tilde{v}_{\tilde{\epsilon}})_{ab}$ vanishes distributionally as $\tilde{\epsilon}\to 0$, so
	$$\lim_{\tilde{\epsilon} \to 0}\int_{B_1}(dj\tilde{v}_{\tilde{\epsilon}}(x'))_{ab}\chi(x')\frac{x'}{|x'|}\,dx'=0.$$
	On the other hand, it then follows from \cref{meas.cvg} that
	\begin{align*}
		\limsup_{\tilde{\epsilon} \to 0}\left|\int_{B_1}(dj\tilde{v}_{\tilde{\epsilon}}(x'))_{ab}\frac{x'}{|x'|}\,dx'\right|
		&=\limsup_{\tilde{\epsilon} \to 0}\left|\int_{B_1}(dj\tilde{v}_{\tilde{\epsilon}}(x'))_{ab}(1-\chi(x'))\frac{x'}{|x'|}\,dx'\right|\\
		&\leq {\textstyle\sum_j}\int_{(P+x_j)\cap B_1}(1-\chi)f_j\,d\mathcal{H}^{n-2}\\
		&<\gamma/2,
	\end{align*}
	contradicting \cref{non.van.rescale}.
	
	The proof of \cref{vort.char} is similar, where in the limiting rescaled picture we have simply ${\textstyle\sum_j} 2\pi \kappa_j(P+x_j)=2\pi \kappa P$.
\end{proof}

In the following section, we will use this together with the following formulas for $\xi_\epsilon$ and $\nabla\xi_\epsilon$, which follow simply from \cref{xi.def} and the formula for the Euclidean Green's function $G(x,y)$ in $\mathbb{R}^n$, together with a simple integration by parts
(recall that $n\omega_{n}=2\pi\omega_{n-2}$). 

\begin{lemma}\label{xi.comp.lem} For any pair $(a,b)$ with $1\leq a<b\leq n$, we have
	\begin{equation}\label{xicomp.1}
		2\pi\omega_{n-2}(\xi_\epsilon)_{ab}(x)=\int_0^{\infty}\frac{1}{r}\left(r^{2-n}\int_{B_r(x)}\varphi(x')[djv_{\epsilon}(x')]_{ab}\,dx'\right) dr,
	\end{equation}
	\begin{equation}\label{nabla.xicomp}
		\nabla(\xi_\epsilon)_{ab}(x)=\frac{n-1}{2\pi\omega_{n-2}}\int_0^{\infty}\frac{1}{r^2}\left(r^{2-n}\int_{B_r(x)}\varphi(x')[djv_{\epsilon}(x')]_{ab}\frac{x-x'}{|x-x'|}\,dx'\right) dr.
	\end{equation}
\end{lemma}
\begin{remark} For non-Euclidean metrics $g_{\epsilon}$ converging smoothly to the Euclidean metric, \cref{green} shows that \eqref{xicomp.1} and \eqref{nabla.xicomp} hold up to errors of size
	$$o(1) \cdot \int_0^{\infty}r^{2-n}\left(\int_{B_r(x)}|\varphi(x')djv_{\epsilon}|\,dx'\right)dr,\quad o(1)\cdot\int_0^{\infty}r^{1-n}\left(\int_{B_r(x)}|\varphi(x')djv_{\epsilon}|\,dx'\right)dr,$$
	both of which can be seen to be $o(\leps)$ (as in the proof of \cref{xi.ptwise.1} below).
\end{remark}

As in \cite{BBO}, note that combining \cref{xicomp.1} with \cref{jv.w} gives, for any $x\in B_{3/2}(0)$,
$$|\xi_\epsilon(x)|\leq C\int_0^{1/4}\frac{1}{r^{n-1}}\left(\int_{B_r(x)}\frac{W(u_{\epsilon})}{\epsilon^2}\right)dr+C\int_{B_{7/4}(0)}\frac{W(u_{\epsilon})}{\epsilon^2},$$
which together with the monotonicity formula \cref{monotonicity} (integrated over $s\in (0,\frac{1}{4})$) gives as in \cite{BBO} the pointwise estimate
\begin{equation}\label{xi.ptwise.1}
	|\xi_\epsilon(x)|\leq CE_{\epsilon}(u_{\epsilon};B_{1/4}(x))+C\int_{B_{7/4}(0)}\frac{W(u_{\epsilon})}{\epsilon^2}
	\leq CE_{\epsilon}(u_{\epsilon}; B_{7/4}(0)).
\end{equation}
Interpolating with \cref{nabxi.l1}, we obtain
\begin{align*}
	&\|\xi_\epsilon\|_{L^{2n}(B_{3/2})}
	\le \|\xi_\epsilon\|_{L^{1}(B_{3/2})}^{1/2n}\|\xi_\epsilon\|_{L^{\infty}(B_{3/2})}^{1-1/2n}
	\le C\leps ^{1-1/2n},
\end{align*}
and using again \cref{nabxi.l1} we get
\begin{align}\label{nabxi.weird}
	&\int_{B_{3/2}}|\xi_\epsilon||d^*\xi_\epsilon|
	\le\|\xi_\epsilon\|_{L^{2n}(B_{3/2})}\|d^*\xi_\epsilon\|_{L^{2n/(2n-1)}(B_{3/2})}
	\le C\leps ^{1-1/2n}.
\end{align}

For a given cutoff function $\chi\in C_c^{\infty}(B_{3/2}(0))$ with $\chi\equiv 1$ on $B_{1}(0)$, a simple integration by parts yields
\begin{align*}
	\int_{B_{3/2}} \chi^2 |d^*\xi_{\epsilon}|^2
	&=\int_{B_{3/2}} \langle \xi_{\epsilon},d(\chi^2\,d^*\xi_{\epsilon})\rangle\\
	&=\int_{B_{3/2}} \langle\xi_{\epsilon},2\chi\,d\chi\wedge d^*\xi_{\epsilon}\rangle+\int_{B_{3/2}} \langle \chi^2\,\xi_{\epsilon},dd^*\xi_{\epsilon}\rangle\\
	&=\int_{B_{3/2}} \langle\xi_{\epsilon},2\chi\,d\chi\wedge d^*\xi_{\epsilon}\rangle-\int_{B_{3/2}}\langle \chi^2\,\xi_{\epsilon},dh_{\epsilon}\rangle
	+\int_{B_{3/2}} \langle \chi^2\,\xi_{\epsilon},djv_{\epsilon}\rangle.
\end{align*}
It follows from \cref{nabxi.l1}, \cref{h.bd}, and \cref{nabxi.weird} that the first two terms on the last line are $o(\leps )$ as $\epsilon\to 0$, so that
\begin{align*}
	\lim_{\epsilon\to 0}\frac{1}{\leps }\int_{B_{3/2}} \chi^2|d^*\xi_{\epsilon}|^2
	&=\lim_{\epsilon\to 0}\frac{1}{\leps }\int_{B_{3/2}}\chi^2\langle \xi_{\epsilon},djv_{\epsilon}\rangle\\
	\text{(using \cref{xi.ptwise.1})}\ &\leq \liminf_{\epsilon\to 0}\frac{CE_{\epsilon}(u_{\epsilon};B_{7/4}(0))}{\leps }\int_{B_{3/2}}|djv_{\epsilon}|\\
	\text{(using \cref{jv.w})}\ &\leq \liminf_{\epsilon\to 0}\frac{CE_{\epsilon}(u_{\epsilon};B_{7/4}(0))}{\leps }\int_{B_{3/2}}\frac{W(u_{\epsilon})}{\epsilon^2}.
\end{align*}
In particular, using \cref{coex.dom} and recalling that $\chi\equiv 1$ on $B_{1}$, we deduce that
$$\lim_{\epsilon\to 0}\frac{E_{\epsilon}(u_{\epsilon};B_{1})}{\leps }
\leq \liminf_{\epsilon\to 0}\frac{CE_{\epsilon}(u_{\epsilon};B_{2})}{\leps }\int_{B_2}\frac{W(u_{\epsilon})}{\epsilon^2}.$$
This computation holds for any sequence of solutions $u_\epsilon:B_2\to\C$ obeying an energy bound $E_\epsilon(u_\epsilon;B_2(0))\le\Lambda\leps $.

Combining this observation with \cref{clearing} and a trivial rescaling, we obtain the following lemma, which will be useful later.
\begin{lemma}\label{w.quant.lem} There exists $c(\Lambda,n)>0$
	such that if $u_{\epsilon}$ solves the Ginzburg--Landau equation on $B_r(x)$, with $r\ge\epsilon$, $B_{r/4}(x)\cap\mathcal{V}(u_{\epsilon})\neq\emptyset$, and
	$$r^{2-n}E_{\epsilon}(u_{\epsilon};B_r(x))\leq \Lambda\log(r/\epsilon),$$
	then
	\begin{equation}
		r^{2-n}\int_{B_r(x)}\frac{W(u_{\epsilon})}{\epsilon^2}>c(\Lambda,n).
	\end{equation}
	Moreover, the simpler estimate
	\begin{equation}
		\epsilon^{2-n}\int_{B_{\epsilon}(x)}\frac{W(u_{\epsilon})}{\epsilon^2}>c(n)
	\end{equation}
	holds for $x\in \mathcal{V}(u_{\epsilon})$ without any additional assumptions.
\end{lemma}

Note that the second conclusion (as well as the first one, when $r$ is comparable to $\epsilon$) follows directly from the bound $|du_\epsilon|\le\frac{C(n)}{\epsilon}$,
which implies that $W(u_\epsilon)\ge\frac{1}{10}$ in a $c(n)\epsilon$-neighborhood of any point in $B_{r/4}(x)\cap\mathcal{V}(u_\epsilon)$.

\section{Preliminary energy identity and related bounds}
In this section we establish two of the key ingredients for the proof of \cref{dens.gap}, proving the preliminary energy identity \eqref{intro.theta.comp} and the new potential bound  \eqref{w.lbd.intro} for arbitrary families satisfying the hypotheses of \cref{dens.gap}. 

Let $u_{\epsilon}$ be a family of solutions satisfying the assumptions of \cref{dens.gap}. Appealing to Lemma \ref{many.good.slices}, choose a family $\delta_{\epsilon}\to 0$ such that
\begin{equation}\label{good.full}
	\liminf_{\epsilon\to 0}|\mathcal{G}_{\epsilon,\delta_{\epsilon}}|\ge\frac{\omega_{n-2}}{2^{n-2}}-\frac{C}{K}>0\quad\text{as }\epsilon\to 0,
\end{equation}
and fix $y_{\epsilon}\in \mathcal{G}_{\epsilon,\delta_{\epsilon}}$.

Next, given $\delta>0$, consider the set $U_{\delta,\epsilon}\subseteq P_{y_{\epsilon}}^{\perp}$ given by
$$U_{\delta,\epsilon}:=\{x=(y_{\epsilon},z)\in P_{y_{\epsilon}}^{\perp} : |u_{\epsilon}(x)|<1-\delta\}.$$

\begin{proposition}\label{cover.claim} There exists $C=C(\delta)$ independent of $\epsilon>0$ and points $p_1^{\epsilon},\ldots, p_k^{\epsilon}\in U_{\delta,\epsilon}$ with $k\leq C$ such that
	$$U_{\delta,\epsilon}\subseteq \bigsqcup_{j=1}^kB_{C\epsilon}(p_j^{\epsilon}),$$
	up to a subsequence, for a \emph{disjoint} family of balls $B_{C\epsilon}(p_j^{\epsilon})$.
\end{proposition}

\begin{proof}
	By a simple Vitali covering argument applied to the covering $\{B_{\epsilon}(p)\mid p\in U_{\delta,\epsilon}\}$, it is clear that there exist $p_1^{\epsilon},\ldots,p_m^{\epsilon}\in U_{\delta,\epsilon}$ such that 
	$$B_{\epsilon}(p_i^{\epsilon})\cap B_{\epsilon}(p_j^{\epsilon})=\varnothing\quad\text{for }i\neq j$$
	and
	$$U_{\delta,\epsilon}\subseteq \bigcup_{j=1}^mB_{5\epsilon}(p_j^{\epsilon}),$$
	where $m=m_{\epsilon}$ depends on $\epsilon$ a priori. On the other hand, since $|u_{\epsilon}(p_j^{\epsilon})|<1-\delta$ and $|du_{\epsilon}|\leq \frac{C}{\epsilon}$, writing $p_j^\epsilon=(y_\epsilon,z_j^\epsilon)$ it is clear that
	$$\int_{\{y_\epsilon\}\times D_\epsilon(z_j^\epsilon)}\frac{W(u_{\epsilon})}{\epsilon^2}\geq c(\delta)$$
	for some $c(\delta)>0$, and summing over $1\leq j\leq m_{\epsilon}$ and using \cref{better.slice} gives
	$$m_{\epsilon}c(\delta)\leq\int_{P_{y_{\epsilon}}^\perp}\frac{W(u_{\epsilon})}{\epsilon^2}\leq 2\omega_n,$$
	hence
	$$m_{\epsilon}\leq C(\delta).$$
	In particular, passing to a subsequence, we may assume that $m_{\epsilon}=m$ is fixed independent of $\epsilon$, and that the (possibly infinite) limits
	\begin{equation}\label{dist.lims}
		\gamma_{ij}:=\lim_{\epsilon\to 0}\frac{|p_i^{\epsilon}-p_j^{\epsilon}|}{\epsilon}
	\end{equation}
	exist. It is then easy to see that the desired conclusion holds with
	\begin{align*}
		&C(\delta):=10+\max_{\ell\in F}\ell,\quad\text{where }F:=\{\gamma_{ij}\mid \gamma_{ij}<\infty\}.
	\end{align*}
	Indeed, we can form an equivalence relation on $\{1,\dots,m\}$ by declaring that $i\sim j$ precisely when $\gamma_{ij}<\infty$,
	and we can take a set of representatives $S\subseteq\{1,\dots,m\}$; with the previous choice of $C(\delta)$, we have
	$$ U_{\delta,\epsilon}\subseteq\bigcup_{j=1}^m B_{5\epsilon}(p_j^\epsilon)\subseteq\bigcup_{i\in S}B_{C(\delta)\epsilon}(p_i^\epsilon), $$
	since if $i$ represents the class of $j$ then $|p_i^\epsilon-p_j^\epsilon|\le(\gamma_{ij}+1)\epsilon$ for $\epsilon$ small enough,
	and the last union is disjoint since for $i,i'\in S$ we have $\frac{|p_i^\epsilon-p_{i'}^\epsilon|}{\epsilon}\to\gamma_{ii'}=\infty$ (unless $i=i'$).
\end{proof}

Now, for $k\le C(\delta)$ and points $p_1^{\epsilon},\dots,p_k^{\epsilon}$ as in \cref{cover.claim}, denote by $D_{j,\epsilon}$ the disks
$$D_{j,\epsilon}:=B_{C\epsilon}(p_j^{\epsilon})\cap P_{y_{\epsilon}}^\perp$$
(note that eventually $D_{j,\epsilon}$ is compactly included in $P_{y_\epsilon}^\perp$, as $|u_\epsilon|\to 1$ on $B_1^{n-2}\times\de D_1$),
and consider the degrees
$$\kappa_j^{\epsilon}:=\deg\left(\frac{u_{\epsilon}}{|u_{\epsilon}|},\partial D_{j,\epsilon}\right).$$

The following proposition gives a very useful (though non-sharp) bound on the degrees $\kappa_j^{\epsilon}$ in terms of the potential $\frac{W(u_{\epsilon})}{\epsilon^2}$, which plays a crucial role in ruling out energy-cancellation due to interactions between vortices with degrees of opposite signs in the proof of \cref{dens.gap}

\begin{proposition}\label{w.deg.prop} For $\epsilon$ sufficiently small (depending on $\delta$),
	\begin{equation}
		\int_{D_{j,\epsilon}}\frac{2W(u_{\epsilon})}{\epsilon^2}\geq \frac{\pi}{2}|\kappa_j^{\epsilon}|(1-5\delta).
	\end{equation}
\end{proposition}

\begin{proof}
	By \cref{cover.claim}, we know that $\partial D_{j,\epsilon}\subseteq\{|u_{\epsilon}|\geq 1-\delta\}$. For each $0<t<1-\delta$, consider the set
	$$\Omega_t:=\{|u_{\epsilon}|<t\}\cap D_{j,\epsilon},$$
	and if $t$ is a regular value for $|u_{\epsilon}|$, consider also the boundary
	$$S_t:=\partial\Omega_t\subset D_{j,\epsilon}.$$
	By the coarea formula for $|u_{\epsilon}|$ on $D_{j,\epsilon}$, we have
	\begin{align*}
		\int_{\Omega_{1-\delta}}\frac{W(u_{\epsilon})}{\epsilon^2}
		&=\int_0^{1-\delta}\frac{(1-r^2)^2}{4\epsilon^2}\left(\int_{S_r}\frac{1}{|d|u_{\epsilon}||}\right) dr.
	\end{align*}
	Next, note that for each regular value $t\in (0,1-\delta)$ of $|u_{\epsilon}|$, a few simple applications of the Cauchy--Schwarz inequality give
	\begin{align*}
		|S_t|&=\int_{S_t}|d|u_{\epsilon}||^{-1/2}|d|u_{\epsilon}||^{1/2}\\
		&\leq \left(\int_{S_t}\frac{1}{|d|u_{\epsilon}||}\right)^{1/2}\left(\int_{S_t}|d|u_{\epsilon}||\right)^{1/2}\\
		&\leq \left(\int_{S_t}\frac{1}{|d|u_{\epsilon}||}\right)^{1/2}\left(\int_{S_t}|d|u_{\epsilon}||^2\right)^{1/4}|S_t|^{1/4},
	\end{align*}
	which we can write equivalently as
	$$\int_{S_t}\frac{1}{|d|u_{\epsilon}||}\geq |S_t|^{3/2}\left(\int_{S_t}|d|u_{\epsilon}||^2\right)^{-1/2},$$
	and applying this in the preceding computation yields
	\begin{equation}\label{w.low.1}
		\int_{\Omega_{1-\delta}}\frac{W(u_{\epsilon})}{\epsilon^2}\geq \int_0^{1-\delta}\frac{(1-r^2)^2}{4\epsilon^2}|S_r|^{3/2}\left(\int_{S_r}|d|u_{\epsilon}||^2\right)^{-1/2}dr.
	\end{equation}
	
	Now, since
	$$|du_{\epsilon}|^2\leq \frac{1-|u_{\epsilon}|^2}{\epsilon^2}+C$$
	on $D_{j,\epsilon}$, we have for any regular value $t\in (0,1-\delta)$ of $|u_{\epsilon}|$ that
	\begin{equation}
		\int_{S_t}|du_{\epsilon}|^2\le\left(\frac{1-t^2}{\epsilon^2}+C\right)|S_t|.
	\end{equation}
	In particular, writing
	$$|du_{\epsilon}|^2=|u_{\epsilon}|^2|d(u_{\epsilon}/|u_{\epsilon}|)|^2+|d|u_{\epsilon}||^2,$$
	it follows that
	$$t^2\int_{S_t}|d(u_{\epsilon}/|u_{\epsilon}|)|^2+\int_{S_t}|d|u_{\epsilon}||^2
	\leq \left(\frac{1-t^2}{\epsilon^2}+C\right)|S_t|,$$
	and an application of Young's inequality on the left-hand side gives
	$$2t\left(\int_{S_t}|d(u_{\epsilon}/|u_{\epsilon}|)|^2\right)^{1/2}\left(\int_{S_t}|d|u_{\epsilon}||^2\right)^{1/2}
	\leq \left(\frac{1-t^2}{\epsilon^2}+C\right)|S_t|.$$
	Moreover, since
	$$2\pi \kappa_j^{\epsilon}=\int_{S_t}(u_{\epsilon}/|u_{\epsilon}|)^*(d\theta)$$
	for each regular value $t\in (0,1-\delta)$, we have that
	\begin{align*}
		2\pi |\kappa_j^{\epsilon}|&\leq \int_{S_t}|d(u_{\epsilon}/|u_{\epsilon}|)|
		\leq |S_t|^{1/2}\left(\int_{S_t}|d(u_{\epsilon}/|u_{\epsilon}|)|^2\right)^{1/2},
	\end{align*}
	which we can combine with the preceding computation to see that
	$$4\pi |\kappa_j^{\epsilon}|t\cdot \left(\int_{S_t}|d|u_{\epsilon}||^2\right)^{1/2}\leq \left(\frac{1-t^2}{\epsilon^2}+C\right)|S_t|^{3/2}.$$
	Rearranging, we see that
	\begin{equation}\label{deg.ubd}
		|S_t|^{3/2}\left(\int_{S_t}|d|u_{\epsilon}||^2\right)^{-1/2}\geq \left(\frac{1-t^2}{\epsilon^2}+C\right)^{-1}\cdot 4\pi |\kappa_j^{\epsilon}|t.
	\end{equation}
	
	Applying \cref{deg.ubd} to the integrand on the right-hand side of \cref{w.low.1}, we deduce that
	\begin{align*}
		\int_{\Omega_{1-\delta}}\frac{W(u_{\epsilon})}{\epsilon^2}
		&\geq \int_0^{1-\delta}\frac{(1-r^2)^2}{4\epsilon^2}\cdot\left(\frac{1-r^2}{\epsilon^2}+C\right)^{-1}\cdot 4\pi |\kappa_j^{\epsilon}| r\,dr\\
		&\geq \pi |\kappa_j^{\epsilon}|\int_0^{1-\delta}(1-r^2-C\epsilon^2)r\,dr\\
		&\geq \pi |\kappa_j^{\epsilon}|\cdot\frac{1}{4}(1-4\delta-C\epsilon^2),
	\end{align*}
	and choosing $\epsilon$ small enough so that $C\epsilon^2\leq \delta,$ we deduce that
	$$\int_{D_{j,\epsilon}}\frac{2W(u_{\epsilon})}{\epsilon^2}\geq \frac{\pi}{2}|\kappa_j^{\epsilon}|(1-5\delta),$$
	as claimed.
\end{proof}

In particular, summing the estimate from \cref{w.deg.prop} over $j=1,\ldots, k$, we deduce that
\begin{align*}
	\frac{\pi}{2}(1-5\delta){\textstyle\sum_{j=1}^k}|\kappa_j^{\epsilon}|
	&\leq {\textstyle\sum_{j=1}^k}\int_{D_{j,\epsilon}}\frac{2W(u_{\epsilon})}{\epsilon^2}
	\leq \int_{P_{y_{\epsilon}}^{\perp}}\frac{2W(u_{\epsilon})}{\epsilon^2}.
\end{align*}
Later, in the proof of \cref{dens.gap}, we will use this together with sharp upper bounds on $\int_{P_{y_{\epsilon}}^{\perp}}\frac{2W(u_{\epsilon})}{\epsilon^2}$ to show that there can be only one zero of nonzero degree in a good slice when the density $\theta<2$.

Now, since $y_{\epsilon}\in \mathcal{G}_{\epsilon,\delta_{\epsilon}}$, we know already that
\begin{equation}\label{theta.xi}
	\pi\theta=\lim_{\epsilon\to 0}\frac{1}{\leps }\int_{P_{y_{\epsilon}}^{\perp}}\frac{1}{2}|d^*\xi_{\epsilon}|^2.
\end{equation}
On our way to proving \eqref{intro.theta.comp}, we show next that the only terms in $d^*\xi_{\epsilon}$ which contribute nontrivially to the limit are those of the form $\partial_a(\xi_\epsilon)_{ab}$, where $\{a,b\}=\{n-1,n\}$.

\begin{lemma}\label{slice.perp.small} For $y_{\epsilon}\in \mathcal{G}_{\epsilon,\delta_{\epsilon}}$ as above, and for any pair $(a,b)\neq(n-1,n)$, we have
	$$\lim_{\epsilon\to 0}\frac{1}{\leps }\int_{P_{y_{\epsilon}}^{\perp}}|\nabla (\xi_{\epsilon})_{ab}|^2=0.$$
\end{lemma}

\begin{proof} To begin, fix $\alpha\in (0,1)$ close to $1$, and consider the distance function
	$$\rho_\epsilon(x):=\dist(x,\mathcal{V}(u_{\epsilon})).$$
	For each $\epsilon^{\alpha}\le r\le\frac{1}{4}$, consider the set $F_r\subset P_{y_{\epsilon}}^{\perp}$ given by
	$$F_r:=\{x\in \{y_\epsilon\}\times D_{1/2} : \rho_\epsilon(x)\leq r\}.$$
	We claim first that, for $\epsilon$ sufficiently small (depending on $\alpha$) and $r\geq \epsilon^{\alpha}$, we have
	\begin{equation}\label{fr.ubd}
		\mathcal{H}^2(F_r)\leq C(\alpha)r^2
	\end{equation}
	for a constant $C(\alpha)$ independent of $\epsilon$. To see this, note that for each $x\in F_r$, there exists $x'\in \bar B_r(x)\cap \mathcal{V}(u_{\epsilon})$.
	In particular, by \cref{clearing}, it follows that
	\begin{align*}
		&r^{2-n}E_\epsilon(u_\epsilon;B_{2r}(x))\ge c(n)\log(2r/\epsilon)
		\ge c(n)(1-\alpha)\leps .
	\end{align*}
	Now, Vitali's covering lemma gives $x_1,\ldots, x_N\in F_r$ such that the balls $B_{2r}(x_1),\ldots, B_{2r}(x_N)$ are disjoint and $F_r\subseteq \bigcup_{j=1}^NB_{10r}(x_j)$. From the disjointness of the balls $B_{2r}(x_j)$ we deduce that
	\begin{align*}
		N\cdot c(n)(1-\alpha)\leps 
		\le{\textstyle\sum_{j=1}^N} r^{2-n}E_\epsilon(u_\epsilon;B_{2r}(x_j))
		\le r^{2-n}E_\epsilon(u_\epsilon;B_{2r}^{n-2}(y_\epsilon)\times D_1).
	\end{align*}
	In particular, since $y_{\epsilon}\in \mathcal{G}_{\epsilon,\delta_{\epsilon}}$, the right-hand side is bounded by $C\leps$, with $C$ independent of $r$ and $\epsilon$, and therefore $N\leq C(\alpha)$ for $\epsilon$ small enough. Since $F_r$ is covered by the $N$ balls $B_{10r}(x_j)$ of radius $10r$, the bound \cref{fr.ubd} follows.
	
	Next, let $(a,b)$ be a pair of indices with $a<n-1$, and fix an arbitrary small $\gamma>0$. Since $y_{\epsilon}\in \mathcal{G}_{\epsilon,\delta_{\epsilon}}$, for $\epsilon$ sufficiently small, the hypotheses of (the first part of) \cref{lim.jac.lem} hold for every ball $B_r(x)$ with $x\in \{y_\epsilon\}\times D_{1/2}$ and $\epsilon^{\alpha}\le r\le\frac{1}{4}$, so that
	\begin{equation}\label{lim.jac.cons}
		\left|r^{2-n}\int_{B_r(x)}(djv_{\epsilon}(x'))_{ab}\frac{x-x'}{|x-x'|}\,dx'\right|<\gamma.
	\end{equation}
	As a consequence, if $x \nin F_{\epsilon^{\alpha}}$, it follows from \cref{xi.comp.lem} that
	\begin{align*}
		|\nabla (\xi_\epsilon)_{ab}(x)|
		&\leq C+C\int_{\rho_\epsilon(x)}^{1/4}\frac{1}{r^2}\left|r^{2-n}\int_{B_r(x)}[djv_{\epsilon}(x')]_{ab}\frac{x-x'}{|x-x'|}\,dx'\right|dr\\
		&\leq C+C\int_{\rho_\epsilon(x)}^{1/4}\frac{\gamma}{r^2}\,dr\\
		&\leq C[1+\gamma/\rho_\epsilon(x)]
	\end{align*}
	(or $|\nabla\xi_\epsilon|\le C$ if $\rho_\epsilon(x)\ge\frac{1}{4}$),
	while clearly $|\nabla(\xi_\epsilon)_{ab}(x)|\le C$ if $x\in\{y_\epsilon\}\times[D_1\setminus D_{1/2}]$, by \cref{xi.comp.lem}
	and the fact that here $\rho_{\epsilon}\ge\frac{1}{4}$ (eventually).
	Combining this with \cref{fr.ubd} and an application of the coarea formula, since $|d\rho_\epsilon|=1$ (a.e.) we then see that
	\begin{align*}
		\limsup_{\epsilon\to 0}\frac{1}{\leps }\int_{P_{y_{\epsilon}}^{\perp}\setminus F_{\epsilon^{\alpha}}}|\nabla (\xi_{\epsilon})_{ab}|^2
		&\leq \limsup_{\epsilon\to 0}\frac{C}{\leps }\int_{P_{y_{\epsilon}}^{\perp}\setminus F_{\epsilon^{\alpha}}}\left(1+\frac{\gamma^2}{\rho_\epsilon(x)^2}\right)\\
		\text{(by the coarea formula) }&\leq \limsup_{\epsilon\to 0}\frac{C}{\leps }\int_{\epsilon^{\alpha}}^{1/4}\frac{\gamma}{r^2}\frac{d}{dr}(\mathcal{H}^2(F_r))\,dr\\
		\text{(integrating by parts) }&\leq \limsup_{\epsilon\to 0}\frac{C}{\leps }\int_{\epsilon^{\alpha}}^{1/4}\frac{2\gamma}{r^3}\mathcal{H}^2(F_r)\,dr\\
		\text{(by \cref{fr.ubd}) }&\leq \limsup_{\epsilon\to 0}\frac{C(\alpha)\gamma}{\leps }\int_{\epsilon^{\alpha}}^{1/4}\frac{dr}{r}\\
		&=C(\alpha)\gamma.
	\end{align*}
	And since $\gamma>0$ was arbitrary, it follows that
	\begin{equation}\label{off.f.van}
		\lim_{\epsilon\to 0}\frac{1}{\leps }\int_{P_{y_{\epsilon}}^{\perp}\setminus F_{\epsilon^{\alpha}}}|\nabla(\xi_{\epsilon})_{ab}|^2=0.
	\end{equation}
	
	To estimate the integral of $|\nabla(\xi_{\epsilon})_{ab}|^2$ on $F_{\epsilon^{\alpha}}$, we first
	observe that, by definition of $\xi_\epsilon$,
	\begin{align*}
		&|\nabla\xi_\epsilon|\le\frac{1}{|x|^{n-1}}*|\varphi\,djv_\epsilon|.
	\end{align*}
	We can then invoke \cref{lorentz} from the appendix and \cref{jv.w} to see that, since $y_{\epsilon}\in \mathcal{G}_{\epsilon,\delta_{\epsilon}}$, the gradient $|\nabla \xi_{\epsilon}|$ satisfies a uniform $L^{2,\infty}$ bound
	$$\|\nabla \xi_{\epsilon}\|_{L^{2,\infty}(P_{y_{\epsilon}}^{\perp})}\leq C(K)$$
	independent of $\epsilon$ along the slice $P_{y_{\epsilon}}^{\perp}$. In other words, we have the uniform estimate
	\begin{equation}\label{l2.infty.bd}
		\mathcal{H}^2(P_{y_\epsilon}^\perp\cap\{|\nabla \xi_{\epsilon}| > \lambda\})\leq \frac{C(K)}{\lambda^2}
	\end{equation}
	for every $\lambda\in (0,\infty)$. Moreover, since $|djv_{\epsilon}|\leq C\frac{W(u_\epsilon)}{\epsilon^2}$, we see from \cref{xi.comp.lem} that
	\begin{align*}
		|\nabla\xi_{\epsilon}| &\leq \int_0^\epsilon\frac{C}{\epsilon^2}+\int_\epsilon^{1/2}\frac{C}{r^2}\cdot r^{2-n}\int_{B_r(x)}\frac{W(u_\epsilon)}{\epsilon^2}+C \\
		&\le\frac{C(K)}{\epsilon}
	\end{align*}
	on $\{y_\epsilon\}\times D_{1/2}$, by definition of $\mathcal{G}_{\epsilon,\delta_\epsilon}$ (which gives $\int_{B_r(x)}\frac{W(u_\epsilon)}{\epsilon^2}\le Kr^{n-2}$), while $|\nabla\xi_\epsilon|\le C$ on the rest of the slice $P_{y_\epsilon}^\perp$.
	Hence, writing 
	$$A_{\lambda}:=\{x\in P_{y_{\epsilon}}^{\perp} : |\nabla\xi_{\epsilon}(x)|> \lambda\},$$
	we find that
	\begin{align*}
		\int_{F_{\epsilon^{\alpha}}}|\nabla\xi_{\epsilon}|^2
		&=\int_0^{C(K)/\epsilon}2\lambda\mathcal{H}^2(F_{\epsilon^{\alpha}}\cap A_\lambda)\,d\lambda \\
		&\leq \epsilon^{-2\alpha}\mathcal{H}^2(F_{\epsilon^\alpha})+\int_{\epsilon^{-\alpha}}^{C(K)\epsilon^{-1}}\frac{C(K)}{\lambda}\,d\lambda\\
		&\le C(\alpha)+C(K)\log(C(K)\epsilon^{\alpha-1}),
	\end{align*}
	thanks to \cref{fr.ubd} and \cref{l2.infty.bd}.
	Combining this with \cref{off.f.van}, we deduce that
	\begin{align*}
		\limsup_{\epsilon\to 0}\frac{1}{\leps }\int_{P_{y_{\epsilon}}^{\perp}}|\nabla(\xi_{\epsilon})_{ab}|^2
		\le C(K)(1-\alpha).
	\end{align*}
	Finally, since $\alpha\in (0,1)$ was arbitrary, we can take $\alpha\to 1$ to deduce that the limit vanishes, as desired.
\end{proof}

With the above preparations in place, we are now ready to prove the identity \eqref{intro.theta.comp}.

\begin{proposition}\label{pre.ener.id} For $u_{\epsilon}$ satisfying the hypotheses of \cref{dens.gap} and a sequence of slices $y_{\epsilon}\in \mathcal{G}_{\epsilon,\delta_{\epsilon}}$ with $p_j^{\epsilon}\in P_{y_{\epsilon}}^{\perp}$ and degrees $\kappa_j^{\epsilon}$ as above, we have
	$$\theta=\lim_{\epsilon \to 0}\left({\textstyle\sum_{j}}(\kappa_j^\epsilon)^2+{\textstyle\sum_{j<\ell}}2\kappa_j^\epsilon\kappa_\ell^\epsilon\frac{\lvert\log|p_j^\epsilon-p_\ell^\epsilon|\rvert}{\leps }\right).$$
\end{proposition}

\begin{proof} 
	
	To begin, fix a cutoff function $\chi(y,z)=\chi(z)$ on $\mathbb{R}^{n-2}\times \mathbb{R}^2$ satisfying $\chi(z)=1$ for $|z|\le\frac{1}{2}$, $\chi(z)=0$ for $|z|\ge\frac{3}{4}$, and $|d\chi|\leq C$.
	Since $y_{\epsilon}\in \mathcal{G}_{\epsilon,\delta_{\epsilon}}$, we have
	$$\int_{P_{y_{\epsilon}}^{\perp}}|jv_\epsilon-d^*\xi_\epsilon|^2=o(\leps );$$
	together with \cref{theta.xi}, this implies in particular that
	$$\|jv_\epsilon\|_{L^2(P_{y_\epsilon}^\perp)}\le C\leps ^{1/2}.$$
	Moreover, using \eqref{xi.conv} and the fact that $\mathcal{V}(u_{\epsilon})\cap P_{y_{\epsilon}}^{\perp}\subset \{y_\epsilon\}\times D_{1/4}$, it is also easy to check that
	$$\int_{\{y_\epsilon\}\times[D_1\setminus D_{1/2}]}(|\xi_\epsilon|^2+|d^*\xi_\epsilon|^2)\le C.$$
	
	As a consequence, using again \cref{theta.xi}, we have
	$$\pi\theta=\lim_{\epsilon \to 0}\frac{1}{\leps }\int_{P_{y_{\epsilon}}^{\perp}}\mz\chi|d^*\xi_\epsilon|^2
	=\lim_{\epsilon\to 0}\frac{1}{\leps }\int_{P_{y_{\epsilon}}^{\perp}}\frac{1}{2}\chi\langle d^*\xi_{\epsilon},jv_{\epsilon}\rangle.$$
	By \cref{slice.perp.small}, we can further refine this to see that
	$$\pi\theta=-\lim_{\epsilon\to 0}\frac{1}{\leps }\int_{P_{y_{\epsilon}}^{\perp}}\frac{1}{2}\chi[\partial_{n-1}(\xi_{\epsilon})_{n-1,n}jv_{\epsilon}(\partial_n)+\partial_n(\xi_{\epsilon})_{n,n-1}jv_{\epsilon}(\partial_{n-1})].$$
	Writing $\beta_{\epsilon}:=(\xi_{\epsilon})_{n-1,n}$ and integrating by parts on $P_{y_{\epsilon}}^{\perp}=\{y_{\epsilon}\}\times D_1$, we obtain
	\begin{align*}
		\pi\theta&=\lim_{\epsilon\to 0}\frac{1}{\leps }\int_{P_{y_{\epsilon}}^{\perp}}\frac{1}{2}\beta_{\epsilon}[\chi\,djv_{\epsilon}+d\chi\wedge jv_{\epsilon}].
	\end{align*}
	Since $\chi\equiv 1$ on $\{y_\epsilon\}\times D_{1/2}$, using the previous bounds and the fact that $|\beta_\epsilon|\le|\xi_\epsilon|$, we see that the second term gives no contribution in the limit. Also, $\chi\equiv 1$ on $\operatorname{spt}(djv_\epsilon)\cap P_{y_\epsilon}^\perp$. We conclude that
	\begin{equation}\label{theta.lim.comp}
		\pi\theta=\lim_{\epsilon\to 0}\frac{1}{\leps }\int_{P_{y_{\epsilon}}^{\perp}}\frac{1}{2}\beta_{\epsilon}\,djv_{\epsilon}(\partial_{n-1},\partial_n).
	\end{equation}
	
	Next, note that $djv_{\epsilon}|_{P_{y_{\epsilon}}^\perp}$ is supported in the set
	$$U_{\delta,\epsilon}=\{|u_{\epsilon}|<1-\delta\}\cap P_{y_{\epsilon}}^{\perp},$$
	and recall from \cref{cover.claim} that there exists a constant $C=C(\delta)$ and points $p_1^{\epsilon},\dots,p_k^{\epsilon}$ with $k\leq C$ such that
	\begin{align}\label{vort.in.disks}
		&\mathcal{V}(u_{\epsilon})\cap P_{y_\epsilon}^\perp\subseteq U_{\delta,\epsilon}\subseteq B_{C\epsilon}(p_1^{\epsilon})\cup\cdots\cup B_{C\epsilon}(p_k^{\epsilon}).
	\end{align}
	Using this in the right-hand side of \cref{theta.lim.comp}, and writing $D_{j,\epsilon}:=B_{C\epsilon}(p_j^{\epsilon})\cap P_{y_{\epsilon}}^{\perp}$, we see that
	\begin{align*}
		\int_{P_{y_{\epsilon}}^{\perp}}\frac{1}{2}\beta_{\epsilon}djv_{\epsilon}(\partial_{n-1},\partial_n)
		&={\textstyle\sum_{j=1}^k}\int_{D_{j,\epsilon}}\frac{1}{2}\beta_{\epsilon}djv_{\epsilon}(\partial_{n-1},\partial_n)\\
		&={\textstyle\sum_{j=1}^k}\beta_{\epsilon}(p_j^{\epsilon})\int_{D_{j,\epsilon}}\frac{1}{2}djv_{\epsilon}(\partial_{n-1},\partial_n)\\
		&\quad+{\textstyle\sum_{j=1}^k}\frac{1}{2}\int_{D_{j,\epsilon}}[\beta_{\epsilon}-\beta_{\epsilon}(p_j^{\epsilon})]djv_{\epsilon}(\partial_{n-1},\partial_n).
	\end{align*}
	Now, since $|d\beta_{\epsilon}|\leq |\nabla \xi_{\epsilon}|\leq C/\epsilon$ (as observed while proving the previous lemma), we have a pointwise bound of the form $|\beta_{\epsilon}-\beta_{\epsilon}(p_j^{\epsilon})|\leq C$ on $D_{j,\epsilon}$. Also, $\int_{P_{y_{\epsilon}}^{\perp}}|djv_\epsilon|\le C$ (by the pointwise bound $|djv_\epsilon|\le C\frac{W(u_\epsilon)}{\epsilon^2}$), and we deduce that
	\begin{align*}
		\pi\theta=\lim_{\epsilon \to 0}\frac{1}{\leps }{\textstyle\sum_{j=1}^k}\beta_{\epsilon}(p_j^\epsilon)\int_{D_{j,\epsilon}}\frac{1}{2}djv_\epsilon(\partial_{n-1},\partial_n).
	\end{align*}
	Moreover, noting that
	$$\int_{D_{j,\epsilon}}djv_{\epsilon}(\partial_{n-1},\partial_n)=2\pi \deg(v_{\epsilon},\partial D_{j,\epsilon})=2\pi \kappa_j^{\epsilon},$$
	it follows that
	\begin{equation}\label{theta.char.rev}
		\pi \theta=\lim_{\epsilon\to 0}\frac{1}{\leps}{\textstyle\sum_{i=1}^k}\beta_{\epsilon}(p_i^{\epsilon})\pi \kappa_i^{\epsilon}.
	\end{equation}
	
	With \eqref{theta.char.rev} in hand, to complete the proof, it suffices to show that
	\begin{equation}\label{beta.lim.char}
		\lim_{\epsilon\to 0}\frac{1}{\leps}(\beta_{\epsilon}(p_i^{\epsilon})-\kappa_i^{\epsilon}\leps-{\textstyle\sum_{j\neq i}} \kappa_j^{\epsilon}\lvert\log|p_i^{\epsilon}-p_j^{\epsilon}|\rvert)=0
	\end{equation}
	for every $i\in \{1,\ldots,k\}$. Up to relabeling the indices, it of course suffices to treat the case $i=1$, and assume that the distances
	$$r_j^{\epsilon}:=|p_1^{\epsilon}-p_j^{\epsilon}|$$
	are in increasing order $0=r_1^{\epsilon}<r_2^{\epsilon}\leq r_3^{\epsilon}\leq\cdots\leq r_k^{\epsilon}.$

	Now, fix $\gamma>0$ small and $\alpha\in (0,1)$ close to $1$, and observe that for $\epsilon$ sufficiently small (depending on $\alpha$ and $\gamma$), the ball $B_r(p_1^{\epsilon})$ satisfies the full hypotheses of \cref{lim.jac.lem}, with degree
	$$\kappa^{\epsilon}(r):={\textstyle\sum_{r_j^{\epsilon}<r}}\kappa_j^{\epsilon},$$
	whenever
	$$r\in \left[\epsilon^{\alpha},\frac{1}{4}\right]\setminus\bigcup_{j=2}^k\left[\frac{r_j^{\epsilon}-C\epsilon}{2},\frac{r_j^{\epsilon}+C\epsilon}{\delta_1}\right],$$
	where $\delta_1=\delta_1(\alpha,\gamma)>0$ is the constant from the hypotheses of \cref{lim.jac.lem}, since this ensures that, for each $j>1$, either $B_{2r}(p_1^\epsilon)$ is disjoint from $B_{C\epsilon}(p_j^\epsilon)$ if $r<r_j^{\epsilon}$ or $B_{C\epsilon}(p_j^\epsilon)\subseteq B_{\delta_1 r}(p_1^\epsilon)$ if $r\ge r_j^{\epsilon}$. In particular, for every such $r$, we have
	\begin{equation}\label{deg.jac}
		\left|2\pi\omega_{n-2}\kappa^{\epsilon}(r)-r^{2-n}\int_{B_r(p_1^{\epsilon})}[djv_{\epsilon}(x')]_{n-1,n}\,dx'\right|<\gamma.
	\end{equation}
	Now, it follows from \cref{xi.comp.lem} (and the bound $|djv_{\epsilon}|\leq C\frac{W(u_\epsilon)}{\epsilon^2}\le\frac{C}{\epsilon^2}$) that
	$$\left|2\pi\omega_{n-2}\beta_{\epsilon}(p_1^{\epsilon})-\int_{\epsilon}^{1/4}\frac{1}{r}\left(r^{2-n}\int_{B_r(p_1^{\epsilon})}[djv_{\epsilon}(x')]_{n-1,n}\,dx'\right) dr\right|\leq C.$$
	
	Combining this with \cref{deg.jac}, we have
	\begin{align*}
		&\left|2\pi\omega_{n-2}\beta_{\epsilon}(p_1^\epsilon)-2\pi\omega_{n-2}\int_{\epsilon}^{1/4}\frac{\kappa^{\epsilon}(r)}{r}\,dr\right| \\
		&\le C+\gamma\log(1/4\epsilon)
		+\int_I\frac{1}{r}\left|2\pi\omega_{n-2}\kappa^{\epsilon}(r)-r^{2-n}\int_{B_r(p_1^{\epsilon})}[djv_{\epsilon}(x')]_{n-1,n}\,dx'\right|,
	\end{align*}
	where
	\begin{align*}
		&I:=(\epsilon,\epsilon^\alpha)\cup\bigcup_{j=2}^k\left(\frac{r_j^{\epsilon}-C\epsilon}{2},\frac{r_j^{\epsilon}+C\epsilon}{\delta_1}\right).
	\end{align*}
	Appealing once more to the uniform bound 
	$$r^{2-n}\int_{B_r(p_1^{\epsilon})}|djv_{\epsilon}|\leq r^{2-n}\int_{B_r(p_1^{\epsilon})}\frac{W(u_{\epsilon})}{\epsilon^2}\leq C(K)=C$$
	(by definition of $\mathcal{G}_{\epsilon,\delta_\epsilon}$) and noting that each $r_j^\epsilon\ge 2C\epsilon$ for $j\geq 2$, it therefore follows that
	\begin{align*}
		&\left|\beta_{\epsilon}(p_1^{\epsilon})-\int_{\epsilon}^{1/4}\frac{\kappa^{\epsilon}(r)}{r}\,dr\right| \\
		&\leq C+\gamma\leps +\int_I\frac{C}{r}\,dr\\
		&\le C+\gamma\leps +C\log(\epsilon^{\alpha-1})+{\textstyle\sum_{j=1}^k}\log\left(\frac{(r_j^{\epsilon}+C\epsilon)/\delta_1}{(r_j^{\epsilon}-C\epsilon)/2}\right)\\
		&\le C+\gamma\leps +C(1-\alpha)\leps +C(\alpha,\gamma).
	\end{align*}
	
	Dividing through by $\leps $ and passing to the limit $\epsilon\to 0$, we deduce that
	$$\limsup_{\epsilon\to 0}\frac{1}{\leps}\left|\beta_{\epsilon}(p_1^{\epsilon})-\int_{\epsilon}^{1/4}\frac{\kappa^{\epsilon}(r)}{r}\,dr\right|\leq \gamma+C(1-\alpha)$$
	for any $\gamma>0$ and $\alpha\in (0,1)$. In particular, taking $\gamma\to 0$ and $\alpha\to 1$, we deduce that
	\begin{equation}\label{limbeta.comp}
		\limsup_{\epsilon\to 0}\frac{1}{\leps}\left|\beta_{\epsilon}(p_1^{\epsilon})-\int_{\epsilon}^{1/4}\frac{\kappa^{\epsilon}(r)}{r}\,dr\right|=0.
	\end{equation}
	But now we need only observe that
	\begin{align*}
		\int_{\epsilon}^{1/4}\frac{\kappa^{\epsilon}(r)}{r}\,dr&=\int_{\epsilon}^{1/4}\frac{1}{r}\cdot{\textstyle\sum_{r_j^{\epsilon}<r}}\kappa_j^{\epsilon}\,dr\\
		&=\int_{\epsilon}^{1/4}\frac{\kappa_1^{\epsilon}}{r}\,dr+{\textstyle\sum_{j=2}^k}\int_{r_j^{\epsilon}}^{1/4}\frac{\kappa_j^{\epsilon}}{r}\,dr\\
		&=\kappa_1^{\epsilon}(\log(1/\epsilon)+\log(1/4))+{\textstyle\sum_{j=2}^k}\kappa_j^{\epsilon}(\log(1/r_j^{\epsilon})+\log(1/4))\\
		&=\kappa_1^{\epsilon}\leps+{\textstyle\sum_{j=2}^k}\kappa_j^{\epsilon}\lvert\log |p_1^{\epsilon}-p_j^{\epsilon}|\rvert+O(1).
	\end{align*}
	Together with \eqref{limbeta.comp}, this gives the desired identity \eqref{beta.lim.char}, completing the proof.
\end{proof}

\section{Solutions of density $<2$}

Denote by $\mathcal{D}\subset\mathbb{R}$ the collection of densities $\theta$ arising as in the statement of \cref{dens.gap}, and set
$$\theta_{min}:=\inf \mathcal{D}.$$
Note that $\theta_{min}>0$, by virtue of \cref{asymp}.
By a simple diagonal sequence argument, we see that $\theta_{min}\in \mathcal{D}$. 
%
In terms of the minimum density $\theta_{min}$, it is clear that \cref{dens.gap} is equivalent to the following proposition.

\begin{proposition}\label{low.dens.prop} Under the hypotheses of \cref{dens.gap}, if $\theta<2\theta_{min}$, then $\theta=1$.
\end{proposition}

In particular, having established \cref{low.dens.prop}, it follows immediately that $\theta_{min}=1$, and that $\mathcal{D}\cap [1,2)=\{1\}$.

To begin the proof of \cref{low.dens.prop}, we note that the assumption $\theta<2\theta_{min}$ allows us to make the following reduction, showing roughly that the vorticity set lies close to a single $(n-2)$-plane at all scales $\geq \epsilon^{\tau_0}$ for some fixed $\tau_0>0$. Recall that $\mathcal{G}_{\epsilon,\delta}$ is the set of $\delta$-good slices for $u_\epsilon$, introduced in the previous section.

\begin{lemma}\label{loc.vort.lem} Let $\tau_0:=\frac{2\theta_{min}-\theta}{4\theta_{min}}\in(0,\mz)$. Under the assumptions of \cref{low.dens.prop}, for any $\eta\in (0,1]$ there exists $c(\eta)>0$ independent of $\epsilon$ such that, for any sequences $\delta_{\epsilon}\to 0$ and $y_{\epsilon}\in \mathcal{G}_{\epsilon,\delta_{\epsilon}}$, there exists
	$$p_{\epsilon}=(y_{\epsilon},z_{\epsilon})\in P_{y_{\epsilon}}^{\perp}$$
	such that
	\begin{equation}\label{close.to.vort}
		B_{\epsilon^{1/2}}(p_\epsilon)\cap\mathcal{V}(u_\epsilon)\neq\emptyset
	\end{equation}
	and, for any 
	$x=(y,z)\in Q$,
	we have
	\begin{equation}\label{vort.dist}
		\dist(x,\mathcal{V}(u_{\epsilon})) \ge c|z-z_{\epsilon}|,
	\end{equation}
	provided that $|z-z_{\epsilon}|\ge\eta\max\{\epsilon^{\tau_0},|y-y_\epsilon|\}$ and that
	$\epsilon$ is sufficiently small (depending on $\eta$).
\end{lemma}

\begin{proof}
	
	Given $y_\epsilon\in \mathcal{G}_{\epsilon,\delta_{\epsilon}}$, choose $z_\epsilon\in D_1^2$ such that $p_\epsilon=(y_\epsilon,z_\epsilon)$ is the closest point in the slice $P_{y_{\epsilon}}^{\perp}$ to the vorticity set $\mathcal{V}(u_\epsilon)$.
	If \cref{close.to.vort} fails along a subsequence, then setting $r_\epsilon:=\dist(p_\epsilon,\mathcal{V}(u_\epsilon))\ge\epsilon^{1/2}$ we can consider the rescaled solutions
	$$\tilde{u}_{\tilde\epsilon}(x):=u_{\epsilon}(p_{\epsilon}+r_{\epsilon}x),$$
	which solve the $\tilde{\epsilon}$-Ginzburg--Landau equation on $B_2(0)$ with $\tilde{\epsilon}:=\epsilon/r_{\epsilon}\le\epsilon^{1/2}$ (note also that $z_\epsilon\to 0$ and $r_\epsilon\to 0$, since $\mathcal{V}(u_\epsilon)\to P\cap B_2(0)$ in the Hausdorff sense).
	By our assumption $y_{\epsilon}\in\mathcal{G}_{\epsilon,\delta_\epsilon}$, we then see that the rescaled solutions $\tilde{u}_{\tilde\epsilon}$ satisfy
	\begin{align}\label{still.van}
		&\lim_{\tilde\epsilon\to 0}\frac{1}{\lteps }\int_{B_2(0)}|d\tilde{u}_{\tilde\epsilon}(P)|^2=0.
	\end{align}
	By \cref{asymp}, the concentrated part of the limiting energy measure
	$$\mu=\lim_{\epsilon\to 0}\frac{e_{\tilde{\epsilon}}(\tilde{u}_{\tilde\epsilon})}{\pi\lteps }\,dx$$
	is a stationary rectifiable varifold $V$, and by reversing the proof of \cref{par.van.lem}, it follows that its tangent planes coincide with $P$, and we deduce that $V$ is given by a locally finite union of $(n-2)$-planes parallel to $P$, each with multiplicity at least $\theta_{min}$.
	
	By assumption, $\mathcal{V}(\tilde u_{\tilde\epsilon})$ does not intersect $B_1(0)$, and since $p_\epsilon$ was chosen to be the closest point in $P_{y_{\epsilon}}^{\perp}$ to the vorticity set,
	we see that there exists $q_{\tilde\epsilon}\in\R^{n-2}$ with $|q_{\tilde\epsilon}|=1$ such that $(q_{\tilde\epsilon},0)\in\mathcal{V}(\tilde u_{\tilde\epsilon})$.
	Since the support of $|V|$ is the Hausdorff limit of $\mathcal{V}(\tilde u_{\tilde\epsilon})$, it follows that it is disjoint from $B_1(0)$, but at the same time it contains the whole $(n-2)$-plane $P+(q_0,0)=P$ (intersected with $B_2$),
	for a subsequential limit $q_0$ of $q_{\tilde\epsilon}$, a contradiction.
	
	Now let us verify \cref{vort.dist}. We proceed by a similar contradiction argument: suppose to the contrary that there exists $x_{\epsilon}=(y'_{\epsilon},z_{\epsilon}')\in 
	Q$ such that $|z_\epsilon'-z_{\epsilon}|\ge\eta\max\{\epsilon^{\tau_0},|y_\epsilon'-y_\epsilon|\}$, but
	\begin{align}\label{almost.vort}
		&\lim_{\epsilon\to 0}\frac{\dist(x_{\epsilon},\mathcal{V}(u_{\epsilon}))}{|z_{\epsilon}'-z_{\epsilon}|}=0.
	\end{align}
	Evidently, since $\dist(x_\epsilon,\mathcal{V}(u_\epsilon))\to 0$, we must have $z_\epsilon'\to 0$, and hence
	$$s_{\epsilon}:=|z_{\epsilon}'-z_{\epsilon}|\to 0.$$
	For fixed small $\delta\in(0,1)$, we can consider the rescaled solutions
	$$\tilde{u}_{\tilde\epsilon}(x):=u_{\epsilon}(p_{\epsilon}+s_{\epsilon}x/\delta),$$
	which solve the $\tilde{\epsilon}$-Ginzburg--Landau equation on $B_2(0)$ with $\tilde{\epsilon}:=\delta\epsilon/s_{\epsilon}\in(\delta\epsilon,\eta^{-1}\epsilon^{1-\tau_0})$ (since $s_\epsilon\ge\eta\epsilon^{\tau_0}$). 
	Again, the rescaled solutions $\tilde{u}_{\tilde\epsilon}$ satisfy \cref{still.van}.
	Also, by \cref{close.to.vort}, since $\tau_0<\mz$ we have
	$$\lim_{\epsilon \to 0}\frac{\dist(p_\epsilon,\mathcal{V}(u_\epsilon))}{s_\epsilon}\le\lim_{\epsilon \to 0}\frac{\epsilon^{1/2}}{\eta\epsilon^{\tau_0}}=0.$$
	After passing to a subsequence, by the last observation and \cref{almost.vort}, the Hausdorff limit of $\mathcal{V}(\tilde{u}_{\tilde\epsilon})\cap B_1(0)$ must contain $0$ and the point
	$$\xi:=\lim_{\epsilon\to 0}\delta \cdot \frac{x_{\epsilon}-p_{\epsilon}}{s_{\epsilon}},$$
	which exists thanks to the assumption that $|y_\epsilon'-y_\epsilon|\le\eta^{-1}|z_\epsilon'-z_\epsilon|=\eta^{-1}s_\epsilon$, and satisfies
	$$|P^{\perp}(\xi)|=\lim_{\epsilon\to 0}\frac{\delta}{s_{\epsilon}}|z_{\epsilon}'-z_{\epsilon}|=\delta.$$
	As observed above, from \cref{still.van} it follows that in $B_1(0)$ the support of the energy concentration varifold contains the parallel $(n-2)$-planes $P$ and $(P+\xi)$, and therefore the limit energy measure $\mu=\lim_{\epsilon\to 0}\frac{e_{\tilde{\epsilon}}(\tilde{u}_{\tilde{\epsilon}})}{\pi\lteps}\,dx$ satisfies
	$$\mu\geq \theta_{min} \mathcal{H}^{n-2}\mrestr P+\theta_{min}\mathcal{H}^{n-2}\mrestr (P+\xi).$$
	Hence,
	$$\liminf_{\epsilon\to 0}\int_{B_1(0)}\frac{e_{\tilde{\epsilon}}(\tilde{u}_{\tilde\epsilon})}{\pi\lteps }\geq \mu(B_1(0))\geq \theta_{min}\omega_{n-2}[1+(1-\delta^2)^{\frac{n-2}{2}}].$$
	Now, returning to the original family of solutions $u_{\epsilon}$, by the monotonicity formula \cref{monotonicity} we have
	\begin{align*}
		\theta&=\lim_{\epsilon\to 0}\frac{1}{\omega_{n-2}}\int_{B_{1}(p_{\epsilon})}\frac{e_{\epsilon}(u_{\epsilon})}{\pi\leps }\\
		&\geq \liminf_{\epsilon\to 0}\frac{1}{\omega_{n-2}(s_{\epsilon}/\delta)^{n-2}}\int_{B_{s_{\epsilon}/\delta}(p_{\epsilon})}\frac{e_{\epsilon}(u_{\epsilon})}{\pi\leps }\\
		&=\liminf_{\epsilon\to 0}\frac{\lteps }{\leps }\frac{1}{\omega_{n-2}}\int_{B_1(0)}\frac{e_{\tilde{\epsilon}}(\tilde{u}_{\tilde\epsilon})}{\pi\lvert\log{\tilde{\epsilon}}\rvert}\\
		&\geq (1-\tau_0)\theta_{min}[1+(1-\delta^2)^{\frac{n-2}{2}}],
	\end{align*}
	and since $\delta>0$ was arbitrary, it follows that
	$$\theta\geq 2(1-\tau_0)\theta_{min}.$$
	However, this cannot hold since we have chosen $\tau_0$ such that $\tau_0<\frac{2\theta_{min}-\theta}{2\theta_{min}}$.
	We thus reach a contradiction, concluding the proof.
\end{proof}

Next, choosing a radial cutoff function $\phi\in C^1_c(D_1)$ with $0\le\phi\le 1$ and $\phi\equiv 1$ on $D_{1/2}$, and applying \cite[Theorem~2.1]{JSoner},
we see that
$$\limsup_{\epsilon\to 0}\left|\int_{D_1}\phi J_\epsilon\right|
\le\lim_{\epsilon\to 0}\int_{P_{y_\epsilon}^{\perp}}\frac{e_\epsilon(u_\epsilon)}{\leps }
=\pi\theta,$$
where $J_\epsilon:=\mz dju_\epsilon$ is the Jacobian of $u_\epsilon$ along the slice $P_{y_\epsilon}^{\perp}\cong D_1$
(in the last equality we used the fact that $y_\epsilon\in\mathcal{G}_{\epsilon,\delta_\epsilon}$,
which gives $\lvert\int_{P_{y_\epsilon}^{\perp}}\frac{e_\epsilon(u_\epsilon)}{\leps }-\pi\theta\rvert\le\delta_\epsilon$).
But, integrating by parts and using polar coordinates, we have
\begin{align*}
	&\int_{D_1}\phi J_\epsilon
	=-\frac{1}{2}\int_0^1\de_r\phi\left[\int_{\{y_\epsilon\}\times\de D_r}ju_\epsilon(d\theta)\right]\,dr,
\end{align*}
and since $|u_\epsilon|\to 1$ on $\operatorname{spt}(d\phi)\subset D_1\setminus D_{1/2}$, the last expression is the same as
\begin{align*}
	-\frac{1}{2}\int_{1/2}^1\de_r\phi\left[\int_{\{y_\epsilon\}\times\de D_r}jv_\epsilon(d\theta)\right]\,dr
	=\pi\deg(v_{\epsilon},\{0\}\times S^1),
\end{align*}
up to an infinitesimal error. We deduce that the degree
$$\kappa_\epsilon:=\deg(v_{\epsilon},\{0\}\times S^1)\in \mathbb{Z}$$
eventually satisfies
$$|\kappa_\epsilon|\le\theta<2\theta_{min}\le 2.$$
Thus, passing to a subsequence, we may assume that $\kappa_\epsilon=\kappa$ is constant and
\begin{align}\label{kappa.at.most.1}
	&|\kappa|\le 1.
\end{align}
Applying \cref{loc.vort.lem} (for fixed $\eta$), note also that
\begin{align}\label{deg.all.scales}
	&\kappa=\deg(v_\epsilon,\{y_\epsilon\}\times\de D_r(z_\epsilon))\quad\text{for all }r\in[\eta\epsilon^{\tau_0},\textstyle{\mz}],
\end{align}
since \cref{vort.dist} implies that $|u_\epsilon|>\mz$ on the annulus $\{y_\epsilon\}\times[D_1(0)\setminus D_{\eta\epsilon^{\tau_0}}(z_\epsilon)]$.

Ultimately, we wish to show that $|\kappa|=1=\theta$. First, we compute the energy contribution from an annular region centered at $p_{\epsilon}$.

\begin{lemma}\label{ann.en.lem} For any fixed $\eta\in (0,1]$ and $p_{\epsilon}$, $\tau\in (0,\tau_0]$ as in \cref{loc.vort.lem}, we have
	$$\lim_{\epsilon\to 0}\frac{1}{\leps (\epsilon^{\tau})^{n-2}}\int_{\mathcal{A}^{\eta}_{\epsilon^{\tau}}(p_{\epsilon})}e_{\epsilon}(u_{\epsilon})=\pi\omega_{n-2} \kappa^2\tau,$$
	where we set
	\begin{align*}
		&\mathcal{A}^{\eta}_{r}(p_{\epsilon}):=B_r^{n-2}(y_\epsilon)\times[D_1^2(0)\setminus D_{\eta r}^2(z_\epsilon)].
	\end{align*}
	More generally, for any family of radii $r_{\epsilon}\in [\epsilon^{\tau_0},1)$ for which $\lim_{\epsilon\to 0}\frac{\log(1/r_{\epsilon})}{\log(1/\epsilon)}=\tau$, we have
	\begin{equation}\label{ann.en.conc}
		\lim_{\epsilon\to 0}\frac{1}{\leps r_{\epsilon}^{n-2}}\int_{\mathcal{A}^{\eta}_{r_{\epsilon}}(p_{\epsilon})}e_{\epsilon}(u_{\epsilon})=\pi \omega_{n-2}\kappa^2\tau.
	\end{equation}
\end{lemma}

\begin{proof}
	Let $r_{\epsilon}\in [\epsilon^{\tau_0},1)$ be a family of radii as above. For simplicity, we assume that $p_{\epsilon}=0$, and write simply $\mathcal{A}_{r_{\epsilon}}^{\eta}=\mathcal{A}^{\eta}_{r_{\epsilon}}(0)$.
	
	Let $\beta_{\epsilon}(x):=(\xi_{\epsilon})_{n-1,n}(x)$.
	To begin with, we claim that
	\begin{equation}\label{ann.en.2}
		\lim_{\epsilon\to 0}\frac{1}{\leps r_{\epsilon}^{n-2}}\int_{\mathcal{A}_{r_{\epsilon}}}e_{\epsilon}(u_{\epsilon})
		=\lim_{\epsilon\to 0}\frac{1}{2\leps r_{\epsilon}^{n-2}}\int_{B_{r_{\epsilon}}^{n-2}(0)}\int_{\partial D_{\eta r_{\epsilon}}(0)}\beta_\epsilon\,d^*\xi_{\epsilon}(\partial_{\theta}).
	\end{equation}
	Indeed, since $0\in \mathcal{G}_{\epsilon,\delta_{\epsilon}}$ for some sequence $\delta_{\epsilon}\to 0$ by assumption, it follows from the definition of $\mathcal{G}_{\epsilon,\delta_{\epsilon}}$ that
	\begin{equation}\label{ann.en.1}
		\lim_{\epsilon\to 0}\frac{1}{\leps r_{\epsilon}^{n-2}}\int_{\mathcal{A}^{\eta}_{r_{\epsilon}}}e_{\epsilon}(u_{\epsilon})
		=\lim_{\epsilon\to 0}\frac{1}{2\leps r_{\epsilon}^{n-2}}\int_{\mathcal{A}^{\eta}_{r_{\epsilon}}}|d^*\xi_{\epsilon}|^2.
	\end{equation}
	Moreover, since $0\in \mathcal{G}_{\epsilon,\delta_{\epsilon}}$, for any fixed $\delta_1>0$, it is easy to see that
	\begin{equation}\label{shifted.bd}
		r^{2-n}\int_{B_{2r}(x)}|du_{\epsilon}(P)|^2<\delta_1\leps
	\end{equation}
	for any $x=(y,z)\in \mathcal{A}_{r_{\epsilon}}$ with $|y|\le\mz$ and $\frac{1}{8}\ge r\geq c\eta r_{\epsilon}\geq c\eta \epsilon^{\tau_0}$, where we take $c>0$ to be the constant from \cref{loc.vort.lem}, for $\epsilon$ sufficiently small. As a consequence, for any given $\gamma>0$ we see that the hypotheses of \cref{lim.jac.lem} hold for every such ball $B_{2r}(x)$.
	Combining this with \cref{xi.comp.lem}, the fact that $\operatorname{spt}(djv_\epsilon)\subseteq\mathcal{V}(u_\epsilon)$, and \cref{jv.w}, we find that, for $x\in \mathcal{A}^{\eta}_{r_{\epsilon}}$ and $(a,b)\neq (n-1,n)$,
	\begin{align*}
		|\nabla (\xi_\epsilon)_{ab}|(x)
		&\leq C\int_{\dist(x,\mathcal{V}(u_{\epsilon}))}^{1/8}\frac{1}{r^2}\left|r^{2-n}\int_{B_r(x)}[djv_{\epsilon}(x')]_{ab}\frac{x-x'}{|x-x'|}\,dx'\right|\,dr+C\\
		&\leq C+C\int_{\dist(x,\mathcal{V}(u_{\epsilon}))}^{1/8}\frac{1}{r^2}\gamma\,dr
		\leq C+\frac{C\gamma}{\dist(x,\mathcal{V}(u_{\epsilon}))},
	\end{align*}
	provided $\epsilon$ is sufficiently small (if $\dist(x,\mathcal{V}(u_{\epsilon}))\ge\frac{1}{8}$, then actually $|\nabla\xi_\epsilon|\le C$). In particular, since $x\in\mathcal{A}_{r_\epsilon}^\eta$, by \cref{loc.vort.lem} we know that $\dist(x,\mathcal{V}(u_{\epsilon}))\geq c|z|$, and it follows that
	\begin{align*}
		\limsup_{\epsilon\to 0}\frac{1}{\leps r_{\epsilon}^{n-2}}\int_{\mathcal{A}^{\eta}_{r_{\epsilon}}}|\nabla (\xi_\epsilon)_{ab}|^2
		&\leq \limsup_{\epsilon\to 0}\frac{1}{\leps r_{\epsilon}^{n-2}}\int_{\mathcal{A}^{\eta}_{r_{\epsilon}}}\frac{C\gamma^2}{|z|^2}\\
		&\leq \limsup_{\epsilon\to 0}C\gamma^2\frac{\log(1/\eta r_{\epsilon})}{\log(1/\epsilon)}\\
		&\leq C\gamma^2 \tau_0
	\end{align*}
	for any $\gamma>0$; hence, 
	\begin{equation}
		\lim_{\epsilon\to 0}\frac{1}{\leps r_\epsilon^{n-2}}\int_{\mathcal{A}^{\eta}_{r_{\epsilon}}}|\nabla (\xi_{\epsilon})_{ab}|^2=0\quad\text{for }(a,b)\neq (n-1,n).
	\end{equation}
	
	Combining this with \cref{ann.en.1}, we then see that
	\begin{align*}
		&\lim_{\epsilon\to 0}\frac{1}{\leps r_{\epsilon}^{n-2}}\int_{\mathcal{A}^{\eta}_{r_{\epsilon}}}e_{\epsilon}(u_{\epsilon})
		=\lim_{\epsilon \to 0} \frac{1}{2\leps r_{\epsilon}^{n-2}}\int_{\mathcal{A}^{\eta}_{r_{\epsilon}}}[d^*\xi_{\epsilon}(\partial_{n-1})\partial_n\beta_{\epsilon}-d^*\xi_{\epsilon}(\partial_n)\partial_{n-1}\beta_{\epsilon}].
	\end{align*}
	
	A simple integration by parts shows that
	\begin{align*}
		&\int_{D_1\setminus D_{\eta r_{\epsilon}}}[d^*\xi_{\epsilon}(\partial_{n-1})\partial_n\beta_{\epsilon}-d^*\xi_{\epsilon}(\partial_n)\partial_{n-1}\beta_{\epsilon}] \\
		&=\int_{\partial D_{\eta r_{\epsilon}}}\beta_{\epsilon}\,d^*\xi_{\epsilon}(\partial_{\theta})-\int_{\partial D_1}\beta_{\epsilon}\,d^*\xi_{\epsilon}(\partial_{\theta})+\int_{D_1\setminus D_{\eta r_\epsilon}}\beta_\epsilon dd^*\xi_\epsilon[\de_{n-1},\de_n]
	\end{align*}
	on any slice $P_y^\perp\cong D_1$.
	Also, recall that $dd^*\xi_\epsilon=djv_\epsilon-dh_\epsilon$.
	Hence, using the fact that $djv_{\epsilon}=0$ on $\mathcal{A}^{\eta}_{r_{\epsilon}}$, as well as $|\beta_\epsilon|\le|\xi_\epsilon|$, we deduce from the preceding identities that
	\begin{align*}
		&\lim_{\epsilon\to 0}\frac{1}{\leps r_{\epsilon}^{n-2}}\left|\int_{\mathcal{A}^{\eta}_{r_{\epsilon}}}e_{\epsilon}(u_{\epsilon})
		-\frac{1}{2}\int_{B_{r_{\epsilon}}^{n-2}}\left(\int_{\partial D_{\eta r_{\epsilon}}}\beta_{\epsilon}\,d^*\xi_{\epsilon}(\partial_{\theta})-\int_{\partial D_1}\beta_{\epsilon}\,d^*\xi_{\epsilon}(\partial_{\theta})\right)\right| \\
		&=\lim_{\epsilon\to 0}\frac{1}{2\leps r_{\epsilon}^{n-2}}\left|\int_{B_{r_{\epsilon}}^{n-2}}\int_{D_1\setminus D_{\eta r_{\epsilon}}}\beta_{\epsilon}dd^*\xi_\epsilon[\de_{n-1},\de_n]\right| \\
		&\le\lim_{\epsilon\to 0}\frac{C}{\leps r_\epsilon^{n-2}}\|dh_\epsilon\|_{L^\infty(Q)}\int_{B_{r_\epsilon}^{n-2}\times D_1}|\xi_\epsilon| \\
		&\le\lim_{\epsilon \to 0}\frac{C}{\leps }\|dh_\epsilon\|_{L^\infty(Q)} \\
		&=0,
	\end{align*}
	where in the last two lines we used the fact that $0\in\mathcal{G}_{\epsilon,\delta_\epsilon}$ (giving $\int_{B_{r_\epsilon}^{n-2}\times D_1}|\xi_\epsilon|\le Kr_\epsilon^{n-2}$) and \cref{h.bd}.
	
	Now, by \cref{bos2} and the monotonicity formula \cref{monotonicity}, together with the pointwise bound \cref{jv.w}, it is easy to see that
	\begin{align}\label{djv.int.est}
		r^{2-n}\left|\int_{B_r(x)}\varphi(x')[djv_{\epsilon}(x')]_{n-1,n}\,dx'\right|\leq C
	\end{align}
	for all $r\geq \epsilon^{1/2}$, for some $C$ independent of $\epsilon$ (alternatively, a similar bound with $CK$ in place of $C$ follows immediately from \cref{jv.w} and the fact that $0=y_\epsilon\in\mathcal{G}_{\epsilon,\delta_\epsilon}$).
	Combining this observation with \cref{xi.comp.lem}, writing
	$$\rho_{\epsilon}(x):=\dist(x,\mathcal{V}(u_{\epsilon})),$$
	we see that
	\begin{equation}\label{pt.xi.ests}
		|\xi_{\epsilon}|\leq C\log(1/\rho_{\epsilon})+C,\quad|d^* \xi_{\epsilon}|\leq C/\rho_{\epsilon},
	\end{equation}
	whenever $\rho_\epsilon\ge\epsilon^{1/2}$.
	In particular, since $\rho_{\epsilon}\to 1$ uniformly on $B_1^{n-2}(0)\times \partial D_1(0)$ as $\epsilon\to 0$, it follows that
	$$\frac{1}{2\leps r_{\epsilon}^{n-2}}\int_{B_{r_{\epsilon}}^{n-2}}\int_{\partial D_1}|\beta_\epsilon\,d^*\xi_{\epsilon}(\partial_{\theta})|\leq \frac{C}{\leps }\to 0$$
	as $\epsilon\to 0$, which proves \cref{ann.en.2}.

	In order to estimate the right-hand side of \cref{ann.en.2}, we let $S_{\epsilon}:=B_{r_{\epsilon}}^{n-2}\times \partial D_{\eta r_{\epsilon}}$ and, for $x=(y,z)\in S_\epsilon$, we first show for any fixed $\gamma>0$ the uniform bound
	\begin{align}\label{ann.en.psi.bd}
		&|2\pi\omega_{n-2}\beta_{\epsilon}(x)-2\pi\kappa\omega_{n-2}\log(1/r_{\epsilon})|\leq C(\gamma,\eta)+\gamma \leps .
	\end{align}
	Note that by \cref{loc.vort.lem} we have
	\begin{align}\label{rho.comp.reps}
		&2r_{\epsilon}\geq \rho_{\epsilon}(x)\geq c\eta r_{\epsilon}\quad\text{for all }x\in S_{\epsilon}.
	\end{align}
	Moreover,
	again by \cref{loc.vort.lem},
	for any $x=(y,z)\in S_{\epsilon}\subset \mathcal{A}_{r_{\epsilon}}$ we have
	$$\mathcal{V}(u_{\epsilon})\cap P_y^{\perp}
	\subseteq \{y\}\times D_{\eta r_\epsilon}(0)
	\subseteq B_{2r_\epsilon}(x).$$
	Hence, fixing an arbitrary $\gamma>0$, we see that the full hypotheses of \cref{lim.jac.lem} hold on $B_{2r}(x)$ for $\epsilon>0$ sufficiently small, whenever
	$$\frac{2r_{\epsilon}}{\delta_1(\tau_0,\gamma)}<r\leq \frac{1}{8}$$
	(note that \cref{lim.jac.lem.h1} holds by \cref{shifted.bd}).
	As in \cref{deg.all.scales}, the degree $\deg(v_\epsilon,\{y\}\times \de D_r(z))=\kappa$, and from \cref{vort.char} we deduce that
	\begin{equation}\label{big.jac.comp}
		\left|2\pi\kappa\omega_{n-2}-r^{2-n}\int_{B_r(x)}[djv_{\epsilon}(x')]_{n-1,n}\right|<\gamma
	\end{equation}
	for $x\in S_\epsilon$ and $\frac{2r_{\epsilon}}{\delta_1(\tau_0,\gamma)}<r\leq \frac{1}{8}$.
	
	Next, recall from \cref{xi.comp.lem} that
	$$2\pi\omega_{n-2}\beta_{\epsilon}(x)=\int_0^{\infty}\frac{1}{r}\left(r^{2-n}\int_{B_r(x)}\varphi(x')[djv_{\epsilon}(x')]_{n-1,n}\,dx'\right) dr,$$
	and note that $\varphi|_{B_r(x)}\equiv 1$ for $r\leq \frac{1}{8}$, so that
	$$\left|2\pi\omega_{n-2}\beta_{\epsilon}(x)-\int_{\rho_{\epsilon}(x)}^{1/8}\frac{1}{r}\left(r^{2-n}\int_{B_r(x)}[djv_{\epsilon}(x')]_{n-1,n}\,dx'\right) dr\right|\leq C$$
	for a suitable $C$ independent of $\epsilon>0$. Note also that, by \cref{rho.comp.reps} and \cref{djv.int.est},
	\begin{align*}
		\left|\int_{\rho_{\epsilon}(x)}^{2r_{\epsilon}/\delta_1}\frac{1}{r}\left(r^{2-n}\int_{B_r(x)}[djv_{\epsilon}(x')]_{n-1,n}\,dx'\right) dr\right|
		&\leq \int_{c\eta r_{\epsilon}}^{2r_{\epsilon}/\delta_1}\frac{C}{r}\,dr
		\leq C\log(2/c\eta\delta_1),
	\end{align*}
	while, by \cref{big.jac.comp},
	\begin{align*}
		\int_{2r_{\epsilon}/\delta_1}^{1/8}\frac{1}{r}\left|2\pi \kappa \omega_{n-2}-r^{2-n}\int_{B_r(x)}[djv_{\epsilon}(x')]_{n-1,n}\,dx'\right|\,dr
		&<\gamma\int_{2r_{\epsilon}/\delta_1}^{1/8}\frac{dr}{r}
		<\gamma \log(1/r_{\epsilon}).
	\end{align*}
	Combining these bounds, we arrive at \cref{ann.en.psi.bd}.
	
	Turning now to the right-hand side of \cref{ann.en.2}, it follows that, for any $\gamma>0$,
	\begin{align*}
		\lim_{\epsilon\to 0}\frac{r_{\epsilon}^{2-n}}{2\leps }\left|\int_{S_{\epsilon}}(\beta_\epsilon-\kappa\lvert\log r_{\epsilon}\rvert)\,d^*\xi_{\epsilon}(\partial_{\theta})\right|
		&\leq \limsup_{\epsilon\to 0}\frac{r_{\epsilon}^{2-n}}{\leps }\int_{S_{\epsilon}}(C(\gamma,\eta)+\gamma \leps )|d^*\xi_{\epsilon}|\\
		\text{(using \cref{pt.xi.ests})}\ &\leq \limsup_{\epsilon\to 0}\frac{r_{\epsilon}^{2-n}}{\leps }\int_{S_{\epsilon}}C\cdot\frac{C(\gamma,\eta)+\gamma\leps }{c\eta r_{\epsilon}}\\
		&=C(\eta)\gamma,
	\end{align*}
	and since $\gamma>0$ was arbitrary, it follows that the limit on the left-hand side must vanish. In particular, returning to \cref{ann.en.2}, we see that
	\begin{align*}
		\lim_{\epsilon\to 0}\frac{1}{\leps r_{\epsilon}^{n-2}}\int_{\mathcal{A}^{\eta}_{r_{\epsilon}}}e_{\epsilon}(u_{\epsilon})
		&=\lim_{\epsilon\to 0}\frac{r_{\epsilon}^{2-n}}{2\leps }\int_{B_{r_{\epsilon}}^{n-2}}\int_{\partial D_{\eta r_{\epsilon}}}\kappa\lvert\log r_{\epsilon}\rvert d^*\xi_{\epsilon}(\partial_{\theta})\\
		&=\lim_{\epsilon\to 0}\frac{r_{\epsilon}^{2-n}}{2\leps }\kappa \lvert\log r_{\epsilon}\rvert\int_{B_{r_{\epsilon}}^{n-2}}\left(\int_{ D_{\eta r_{\epsilon}}} dd^*\xi_{\epsilon}\right)\\
		&=\lim_{\epsilon\to 0}\frac{r_{\epsilon}^{2-n}}{2\leps }\kappa \lvert\log r_{\epsilon}\rvert\int_{B_{r_{\epsilon}}^{n-2}}\left(\int_{ D_{\eta r_{\epsilon}}} djv_{\epsilon}\right)\\
		&=\lim_{\epsilon \to 0}\frac{r_{\epsilon}^{2-n}\kappa\lvert\log r_{\epsilon}\rvert}{2\leps }\int_{B_{r_\epsilon}^{n-2}}2\pi\deg(v_{\epsilon},\{y\}\times\partial D_{\eta r_{\epsilon}})\,dy\\
		&=\pi\omega_{n-2}\kappa^2 \lim_{\epsilon\to 0}\frac{\log(1/r_{\epsilon})}{\log(1/\epsilon)},
	\end{align*}
	where we used again \cref{deg.all.scales} and, in the third equality, the fact that $djv_\epsilon-dd^*\xi_\epsilon=dh_\epsilon$ is bounded pointwise by $C\leps ^{1/2}$.
	This concludes the proof.
\end{proof}

Next, note that for $p_{\epsilon}$ and $\tau_0$ as in \cref{loc.vort.lem}, using the fact that $y_{\epsilon}\in \mathcal{G}_{\epsilon,\delta_{\epsilon}}$ and the inclusion $B^{n-2}_{\epsilon^{\tau_0}}(y_\epsilon)\times D^2_{\eta\epsilon^{\tau_0}}(z_\epsilon)\subset B^n_{(1+\eta)\epsilon^{\tau_0}}(p_\epsilon)$, we have for any given $\eta\in (0,1]$ that
\begin{align*}
	\lim_{\epsilon\to 0}\frac{E_\epsilon(u_\epsilon;\mathcal{A}^{\eta}_{\epsilon^{\tau_0}}(p_{\epsilon}))}{\leps(\epsilon^{\tau_0})^{n-2}}
	&\ge\lim_{\epsilon\to 0}\frac{E_\epsilon(u_\epsilon;B_{\epsilon^{\tau_0}}^{n-2}(y_{\epsilon})\times D_1^2(0))}{\leps(\epsilon^{\tau_0})^{n-2}}-\limsup_{\epsilon\to 0}\frac{E_\epsilon(u_\epsilon;B_{(1+\eta)\epsilon^{\tau_0}}(p_{\epsilon}))}{\leps(\epsilon^{\tau_0})^{n-2}}\\
	&=\liminf_{\epsilon\to 0}\left(\frac{E_\epsilon(u_\epsilon;B_1(p_{\epsilon}))}{\leps }-\frac{E_\epsilon(u_\epsilon;B_{(1+\eta)\epsilon^{\tau_0}}(p_{\epsilon}))}{\leps(\epsilon^{\tau_0})^{n-2}}\right)\\
	&\geq \liminf_{\epsilon\to 0}\left(\frac{E_\epsilon(u_\epsilon;B_1(p_{\epsilon}))}{\leps }-\frac{E_\epsilon(u_\epsilon;B_{(1+\eta)\epsilon^{\tau_0}}(p_{\epsilon}))}{\leps[(1+\eta)\epsilon^{\tau_0}]^{n-2}}\right)\\
	&\quad-C\eta\limsup_{\epsilon\to 0}\frac{E_\epsilon(u_\epsilon;B_{(1+\eta)\epsilon^{\tau_0}}(p_{\epsilon}))}{\leps[(1+\eta)\epsilon^{\tau_0}]^{n-2}},
\end{align*}
where the equality comes from the fact that $\lim_{\epsilon\to 0}\frac{E_\epsilon(u_\epsilon;B_{\epsilon^{\tau_0}}^{n-2}(y_{\epsilon})\times D_1^2(0))}{\leps(\epsilon^{\tau_0})^{n-2}}=\pi\omega_{n-2}\theta$ (as $y_\epsilon\in\mathcal{G}_{\epsilon,\delta_\epsilon}$),
which in turn equals $\lim_{\epsilon\to 0}\frac{E_\epsilon(u_\epsilon;B_1(p_{\epsilon}))}{\leps }$
since energy concentrates on a plane with multiplicity $\theta$.
In particular, combining \cref{ann.en.lem} with the monotonicity formula \cref{monotonicity} (integrated between radii $(1+\eta)\epsilon^{\tau_0}$ and $1$), we see that
\begin{equation}\label{big.mono.bd}
	\limsup_{\epsilon\to 0}\frac{1}{\log(1/\epsilon^{\tau_0})}\int_{(1+\eta)\epsilon^{\tau_0}}^1\frac{1}{r^{n-1}}\int_{B_r(p_{\epsilon})}\frac{2W(u_{\epsilon})}{\epsilon^2}
	\leq\pi\omega_{n-2}\kappa^2+C\eta.
\end{equation}
As an easy consequence, we have the following proposition.

\begin{proposition}\label{w.upper.bd} For $p_{\epsilon}$ and $\tau_0$ as in \cref{loc.vort.lem}, given $\delta\in(0,1)$, there exists $r_{\epsilon}\in (\epsilon^{\tau_0},\epsilon^{\delta\tau_0})$ such that
	\begin{equation}\label{w.pre.ubds}
		\limsup_{\epsilon\to 0}\frac{1}{(r_{\epsilon}/2)^{n-2}}\int_{B_{r_{\epsilon}/2}(p_{\epsilon})}\frac{2W(u_{\epsilon})}{\epsilon^2}
		\leq \frac{\pi\omega_{n-2}\kappa^2}{1-\delta}.
	\end{equation}
	In particular, we can conclude that $|\kappa|=1$, and
	\begin{equation}\label{w.pi.bd}
		\limsup_{\epsilon\to 0}\frac{1}{(r_{\epsilon}/2)^{n-2}}\int_{B_{r_{\epsilon}/2}(p_{\epsilon})}\frac{2W(u_{\epsilon})}{\epsilon^2}
		\leq\frac{\pi\omega_{n-2}}{1-\delta}.
	\end{equation}
\end{proposition}

\begin{proof}
	Writing 
	$$\iota_{\epsilon}:=\inf_{r\in (\epsilon^{\tau_0}/2,\epsilon^{\delta\tau_0}/2)}r^{2-n}\int_{B_r(p_{\epsilon})}\frac{2W(u_{\epsilon})}{\epsilon^2},$$
	we see that the existence of a sequence $r_{\epsilon}$ satisfying \cref{w.pre.ubds} is equivalent to the statement that $\limsup_{\epsilon\to 0}\iota_{\epsilon}\leq \frac{\pi\omega_{n-2}\kappa^2}{1-\delta}$. By \cref{big.mono.bd} we have, for every $\eta>0$,
	\begin{align*}
		\pi\omega_{n-2}\kappa^2+C\eta
		&\geq\limsup_{\epsilon\to 0}\frac{1}{\log(1/\epsilon^{\tau_0})}\int_{(1+\eta)\epsilon^{\tau_0}}^{\epsilon^{\delta\tau_0}/2}\frac{\iota_{\epsilon}}{r}\,dr\\
		&=\limsup_{\epsilon\to 0}\frac{\iota_{\epsilon}\log(\epsilon^{\delta\tau_0}/2(1+\eta)\epsilon^{\tau_0})}{\log(1/\epsilon^{\tau_0})}\\
		&=(1-\delta)\limsup_{\epsilon\to 0}\iota_{\epsilon}.
	\end{align*}
	Since $\eta>0$ was arbitrary, \cref{w.pre.ubds} clearly follows.
	
	As discussed earlier in the section, since $\theta<2\theta_{min}\leq 2$, we know that $|\kappa|\leq 1$, so to complete the proof of the proposition, we simply need to demonstrate that $|\kappa|\neq 0$. And this is straightforward: in view of the bound \cref{w.pre.ubds}, if $\kappa=0$, it would follow that
	$$\frac{1}{(r_{\epsilon}/2)^{n-2}}\int_{B_{r_{\epsilon}/2}(p_{\epsilon})}\frac{W(u_\epsilon)}{\epsilon^2}\to 0$$
	as $\epsilon\to 0$; but since $p_{\epsilon}\in\mathcal{G}_{\epsilon,\delta_{\epsilon}}$, this would contradict \cref{w.quant.lem} (which applies by \cref{close.to.vort}). Hence, we must have $\kappa=\pm 1$, as claimed.
\end{proof}

\begin{remark}\label{beta0}
	This bound gives another proof that, in \cref{asymp}, the support of $|V|$ is characterized as the limit of the zero sets $u_\epsilon^{-1}\{0\}$ (which can also be deduced from Proposition \ref{pre.ener.id}).
	To check this, it is enough to show that the energy cannot concentrate when $u_\epsilon\neq 0$ everywhere on $B_2$. And indeed, if this happened, we could define a minimal density $\theta_{min}'>0$, among all densities $\theta$ arising as in \cref{dens.gap} with the additional constraint that $u_\epsilon\neq 0$ on $B_2$. Repeating the previous arguments, since now the degree $\kappa=0$, we would reach a contradiction to \cref{w.quant.lem}.
\end{remark}

In the next section, we show that a sequence of Ginzburg--Landau solutions on $B_1^n(0)$ with energy concentrating along the $(n-2)$-plane $P$ must have energy $\approx \pi\omega_{n-2}\leps $ as $\epsilon\to 0$, provided that an additional assumption such as \cref{w.pi.bd} holds with $r_\epsilon=1$. In particular, by combining \cref{w.upper.bd} with \cref{unit.w.cor} of the next section, we can complete the proof of \cref{low.dens.prop} as follows.

\begin{proof}[Proof of \cref{low.dens.prop}] Let $p_{\epsilon}=(y_{\epsilon},z_{\epsilon})$ and $\tau_0$ be as in \cref{loc.vort.lem}. For a given $\delta\in(0,1)$ to be specified below, let $r_{\epsilon}$ be the sequence of scales satisfying \cref{w.pi.bd}, whose existence is guaranteed by \cref{w.upper.bd}.
	Passing to a subsequence so that $\frac{\log(1/r_{\epsilon})}{\log(1/\epsilon)}$ converges, let
	$$\tau:=\lim_{\epsilon\to 0}\frac{\log(1/r_{\epsilon})}{\log(1/\epsilon)},$$
	and note that $\tau\in(0,\tau_0]$ since $r_\epsilon\in(\epsilon^{\tau_0},\epsilon^{\delta\tau_0})$.
	
	Fixing an arbitrary $\gamma>0$ and recalling \cref{close.to.vort}, we then see that $B_{2r_{\epsilon}}(p_{\epsilon})$ satisfies the hypotheses of \cref{unit.w.cor} in the next section for $\epsilon$ sufficiently small, provided that we choose $\delta$ such that $\frac{\pi\omega_{n-2}}{1-\delta}<\pi\omega_{n-2}(1+\delta_2)$ (with $\delta_2$ as in \cref{unit.w.cor}), and provided that
	\begin{equation}\label{sing.vort.conc}
		\lim_{\epsilon\to 0}r_{\epsilon}^{2-n}\int_{B_{2r_{\epsilon}}(p_{\epsilon})\setminus B_{\delta_2r_{\epsilon}}(P+p_{\epsilon})}\frac{e_{\epsilon}(u_{\epsilon})}{\leps }=0
	\end{equation}
	holds. To check \cref{sing.vort.conc}, observe that
	\begin{align*}
		B_{2r_{\epsilon}}(p_{\epsilon})\setminus B_{\delta_2r_{\epsilon}}(P+p_{\epsilon})
		&\subseteq B_{2r_{\epsilon}}^{n-2}(y_{\epsilon})\times [D_{2r_{\epsilon}}^2(z_{\epsilon})\setminus D_{\delta_2r_{\epsilon}}^2(z_{\epsilon})]\\
		&=\mathcal{A}_{2r_{\epsilon}}^{\delta_2/2}(p_{\epsilon})\setminus \mathcal{A}_{2r_{\epsilon}}^1(p_{\epsilon}),
	\end{align*}
	and using \cref{ann.en.lem}, it follows that
	\begin{align*}
		\lim_{\epsilon\to 0}r_{\epsilon}^{2-n}\int_{B_{2r_{\epsilon}}(p_{\epsilon})\setminus B_{\delta_2r_{\epsilon}}(P+p_{\epsilon})}\frac{e_{\epsilon}(u_{\epsilon})}{\leps }
		&\leq \lim_{\epsilon\to 0}\left(\frac{E_\epsilon(u_\epsilon;\mathcal{A}_{2r_{\epsilon}}^{\delta_2/2}(p_\epsilon))}{\leps r_{\epsilon}^{n-2}}
		-\frac{E_\epsilon(u_\epsilon;\mathcal{A}_{2r_{\epsilon}}^{1}(p_\epsilon))}{\leps r_{\epsilon}^{n-2}}\right)\\
		&=2^{n-2}\pi\omega_{n-2}\kappa^2\tau-2^{n-2}\pi\omega_{n-2}\kappa^2\tau \\
		&=0,
	\end{align*}
	as desired.
	
	In particular, for $\epsilon$ sufficiently small, we can now apply \cref{unit.w.cor} on the ball $B_{2r_{\epsilon}}(p_{\epsilon})$ to conclude that
	$$
	\left|\frac{r_{\epsilon}^{2-n}}{\log(r_{\epsilon}/\epsilon)}\int_{B_{r_{\epsilon}}^{n-2}(y_{\epsilon})\times D_{r_{\epsilon}}(z_{\epsilon})}e_{\epsilon}(u_{\epsilon})-\pi\omega_{n-2}\right|<\gamma,$$
	which implies
	\begin{equation}\label{off.ann.en}
		\left|\frac{r_{\epsilon}^{2-n}}{\leps }\int_{B_{r_{\epsilon}}^{n-2}(y_{\epsilon})\times D_{r_{\epsilon}}(z_{\epsilon})}e_{\epsilon}(u_{\epsilon})-\pi\omega_{n-2}\frac{\log(r_{\epsilon}/\epsilon)}{\leps }\right|<\gamma.
	\end{equation}
	On the other hand, it follows from \cref{ann.en.lem} (together with $|\kappa|=1$) that, for $\epsilon$ sufficiently small,
	$$\left|\frac{r_{\epsilon}^{2-n}}{\leps }\int_{\mathcal{A}_{r_{\epsilon}}^1(p_{\epsilon})}e_{\epsilon}(u_{\epsilon})-\pi\omega_{n-2}\frac{\log(1/r_{\epsilon})}{\log(1/\epsilon)}\right|<\gamma$$
	as well. Since $B_{r_{\epsilon}}^{n-2}(y_{\epsilon})\times D_1(0)=\mathcal{A}_{r_{\epsilon}}^1(p_{\epsilon})\sqcup [B_{r_{\epsilon}}^{n-2}(y_{\epsilon})\times D_{r_{\epsilon}}(z_{\epsilon})]$ and
	$$\log(r_{\epsilon}/\epsilon)+\log(1/r_{\epsilon})=\log(1/\epsilon),$$
	we can combine these estimates to 
	conclude that
	$$\lim_{\epsilon\to 0}\frac{r_{\epsilon}^{2-n}}{\leps }\int_{B_{r_{\epsilon}}^{n-2}(y_{\epsilon})\times D_1(0)}e_{\epsilon}(u_{\epsilon})=\pi\omega_{n-2}.$$
	Since $y_{\epsilon}\in \mathcal{G}_{\epsilon,\delta_{\epsilon}}$, it follows that $\theta=1$, as desired.
\end{proof}

\section{From bounds on $W(u)/\epsilon^2$ to unit density}\label{w.ests.sec}

In this section we show that if, in addition to the hypotheses of \cref{low.dens.prop}, we have the bound
\begin{equation}\label{w.ubd}
	\limsup_{\epsilon\to 0}\int_{B_{1/2}(0)}\frac{2W(u_{\epsilon})}{\epsilon^2}\leq \frac{\pi\omega_{n-2}}{2^{n-2}}(1+\delta),
\end{equation}
for some (explicit) $\delta>0$ small enough,
then the limiting density $\theta=1$.
In other words, we are going to prove \cref{low.dens.prop} with the additional assumption \cref{w.ubd}.
As we saw above, this combines with the analysis of the preceding section to give \cref{low.dens.prop} in full generality, from which \cref{dens.gap} follows.

Since the measures $\frac{W(u_{\epsilon}(x))}{\epsilon^2}\,dx$ converge to an absolutely continuous measure with respect to $\mathcal{H}^{n-2}\mrestr P$
(by \cref{asymp}), where the plane $P=\mathbb{R}^{n-2}\times\{0\}$, the estimate \cref{w.ubd} also gives
$$\limsup_{\epsilon\to 0}\int_{B_{1/2}^{n-2}(0)\times D_1^2(0)}\frac{2W(u_{\epsilon})}{\epsilon^2}\leq \frac{\pi \omega_{n-2}}{2^{n-2}}(1+\delta).$$
In particular, this implies that
\begin{align*}
	\frac{\pi\omega_{n-2}}{2^{n-2}}(1+\delta) &\geq \limsup_{\epsilon\to 0} \int_{\mathcal{G}_{\epsilon,\delta_{\epsilon}}\times D_1^2}\frac{2W(u_{\epsilon})}{\epsilon^2}\\
	&\geq \limsup_{\epsilon\to 0}\left(|\mathcal{G}_{\epsilon,\delta_{\epsilon}}|\cdot\inf_{y\in \mathcal{G}_{\epsilon,\delta_{\epsilon}}}\int_{\{y_{\epsilon}\}\times D_1^2}\frac{2W(u_{\epsilon})}{\epsilon^2}\right),
\end{align*}
which together with \cref{good.full} implies the existence of $y_{\epsilon}\in \mathcal{G}_{\epsilon,\delta_{\epsilon}}$ for which
\begin{equation*}
	\limsup_{\epsilon\to 0}\int_{\{y_{\epsilon}\}\times D_1^2}\frac{2W(u_{\epsilon})}{\epsilon^2}\leq\frac{(\pi\omega_{n-2}/2^{n-2})(1+\delta)}{(\omega_{n-2}/2^{n-2})-(C/K)}.
\end{equation*}
We now fix $K$ large enough (e.g., $K=K(\delta)=\frac{4C}{\delta}\frac{2^{n-2}}{\omega_{n-2}}$), in such a way that the previous estimate becomes
\begin{align}\label{better.slice}
	&\limsup_{\epsilon\to 0}\int_{\{y_{\epsilon}\}\times D_1^2}\frac{2W(u_{\epsilon})}{\epsilon^2}\leq\pi(1+2\delta).
\end{align}

Applying Propositions \ref{cover.claim} and \ref{w.deg.prop}, we deduce that for these $y_{\epsilon}$ such that \eqref{better.slice} holds, the set
$$U_{\delta,\epsilon}=\{x\in P_{y_{\epsilon}}^{\perp} : |u_{\epsilon}(x)|<1-\delta\}$$
is contained in a disjoint collection of disks $D_{C\epsilon}(p_1^{\epsilon}),\ldots,D_{C\epsilon}(p_k^{\epsilon})$ such that the degrees $\kappa_j^{\epsilon}$ of $u_{\epsilon}$ around $\partial D_{C\epsilon}(p_j^{\epsilon})$ satisfy
$$\limsup_{\epsilon\to 0}\frac{\pi}{2}(1-5\delta){\textstyle\sum_{j=1}^k}|\kappa_j^{\epsilon}|\leq \pi(1+2\delta).$$
In particular, since $\kappa_j^{\epsilon}\in \mathbb{Z}$, taking $\delta$ small enough it follows that
$${\textstyle\sum_{j=1}^k}|\kappa_j^{\epsilon}|\leq 2$$
for $\epsilon$ sufficiently small.

On the other hand, by \cref{w.upper.bd}, we know that
$$|\kappa|=|{\textstyle\sum_{j=1}^k}\kappa_j^\epsilon|=1,$$
so we see by parity that the case ${\textstyle\sum_{j=1}^k}|\kappa_j^{\epsilon}|\in\{0,2\}$ is impossible, so we must have
$${\textstyle\sum_{j=1}^k}|\kappa_j^{\epsilon}|=1.$$

In other words, up to relabeling $p_1^{\epsilon},\ldots, p_k^{\epsilon}$ and possibly replacing $u_{\epsilon}$ with the conjugate solution $\bar{u}_{\epsilon}$, for $\epsilon$ sufficiently small, we must have
\begin{equation}\label{mult.1.vorts}
	\kappa_1^{\epsilon}=1,\ \kappa_2^{\epsilon}=\cdots=\kappa_k^{\epsilon}=0.
\end{equation}

Finally, combining this with Proposition \ref{pre.ener.id} immediately gives the following conclusion.

\begin{lemma}
	Suppose that the hypotheses of \cref{dens.gap} and the potential bound \cref{w.ubd} are both satisfied. Then the limiting density $\theta=1$.
\end{lemma}

By a simple contradiction and scaling argument, we can recast the result in the following `quantitative' form.

\begin{proposition}\label{unit.w.cor}
	For any $\gamma>0$, there exists $\delta_2(\gamma)>0$ such that if $u_{\epsilon}$ solves the Ginzburg--Landau equation on a ball $B_{2r}(x)$ (where $x=(y,z)\in \mathbb{R}^{n-2}\times \mathbb{R}^2$) with $\epsilon\le\delta_2 r$, and satisfies
	$$(r/2)^{2-n}\int_{B_{r/2}(x)}\frac{2W(u_{\epsilon})}{\epsilon^2}\leq \pi\omega_{n-2}+\delta_2,$$
	$$\mathcal{V}(u_{\epsilon})\cap B_r(x)\neq \varnothing,$$
	and (to ensure that all energy concentrates along the $(n-2)$-plane $P+x$)
	$$r^{2-n}\int_{B_{2r}(x)\setminus B_{\delta_2 r}(P+x)}e_{\epsilon}(u_{\epsilon})\le\delta_2\log(r/\epsilon),$$
	then
	$$\left|\frac{r^{2-n}}{\log(r/\epsilon)}\int_{B_r^{n-2}(y)\times D_r^2(z)}e_{\epsilon}(u_{\epsilon})-\pi\omega_{n-2}\right|<\gamma.$$
\end{proposition}

\section{Solutions concentrating with prescribed density $\theta\in \{1\}\cup [2,\infty)$}

In this section, we explain how to use the entire solutions of the Ginzburg--Landau equations constructed in \cite{DDMR} to prove \cref{prescribe.dens}. More precisely, we prove the following proposition.

\begin{proposition}\label{kap.tau.sols} For each integer $\kappa\geq 2$ and $\tau\in [0,1)$, there exists a family of solutions $(u^{\tau}_{\epsilon})_{\epsilon\in (0,\epsilon_0(\kappa,\tau))}$ in the unit $3$-ball $B^3_1(0)\subset \mathbb{R}^3$ with energy concentrating along the line $P=\{0\}\times \mathbb{R}$, degree $\kappa=\deg(u_{\epsilon},\mz S^1\times \{0\})$
	and limiting energy measure
	$$\lim_{\epsilon\to 0}\frac{e_{\epsilon}(u_{\epsilon})}{\pi\leps }=\theta(\kappa,\tau)\mathcal{H}^1\mrestr P,$$
	where
	$$\theta(\kappa,\tau):=\kappa+\kappa(\kappa-1)\frac{\tau}{1+\tau}\in \left[\kappa,\frac{\kappa(\kappa+1)}{2}\right).$$
\end{proposition} 

It is straightforward to check that \cref{prescribe.dens} follows from \cref{kap.tau.sols}, since we have $\bigcup_{\kappa=2}^{\infty}[\kappa,\frac{\kappa(\kappa+1)}{2})=[2,\infty)$. The solutions  described in \cref{kap.tau.sols} are obtained by rescaling families of entire solutions with $\kappa$ helical vortex filaments constructed in \cite{DDMR}. Namely, we rely on the following result. (In what follows, we make the identifications $\R^2\cong\C$ and $S^1\cong\mathbb{R}/2\pi\mathbb{Z}$.)

\begin{theorem}\label{ddmr.thm}\cite[Theorem~1]{DDMR} For $\kappa\in \{2,3,\ldots\}$ and $\epsilon<\epsilon_0(\kappa)$ sufficiently small, there exists a solution $v_{\epsilon}: \mathbb{R}^2\times S^1\to \mathbb{R}^2$ of the Ginzburg--Landau equations
	\begin{equation}\label{gl.e.5}
		\epsilon^2\Delta v_{\epsilon}=DW(v_{\epsilon})
	\end{equation}
	satisfying
	$$\left|v_{\epsilon}(z,t)-{\textstyle\prod_{j=1}^{\kappa}}w(\epsilon^{-1}[z-f_j^{\epsilon}(t)])\right|
	\leq \frac{C(\kappa)}{\leps },$$
	where $w: \mathbb{R}^2\to \mathbb{R}^2$ is the radially symmetric degree-one solution constructed in \cite{HH}, and $f_j^{\epsilon}: S^1\to \mathbb{R}^2$ satisfies
	$$\lim_{\epsilon\to 0}\left|\sqrt{\leps }f_j^{\epsilon}(t)-\sqrt{n-1}e^{it}e^{2i(j-1)\pi/\kappa}\right|=0.$$
	Moreover, these $v_{\epsilon}$ have the additional symmetry
	\begin{equation}\label{screw.sym}
		v_{\epsilon}(z,t)=e^{\kappa it}\tilde{v}_{\epsilon}(e^{-it}z)
	\end{equation}
	for some map $\tilde{v}_{\epsilon}: \mathbb{R}^2\to \mathbb{R}^2$.
\end{theorem}

By looking closely at the construction of these solutions and keeping track of a few key estimates in \cite{DDMR}, we are able to check that the following estimate holds.

\begin{lemma}\label{enerbd.claim} For every $\tau\in [0,1)$, there exists a constant $C(\kappa)<\infty$ such that the solutions $v_{\epsilon}:\mathbb{R}^2\times S^1\to \mathbb{R}^2$ from \cref{ddmr.thm} satisfy
	\begin{equation}
		\int_{D_{\epsilon^{-\tau}}\times S^1}e_{\epsilon}(v_{\epsilon})\leq C(\kappa)\log(1/\epsilon^{\tau+1})
	\end{equation}
	for $\epsilon<\epsilon_0(\kappa,\tau)$, and, for every $\delta\in(0,1)$, there exists moreover a constant $C(\delta,\kappa,\tau)<\infty$ such that
	\begin{equation}
		\int_{[D_{\epsilon^{-\tau}}\setminus D_{\delta\epsilon^{-\tau}}]\times S^1}e_{\epsilon}(v_{\epsilon})\leq C(\delta,\kappa,\tau)
	\end{equation}
	for $\epsilon<\epsilon_1(\delta,\kappa,\tau)$.
\end{lemma}

We postpone the proof of \cref{enerbd.claim} to the end of the section; next, we show how the results of \cref{ddmr.thm} and \cref{enerbd.claim} can be used to prove \cref{kap.tau.sols}.

\begin{proof}[Proof of \cref{kap.tau.sols}] 
	
	As in \cite{DDMR}, we identify the solutions $v_{\epsilon}:\mathbb{R}^2\times S^1\to \mathbb{R}^2$ given by \cref{ddmr.thm} with solutions on $\mathbb{R}^3$ that are $2\pi$-periodic in the third variable. Under this identification, note that \cref{enerbd.claim} gives
	\begin{align*}
		\int_{B^3_{\epsilon^{-\tau}}}e_{\epsilon}(v_{\epsilon})
		&\leq\int_{D_{\epsilon^{-\tau}}\times [-\epsilon^{-\tau},\epsilon^{-\tau}]}e_{\epsilon}(v_{\epsilon})\\
		&\leq C\epsilon^{-\tau}\int_{D_{\epsilon^{-\tau}}\times S^1}e_{\epsilon}(v_{\epsilon})\\
		&\leq  C(\kappa)\epsilon^{-\tau}\log(\epsilon^{-\tau-1}),
	\end{align*}
	and similarly
	\begin{align*}
		\int_{B^3_{\epsilon^{-\tau}}(0)\setminus B_{\delta\epsilon^{-\tau}}(P)}e_{\epsilon}(v_{\epsilon})
		&\leq C\epsilon^{-\tau}\int_{[D_{\epsilon^{-\tau}}\setminus D_{\delta\epsilon^{-\tau}}]\times S^1}e_{\epsilon}(v_{\epsilon})\\
		&\leq C(\delta,\kappa,\tau)\epsilon^{-\tau}.
	\end{align*}
	
	It is then straightforward to see that, for $\tilde{\epsilon}=\epsilon^{1+\tau}$, the rescaled maps
	$$u_{\tilde{\epsilon}}^{\tau}: \mathbb{R}^3\to \mathbb{R}^2,\quad u^{\tau}_{\tilde{\epsilon}}(x):=v_{\epsilon}(x/\epsilon^{\tau})$$
	solve the $\tilde{\epsilon}$-Ginzburg--Landau equations
	$$\tilde{\epsilon}^2\Delta u_{\tilde{\epsilon}}^{\tau}=DW(u_{\tilde{\epsilon}}^{\tau})$$
	on $\mathbb{R}^3$, and satisfy
	\begin{align*}
		\int_{B_1^3(0)}e_{\tilde{\epsilon}}(u^{\tau}_{\tilde{\epsilon}})&=\epsilon^{\tau}\int_{B_{\epsilon^{-\tau}}^3(0)}e_{\epsilon}(v_{\epsilon})
		\leq C(\kappa,\tau)\log(1/\tilde{\epsilon})
	\end{align*}
	and, for any $\delta>0$,
	\begin{align*}
		\int_{B_1^3(0)\setminus B_{\delta}(P)}e_{\tilde{\epsilon}}(u_{\tilde{\epsilon}}^{\tau})&=\epsilon^{\tau}\int_{B^3_{\epsilon^{-\tau}}(0)\setminus B_{\delta\epsilon^{-\tau}}(P)}e_{\epsilon}(v_{\epsilon})
		\leq C(\delta,\kappa,\tau).
	\end{align*}
	In particular, it follows that the maps
	$$u^{\tau}_{\tilde{\epsilon}}:B_1^3(0)\to \mathbb{R}^2,\quad\text{for }\tilde\epsilon\in(0,\epsilon_0(\kappa)^{1+\tau}),$$
	give a family of solutions to the Ginzburg--Landau equations on $B_1^3(0)$ (or similarly, any fixed compact subset of $\mathbb{R}^3$) with energy of order $\log(1/\tilde{\epsilon})$ concentrating along $P$ as $\tilde{\epsilon}\to 0$; i.e.\ (up to subsequences),
	$$\lim_{\tilde{\epsilon}\to 0}\frac{e_{\tilde{\epsilon}}(u_{\tilde{\epsilon}}^{\tau})}{\pi \lteps }\,dx = \theta \mathcal{H}^1 \mrestr P$$
	for some $\theta>0$.
	
	To compute $\theta$, we appeal to \cref{pre.ener.id}, together with \cref{ddmr.thm}. Fixing a small (but arbitrary) $\delta\in (0,\mz)$, consider $\delta_{\tilde{\epsilon}}\to 0$ and $t_{\tilde{\epsilon}}\in \mathcal{G}_{\tilde{\epsilon},\delta_{\tilde{\epsilon}}}\subseteq (-\mz,\mz)$ a family of $\delta_{\tilde{\epsilon}}$-good slices for $u_{\tilde{\epsilon}}^{\tau}$ as in the preceding sections. By virtue of the symmetry \cref{screw.sym} of $v_{\epsilon}$, note that we can simply take $t_{\tilde{\epsilon}}=0$.
	Following the analysis of the preceding section, consider the set
	$$U_{\delta,\tilde{\epsilon}}:=\{z\in \mathbb{R}^2 : |u_{\tilde{\epsilon}}^{\tau}(z,0)|<1-\delta\}.$$
	By \cref{ddmr.thm}, we see that if $|v_{\epsilon}(z,0)|<1-\delta$, then
	$$\left|{\textstyle\prod_{j=1}^{\kappa}} w(\epsilon^{-1}[z-f_j^{\epsilon}(0)])\right|\leq \frac{C(\kappa)}{\leps }+1-\delta\leq 1-\frac{\delta}{2}$$
	for $\epsilon$ sufficiently small, and since the model single-vortex solution $w$ satisfies $|w(z)|\to 1$ as $|z|\to\infty$, it follows that
	$$z\in \bigcup_{j=1}^{\kappa}B_{C(\delta)\epsilon}(f_j^{\epsilon}(0)).$$
	Moreover, note that for $1\leq j<l\leq \kappa$, \cref{ddmr.thm} gives
	\begin{align*}
		|f_j^{\epsilon}(0)-f_l^{\epsilon}(0)|&\geq \frac{\sqrt{n-1}}{\sqrt{\leps }}|e^{2i(j-1)\pi/\kappa}-e^{2i(l-1)\pi/\kappa}|+o(\leps ^{-1/2})\\
		&\geq \frac{c(\kappa)}{\sqrt{\leps }}
	\end{align*}
	for $\epsilon$ sufficiently small and $c(\kappa)>0$, and a similar upper bound also holds. In particular, it follows that the balls $B_{C(\delta)\epsilon}(f_j^{\epsilon}(0))$ are mutually disjoint for $\epsilon$ sufficiently small, and it follows from the $C^0$ closeness
	$$\left|v_{\epsilon}(z,0)-{\textstyle\prod_{j=1}^k}w(\epsilon^{-1}[z-f_j^{\epsilon}(t)])\right|\leq \frac{C(\kappa)}{\leps }$$
	that 
	$$\deg(v_{\epsilon}, \partial D_{C(\delta)\epsilon}(f_j^{\epsilon}(0)))=1.$$
	
	In particular, for our rescaled solutions $u_{\tilde{\epsilon}}^{\tau}(z,0)=v_{\epsilon}(z/\epsilon^{\tau},0),$ writing
	$$p_j^{\tilde{\epsilon}}:=\epsilon^{\tau}f_j^{\epsilon}(0),$$
	it follows that
	\begin{align*}
		U_{\delta,\tilde{\epsilon}}&=\{(\epsilon^{\tau}z,0) : |v_{\epsilon}(z,0)|<1-\delta\}
		\subseteq \bigcup_{j=1}^{\kappa}B_{C(\delta)\epsilon^{1+\tau}}(\epsilon^{\tau}f_j^{\epsilon}(0))
		=\bigcup_{j=1}^{\kappa}B_{C(\delta)\tilde{\epsilon}}(p_j^{\tilde{\epsilon}}),
	\end{align*}
	where the balls $B_{C(\delta)\tilde{\epsilon}}(p_j^{\tilde{\epsilon}})$ are mutually disjoint, $u_{\tilde{\epsilon}}^{\tau}$ has degree 
	$$\kappa_j^{\tilde{\epsilon}}:=\deg(u_{\tilde{\epsilon}}^{\tau}, \partial D_{C(\delta)\tilde{\epsilon}}(\tilde p_j^{\epsilon}))=1,$$
	and 
	$$\frac{c(\delta,\kappa)\tilde{\epsilon}^{\tau/(1+\tau)}}{\sqrt{\leps }}=\frac{c(\delta,\kappa)\epsilon^{\tau}}{\sqrt{\leps }}
	\leq |\tilde p_j^{\epsilon}-\tilde p_l^{\epsilon}|\leq \frac{C(\delta,\kappa)\epsilon^{\tau}}{\sqrt{\leps }}
	=\frac{C(\delta,\kappa)\tilde{\epsilon}^{\tau/(1+\tau)}}{\sqrt{\leps }}$$
	for $1\leq j<l\leq \kappa$. Thus, applying \cref{pre.ener.id}, we deduce that
	\begin{align*}
		\theta&=\lim_{\tilde{\epsilon}\to 0}\left({\textstyle\sum_{j=1}^{\kappa}}1+2{\textstyle\sum_{j<l}}1\cdot 1\cdot\frac{\lvert\log |p_j^{\tilde{\epsilon}}-p_l^{\tilde{\epsilon}}|\rvert}{\lteps }\right)\\
		&=\kappa+\kappa(\kappa-1)\frac{\tau}{1+\tau},
	\end{align*}
	completing the proof of \cref{kap.tau.sols}.
\end{proof}

It remains now to prove \cref{enerbd.claim}, verifying that natural energy growth conditions hold for the solutions constructed in \cite{DDMR}.

\begin{proof}[Proof of \cref{enerbd.claim}]
	For simplicity, we specialize to the case $\kappa=2$, for which the construction in \cite{DDMR} is carried out in detail. In this case, the solutions $v_{\epsilon}: \mathbb{R}^2\times S^1\to \mathbb{R}^2$ of \cref{ddmr.thm} have the form
	$$v_{\epsilon}(z,t)=e^{2 i t}V_{\epsilon}(e^{-it}z/\epsilon),$$
	for a map $V_{\epsilon}$ of the form
	$$V_{\epsilon}(z):=w(z-\tilde{d}_{\epsilon})w(z+\tilde{d}_{\epsilon})[\eta\cdot(1+i\psi_{\epsilon})+(1-\eta)\cdot e^{i\psi_{\epsilon}}]$$
	(see \cite[eq.~(3.2)]{DDMR}), where $w: \mathbb{R}^2\to \mathbb{R}^2$ is the radially symmetric solution of $$\Delta v+(1-|v|^2)v=0$$ with degree one constructed in \cite{HH}, $\tilde{d}_{\epsilon}\in \mathbb{R}^2$ are points in the plane with 
	$$|\tilde{d}_{\epsilon}|\leq \frac{C}{\epsilon\sqrt{\leps }},$$
	$\eta$ is a cutoff function of the form
	$$\eta(z)=\eta_1(|z-\tilde{d}_{\epsilon}|)+\eta_1(|z+\tilde{d}_{\epsilon}|),$$
	with $\eta_1(t)=1$ for $t\leq 1$ and $\eta_1(t)=0$ for $t\geq 2$, and $\psi_{\epsilon}: \mathbb{R}^2\to \mathbb{C}$ is an unknown function whose implicit construction (with estimates) is the content of the proof of \cref{ddmr.thm}. 
	
	It follows from \cite[Proposition~6.1]{DDMR} (see also \cite[p.~18]{DDMR} for the definition of $\|\psi\|_*$) that the real part $\Re(\psi_{\epsilon})$ of $\psi_{\epsilon}$ satisfies
	\begin{equation}\label{re.psi.est}
		|\Re(\psi_{\epsilon})|\leq \frac{C}{\leps }\quad\text{where }\min\{|z\pm \tilde{d}_{\epsilon}|\}>2
	\end{equation}
	and, for any fixed $\sigma\in (0,1]$, the imaginary part satisfies\footnote{Note that the first occurrence of $\epsilon^{\sigma-2}$ in the definition of $\|\psi_2\|_{2,*}$ from \cite[p.~18]{DDMR} should read $\epsilon^{2-\sigma}$.}
	\begin{equation}\label{im.psi.est}
		|\Im(\psi_{\epsilon})|\leq C(\sigma)(|z-\tilde{d}_{\epsilon}|^{\sigma-2}+|z+\tilde{d}_{\epsilon}|^{\sigma-2}+\epsilon^{2-\sigma})\quad\text{where }\min\{|z\pm\tilde{d}_{\epsilon}|\}>2.
	\end{equation}
	In particular, where $\min\{|z\pm\tilde{d}_{\epsilon}|\}>2$, we have
	\begin{align*}
		|1-|V_{\epsilon}(z)||&=|1-|w(z-\tilde{d}_{\epsilon})||w(z+\tilde{d}_{\epsilon})|e^{-\Im(\psi_{\epsilon})}|\\
		&\leq |1-|w(z-\tilde{d}_{\epsilon})||+|1-|w(z+\tilde{d}_{\epsilon})||+|1-e^{-\Im(\psi_{\epsilon})}|\\
		&\leq C(\sigma)(|z-\tilde{d}_{\epsilon}|^{\sigma-2}+|z+\tilde{d}_{\epsilon}|^{\sigma-2}+\epsilon^{2-\sigma}),
	\end{align*}
	where we used that the model single-vortex solution $w(z)$ satisfies $0\leq 1-|w(z)|\leq \frac{C}{|z|^2}$ (cf.\ \cite[Lemma~7.1]{DDMR}). In particular, scaling back down to the solutions $v_{\epsilon}=e^{2it}V_{\epsilon}(e^{-it}z/\epsilon)$, it follows that
	\begin{align*}
		|1-|v_{\epsilon}(z,t)||&\leq C(\sigma)\epsilon^{2-\sigma}(|z-e^{it}\epsilon\tilde{d}_{\epsilon}|^{\sigma-2}+|z+e^{it}\epsilon\tilde{d}_{\epsilon}|^{\sigma-2}+1)
	\end{align*}
	where $|z\pm e^{it}\epsilon\tilde{d}_{\epsilon}|>2\epsilon$.
	Also, from the bound \cref{bound.general.critical} in the appendix (and a trivial rescaling), we have
	$|u(x)|\le 1+C\epsilon^2R^{-2}$ on $B_{R/2}(0)$, for all $R\ge 1$, which implies that $|v_\epsilon|\le 1$ everywhere.
	
	In particular, by the preceding estimates, setting
	$$\rho_{\epsilon}(z,t):=\min\{|z-e^{it}\epsilon\tilde{d}_{\epsilon}|,|z+e^{it}\epsilon\tilde{d}_{\epsilon}|\},$$ we see that for $\sigma,\tau\in(0,1)$
	\begin{align*}
		\int_{D_{2\epsilon^{-\tau}}\times S^1}\frac{(1-|v_{\epsilon}|^2)^2}{\epsilon^2}
		&\leq \int_{\{\rho_{\epsilon}\leq 2\epsilon\}}\frac{(1-|v_{\epsilon}|^2)^2}{\epsilon^2}\\
		&\quad+\int_{\{4\epsilon^{-\tau}\geq\rho_{\epsilon}>2\epsilon\}}\frac{(1-|v_{\epsilon}|^2)^2}{\epsilon^2}\\
		&\leq C\epsilon^2\cdot\frac{C}{\epsilon^2}+C(\sigma)\int_{\{4\epsilon^{-\tau}\geq\rho_{\epsilon}>2\epsilon\}}\frac{1}{\epsilon^2}[\epsilon^{2-\sigma}(\rho_{\epsilon}^{\sigma-2}+1)]^2\\
		&\leq C+\frac{C(\sigma)}{\epsilon^2}\cdot \epsilon^{2(2-\sigma)}(\epsilon^{2\sigma-2}+\epsilon^{-2\tau})\\
		&\leq C(\sigma)(1+\epsilon^{2(1-\sigma-\tau)}),
	\end{align*}
	where we used the coarea formula to bound $\int_{\{4\epsilon^{-\tau}\geq\rho_{\epsilon}>2\epsilon\}}\rho_{\epsilon}^{2(\sigma-2)}\le C(\sigma)\epsilon^{2\sigma-2}$
	(as each level set $\{\rho_\epsilon=r\}$ has length at most $Cr$);
	hence, taking $\sigma=1-\tau$ gives
	\begin{equation}\label{v.w.bd}
		\int_{D_{2\epsilon^{-\tau}}\times S^1}\frac{(1-|v_{\epsilon}|^2)^2}{\epsilon^2}\leq C(\tau)
	\end{equation}
	for all $\tau\in(0,1)$, and thus also for all $\tau\in[0,1)$.
	
	With the bound \cref{v.w.bd} in place, it follows from \cref{bos1} (after a suitable rescaling) that 
	\begin{equation}\label{d.abs.delpino}
		\int_{D_{\epsilon^{-\tau}}\times S^1}|d|v_{\epsilon}||^2\leq C(\tau)
	\end{equation}
	as well, so to obtain the desired energy estimates for $v_{\epsilon}$, it remains to estimate the contribution from
	$$|dv_{\epsilon}|^2-|d|v_{\epsilon}||^2=|v_{\epsilon}|^{-2}|jv_{\epsilon}|^2,$$
	recalling that
	$$jv_{\epsilon}=v_{\epsilon}^*(r^2\,d\theta)=v_{\epsilon}^1\,dv_{\epsilon}^2-v_{\epsilon}^2\,dv_{\epsilon}^1.$$
	To start, observe that on $\{\rho_{\epsilon}\leq 4\epsilon\}$, we have
	\begin{equation}
		\int_{\{\rho_{\epsilon}\leq 4\epsilon\}} |dv_{\epsilon}|^2\leq C\epsilon^2\cdot \frac{C}{\epsilon^2}\leq C,
	\end{equation}
	so we only need to estimate the energy contribution from the region
	$$A:=[D_{\epsilon^{-\tau}}\times S^1]\cap\{ \rho_{\epsilon}\geq 4\epsilon\}.$$
	To this end, for $R\in [2\epsilon,2\epsilon^{-\tau}]$, consider the annular regions
	$$\Omega_R:=\{R < \rho_{\epsilon}(z,t) < 5R\}$$
	and
	$$\Omega_R':=\{2R\leq \rho_{\epsilon}(z,t)\leq 4R\},$$
	so that, for $\epsilon<\epsilon_0(\tau)$ sufficiently small (since $|e^{it}\epsilon\tilde{d}_{\epsilon}|\leq C/\sqrt{\leps }\le\epsilon^{-\tau}$), we have
	\begin{equation}\label{a.subset}
		A\subseteq \bigcup_{j=1}^{J_{\epsilon,\tau}}\Omega_{2^j\epsilon}'
	\end{equation}
	where $J_{\epsilon,\tau}:=\lceil \log(\epsilon^{-\tau-1})/\log 2\rceil$ and, for $\epsilon<\epsilon_1(\delta,\tau)$ (since $|e^{it}\epsilon\tilde{d}_{\epsilon}|\leq C/\sqrt{\leps }\le\mz\delta\epsilon^{-\tau}$),
	\begin{equation}\label{smallann.subset}
		[D_{\epsilon^{-\tau}}\setminus D_{\delta\epsilon^{-\tau}}]\times S^1\subseteq \bigcup_{j=I_{\epsilon,\tau,\delta}}^{J_{\epsilon,\tau}}\Omega_{2^j\epsilon}'
	\end{equation}
	where $I_{\epsilon,\tau,\delta}:=\lfloor \log(\delta\epsilon^{-\tau-1})/\log(2)\rfloor-3$. Now, given $R\in [2\epsilon,2\epsilon^{-\tau}],$ let $\chi_R$ be a cutoff function such that
	$$0\leq \chi_R\in C_c^{\infty}(\Omega_R),\quad\chi_R\equiv 1\text{ on }\Omega_R',\quad|d\chi_R|\leq \frac{C}{R}.$$
	
	Next, observe that where $\rho_{\epsilon}\geq 2\epsilon$, $v_{\epsilon}$ has the form
	$$v_{\epsilon}(z,t)=e^{2it}w(e^{-it}\epsilon^{-1}z-\tilde{d}_{\epsilon})w(e^{-it}\epsilon^{-1}z+\tilde{d}_{\epsilon})e^{i\psi_{\epsilon}(e^{-it}\epsilon^{-1}z)},$$
	and since the model single-vortex solution $w$ satisfies $\frac{w(z)}{|w(z)|}=\frac{z}{|z|}$, it follows that
	$$\frac{v_{\epsilon}(z,t)}{|v_{\epsilon}|(z,t)}=\frac{z-e^{it}\epsilon\tilde{d}_{\epsilon}}{|z-e^{it}\epsilon\tilde{d}_{\epsilon}|}\cdot \frac{z+e^{it}\epsilon\tilde{d}_{\epsilon}}{|z+e^{it}\epsilon\tilde{d}_{\epsilon}|}\cdot e^{i\varphi_{\epsilon}},$$
	where we set
	$$\varphi_{\epsilon}(z,t):=\Re(\psi_{\epsilon}(e^{-it}\epsilon^{-1}z)).$$
	It is straightforward to check that
	$$\left\lvert d\left(\frac{z\pm e^{it}\epsilon\tilde{d}_{\epsilon}}{|z\pm e^{it}\epsilon\tilde{d}_{\epsilon}|}\right)\right\rvert\leq \frac{C}{|z\pm e^{it\epsilon}\tilde{d}_{\epsilon}|},$$
	and as a consequence,
	\begin{align*}
		|j(v_{\epsilon}/|v_{\epsilon}|)-d\varphi_{\epsilon}|&=|j(e^{-i\varphi_{\epsilon}}v_{\epsilon}/|v_{\epsilon}|)|
		\leq  \frac{C}{\rho_{\epsilon}}
	\end{align*}
	where $\rho_{\epsilon}\geq 2\epsilon$. Moreover, recall that, since $v_{\epsilon}$ solves the Ginzburg--Landau equations, we have as always $d^*jv_{\epsilon}=0$; as a consequence, for any $R\in [2\epsilon,2\epsilon^{-\tau}]$, we see that
	$$\int \chi_R^2\langle jv_{\epsilon},d\varphi_{\epsilon}\rangle=-\int \varphi_{\epsilon}\langle jv_{\epsilon},d(\chi_R^2)\rangle,$$
	and therefore
	\begin{align*}
		\int \chi_R^2|v_{\epsilon}|^{-2}|jv_{\epsilon}|^2&=\int \chi_R^2\langle jv_{\epsilon},j(v_{\epsilon}/|v_{\epsilon}|)\rangle\\
		&=\int \chi_R^2\langle jv_{\epsilon},j(v_{\epsilon}/|v_{\epsilon}|)-d\varphi_{\epsilon}\rangle+\int \chi_R^2\langle jv_{\epsilon},d\varphi_{\epsilon}\rangle\\
		&\leq \int \chi_R^2|jv_{\epsilon}|\cdot \frac{C}{\rho_{\epsilon}}-\int \varphi_{\epsilon}\langle jv_{\epsilon},d(\chi_R^2)\rangle\\
		&\leq C\|\chi_R\,jv_{\epsilon}\|_{L^2} \left(\int \rho_{\epsilon}^{-2}\chi_R^2+\|\varphi_{\epsilon}\,d\chi_R\|_{L^{\infty}}^2\cdot |\Omega_R|\right)^{1/2}.
	\end{align*}
	Now, since $\chi_R$ is supported on the set $\Omega_R=\{R< \rho_{\epsilon}< 5R\}$, whose area is $\leq CR^2$, we see that
	$$\int \rho_{\epsilon}^{-2}\chi_R^2\leq \int_{\Omega_R}\frac{C}{R^2}\leq C.$$
	Moreover, it follows from \cref{re.psi.est} that $|\varphi_{\epsilon}|\leq \frac{C}{\leps }$ on $\Omega_R$ for $R\geq 2\epsilon$, so that
	$$\|\varphi_{\epsilon}\,d\chi_R\|_{L^{\infty}}^2\cdot |\Omega_R|\leq \frac{C}{\leps ^2}\cdot \frac{C}{R^2}\cdot CR^2\leq \frac{C}{\leps ^2}.$$
	Finally, since $|v_{\epsilon}|\leq 1$, we have
	$$\|\chi_R\,jv_{\epsilon}\|_{L^2}\leq \|\chi_R|v_{\epsilon}|^{-1}\,jv_{\epsilon}\|_{L^2},$$
	and putting together the preceding computations gives
	$$\|\chi_R|v_{\epsilon}|^{-1}\,jv_{\epsilon}\|_{L^2}^2\le C\|\chi_R|v_{\epsilon}|^{-1}\,jv_{\epsilon}\|_{L^2},$$
	and hence
	\begin{equation}\label{ann.bd.delpino}
		\int_{\Omega_R'}|v_{\epsilon}|^{-2}|jv_{\epsilon}|^2\leq \int \chi_R^2|v_{\epsilon}|^{-2}|jv_{\epsilon}|^2\leq C.
	\end{equation}
	
	Now, applying \cref{a.subset}, it follows that
	\begin{align*}
		\int_{D_{\epsilon^{-\tau}}\times S^1}|v_{\epsilon}|^{-2}|jv_{\epsilon}|^2&=\int_{\{\rho_{\epsilon}\leq 4\epsilon\}}|v_{\epsilon}|^{-2}|jv_{\epsilon}|^2+\int_A|v_{\epsilon}|^{-2}|jv_{\epsilon}|^2\\
		&\leq C+{\textstyle\sum_{j=1}^{J_{\epsilon,\tau}}}\int_{\Omega_{2^j\epsilon}'}|v_{\epsilon}|^{-2}|jv_{\epsilon}|^2\\
		&\leq C+C J_{\epsilon,\tau}\\
		&\leq C\log(1/\epsilon^{\tau+1}),
	\end{align*}
	and since we have already shown (in \cref{v.w.bd} and \cref{d.abs.delpino}) that
	$$\int_{D_{\epsilon^{-\tau}}\times S^1}(e_{\epsilon}(u_{\epsilon})-|v_{\epsilon}|^{-2}|jv_{\epsilon}|^2)\leq C(\tau),$$
	it follows that
	$$\int_{D_{\epsilon^{-\tau}}\times S^1}e_{\epsilon}(u_{\epsilon})\leq C\log(1/\epsilon^{\tau+1})+C(\tau)\le C\log(1/\epsilon^{\tau+1})$$
	for $\epsilon<\epsilon_0(\tau)$ sufficiently small, as claimed.
	
	Moreover, for any $\delta\in(0,1)$ and $\epsilon<\epsilon_0(\delta)$ sufficiently small, it follows from \cref{smallann.subset} and \cref{ann.bd.delpino} that
	\begin{align*}
		\int_{[D_{\epsilon^{-\tau}}\setminus D_{\delta\epsilon^{-\tau}}]\times S^1}|v_{\epsilon}|^{-2}|jv_{\epsilon}|^2
		&\leq {\textstyle\sum_{j=I_{\epsilon,\tau,\delta}}^{J_{\epsilon,\tau}}}\int_{\Omega'_{2^j\epsilon}}|v_{\epsilon}|^{-2}|jv_{\epsilon}|^2\\
		&\leq {\textstyle\sum_{j=I_{\epsilon,\tau,\delta}}^{J_{\epsilon,\tau}}}C\\
		&=C(\lceil \log(\epsilon^{-\tau-1})/\log 2\rceil-(\lfloor \log(\delta \epsilon^{-\tau-1})/\log(2)\rfloor-3))\\
		&\leq C(5-\log(\delta)),
	\end{align*}
	hence
	$$\int_{[D_{\epsilon^{-\tau}}\setminus D_{\delta\epsilon^{-\tau}}]\times S^1}e_{\epsilon}(u_{\epsilon})\leq C(\delta,\tau),$$
	completing the proof of the claim.
\end{proof}

\appendix
\section*{Appendix.}
\renewcommand{\thesection}{A}
\setcounter{theorem}{0}
\setcounter{equation}{0}

In this appendix we collect some fundamental estimates for maps $u:B_1^n(0)\to\C$ which are critical for the Ginzburg--Landau energy
$$E_\epsilon(u)=\int_{B_1}\left(\frac{|du|_g^2}{2}+\frac{W(u)}{\epsilon^2}\right) d\,\operatorname{vol}_g,$$
with respect to a smooth Riemannian metric $g$, defined on the closure $\bar B_1$.
Recall that $u$ solves the nonlinear elliptic equation
\begin{align*}
	&\epsilon^2\Delta_g u+(1-|u|^2)u=0.
\end{align*}
In the proof of these results a central ingredient, which also appears in our arguments, is the following \emph{monotonicity formula} (see, e.g., \cite[Proposition~A.1]{Stern.jdg}).

\begin{proposition}
	For any $x\in B_1^n(0)$, denoting by $\mathcal{B}_s(x)$ the geodesic ball with respect to $g$, we have
	\begin{align}\label{monotonicity}
		&\frac{d}{ds}\left(e^{C(g)s^2}\frac{E_\epsilon(u;\mathcal{B}_s(x))}{s^{n-2}}\right)
		\ge\frac{1}{s^{n-2}}\int_{\de\mathcal{B}_s(x)}|\de_\nu u|_g^2+\frac{1}{s^{n-1}}\int_{\mathcal{B}_s(x)}\frac{2W(u_\epsilon)}{\epsilon^2}
	\end{align}
	for all $s\in(0,\operatorname{inj}_g(x))$, where we omit the volume element of $g$. In particular,
	\begin{align*}
		&s\mapsto e^{C(g)s^2}\frac{E_\epsilon(u;\mathcal{B}_s(x))}{s^{n-2}}
	\end{align*}
	is an increasing function of the radius $s\in(0,\operatorname{inj}_g(x))$.
\end{proposition}

Note that the constant $C(g)\to 1$, when we let $g$ converge smoothly to the Euclidean metric.
We also record some useful pointwise bounds for $u$ and its differential.

\begin{proposition}
	Assuming $\epsilon\le 1$, on the smaller ball $B_{1/2}^n(0)$ we have
	\begin{align}\label{bound.general.critical}
		&|u(x)|\le 1+C(g,n)\epsilon^2,\quad |du|_g\le \frac{C(g,n)}{\epsilon}.
	\end{align}
	Also, if the energy $E_\epsilon(u)\le\Lambda\leps $, then on $B_{1/2}^n(0)$
	\begin{align}\label{bound.general.critical2}
		&|du|_g^2\le\frac{1-|u|^2}{\epsilon^2}+C(g,\Lambda,n).
	\end{align}	
\end{proposition}

\begin{proof}
	The function $\rho:=|u|$ satisfies
	\begin{align*}
		&-\Delta_g\rho+\frac{(\rho+1)(\rho-1)\rho}{\epsilon^2}\le 0,
	\end{align*}
	while it is easy to check that, for any fixed $s\in(\mz,1)$, $b_\gamma(x):=1+\gamma\epsilon^2\frac{s^2}{(s^2-|x|^2)^2}$ is a supersolution on $B_s=B_s^n(0)$, for $\gamma\ge\bar\gamma(g,n)>0$ large enough.
	
	On $B_s$, we have $\rho\le b_\gamma$ for some least $\gamma\ge 0$.
	However, we cannot have $\gamma\ge\bar\gamma$, since then the supersolution $b_\gamma$ would touch the subsolution $\rho$ from above (at an interior point), violating the maximum principle for semilinear equations.
	
	Thus, we must have $\rho\le b_{\bar\gamma}$ on $B_s$, and letting $s\to 1$ we get
	$$|u(x)|\le 1+C(g,n)\frac{\epsilon^2}{(1-|x|^2)^2}$$
	on $B_1$, from which the first half of \cref{bound.general.critical} follows.
	Using also the equation, it follows that $|u|\le C(g,n)$ and $|\Delta_g u|\le\frac{C(g,n)}{\epsilon^2}$ on $B_{15/16}$,
	which easily imply the bound
	$$|du|_g\le\frac{C(g,n)}{\epsilon}$$
	on $B_{7/8}$. Indeed, the bounds $|\tilde u|+|\Delta_{\tilde g}\tilde u|\le C(g,n)$ for $\tilde u(x):=u(\epsilon x)$ on $B_{15/16\epsilon}$ (with the rescaled metric $\tilde g$)
	easily give the desired bound $|d\tilde u|_{\tilde g}\le C(g,n)$ on any ball $B_{1/8}(x)\subseteq B_{7/8\epsilon}(0)$, hence on $B_{7/8\epsilon}$ (as $\epsilon\le 1$).
	
	It is interesting to observe that, even without assumptions on the energy of $u$ on $B_1$, the previous inequalities give $E_\epsilon(u;B_{1/2})\le \frac{C(g,n)}{\epsilon^2}$
	(which is sharp, for the trivial unstable solution $u\equiv 0$).
	
	In order to improve on the previous pointwise bound for $|du|=|du|_g$, we observe that
	\begin{align*}
		&\Delta_g\frac{|du|^2}{2}
		\ge \langle du,d\Delta_g u\rangle+\operatorname{Ric}_g(du,du)
		\ge -\frac{1-|u|^2}{\epsilon^2}|du|^2-\|\operatorname{Ric}_g\|_{L^\infty}|du|^2
	\end{align*}
	by Bochner's formula, and
	\begin{align*}
		&\Delta_g\frac{1-|u|^2}{2\epsilon^2}
		=\frac{|u|^2}{\epsilon^2}\cdot\frac{1-|u|^2}{\epsilon^2}-\frac{|du|^2}{\epsilon^2}.
	\end{align*}
	As a consequence, the difference
	$$f:=\frac{|du|^2}{2}-(1+\epsilon^2\|\operatorname{Ric}_g\|_{L^\infty})\frac{1-|u|^2}{2\epsilon^2}$$
	satisfies
	$$\Delta_g f\ge\frac{2|u|^2}{\epsilon^2}f.$$
	In particular, the positive part $f^+$ is subharmonic, and it follows that
	$$f\le C(g,\Lambda,n)\leps $$
	on $B_{7/8}$.
	Also, by the bound $|du|\le\frac{C(g,n)}{\epsilon}$ and \cref{bos2} below, we have
	\begin{align*}
		&\int_{B_{3/4}\cap\{|u|\le\mz\}}f^+
		\le C(g,n)\int_{B_{3/4}}\frac{W(u)}{\epsilon^2}
		\le C(g,\Lambda,n).
	\end{align*}
	On the other hand, the subequation for $f$ easily implies that
	\begin{align*}
		&\int_{B_{7/8}}\varphi^2\left[|df^+|^2+\frac{|u|^2}{\epsilon^2} (f^+)^2\right]
		\le C\int_{B_{7/8}} (f^+)^2|d\varphi|^2,
	\end{align*}
	for any $\varphi\in C^\infty_c(B_{7/8}^n)$. In particular, by Cauchy--Schwarz,
	\begin{align*}
		&\int_{B_{3/4}\cap\{|u|\ge\mz\}}f^+
		\le C(g,n)\left[\int_{B_{7/8}\cap\{|u|\ge\mz\}}|u|^2(f^+)^2\right]^{1/2}
		\le C(g,\Lambda,n)\epsilon\leps ,
	\end{align*}
	where we used the bound $0\le f^+\le C(g,\Lambda,n)\leps $ on $B_{7/8}$.
	Together with the previous bound, using again the subharmonicity of $f^+$, we arrive at
	\begin{align*}
		&f
		\le C(g,n)\int_{B_{3/4}}f^+
		\le C(g,\Lambda,n)(1+\epsilon\leps )
	\end{align*}
	on $B_{1/2}$, which gives \cref{bound.general.critical2}.
\end{proof}

In the asymptotic analysis, the most fundamental tool is the \emph{clearing-out} for the vorticity, which we state here for arbitrary metrics (the proof is a simple localization of the arguments from \cite[Section~4.3]{Stern.thesis}).

\begin{theorem}\label{clearing}
	Given $\beta\in(0,1)$, there exist constants $\eta(\beta,n)$ and $c(\beta,g,n)$ such that, for a geodesic ball $\mathcal{B}_r(x)\subseteq B_1^n(0)$ with $\epsilon\le r\le c$, if $E_\epsilon(u;\mathcal{B}_r(x))\le \eta r^{n-2}\log(r/\epsilon)$, then $|u(x)|>\beta$.
\end{theorem}

As we saw in \cref{ju.main}, the logarithmic growth of the energy exhibited by typical solutions $u$ is caused solely by the angular part $ju=u^*(r^2\,d\theta)=u^1\,du^2-u^2\,du^1$ of the differential. This fact relies on two inequalities: first of all, we can bound the radial part $d|u|$ in terms of the potential as follows (see, e.g., the argument from \cite[pp.~329--331]{BOS}, which readily generalizes to arbitrary metrics).

\begin{proposition}\label{bos1}
	On the smaller ball $B_{1/2}=B_{1/2}^n(0)$ we have
	\begin{align}
		&\int_{B_{1/2}}|d|u||^2
		\le C(g,n)\int_{B_{1}}\frac{(1-|u|^2)^2}{4\epsilon^2}+C(g,n)\epsilon^2,
	\end{align}
	provided that $\epsilon\le 1$.
\end{proposition}

Also, we have
the following sharp bound, which constitutes one of the main contributions of \cite{BOS}, and allows to deduce the same bound for the previous integral of $|d|u||^2$.

\begin{proposition}\label{bos2}
	On the smaller ball $B_{1/2}=B_{1/2}^n(0)$ we have
	\begin{align}
		&\int_{B_{1/2}}\frac{(1-|u|^2)^2}{4\epsilon^2}
		\le C(g,n)\frac{E_\epsilon(u;B_1)}{\leps }\log\left(2+\frac{E_\epsilon(u;B_1)}{\leps }\right),
	\end{align}
	provided that $\epsilon\le c$ and $E_\epsilon(u;B_1)\le\epsilon^{-\alpha_0}$, for some $c=c(g,n)$ and $\alpha_0=\alpha_0(n)$.
	In particular, assuming $E_\epsilon(u;B_1)\le\Lambda\leps $, it follows that
	\begin{align*}
		&\int_{B_{1/2}}\frac{(1-|u|^2)^2}{4\epsilon^2}
		\le C(g,\Lambda,n)
	\end{align*}
	for $\epsilon$ small enough.
\end{proposition}

The proof relies on a covering argument using \cref{clearing} (see \cite[pp.~323--328]{BOS}\footnote{Note that (2.3) in \cite[Proposition~2.2]{BOS} should read $|u_\epsilon(x)|\le 1+\frac{C\epsilon^2}{\operatorname{dist}(x,\de\Omega)^2}$ (which follows from the bound \cref{bound.general.critical} in the present paper, by scaling) and that the assumption in \cite[Proposition~2.2 and Proposition~2.3]{BOS} should be $\operatorname{dist}(x,\de\Omega)>\epsilon$.}),
and adapts to arbitrary metrics with straightforward modifications,
using balls with respect to $g$ in the statement of \cite[Proposition~2.4]{BOS} (see \cite[Thereom~2.8.14]{Federer} for a proof of the Besicovitch covering theorem on Riemannian manifolds).

The conclusion then follows from an estimate off the vorticity set $\{|u|\le 1-\sigma_0\}$, for some $\sigma_0$ small enough
(see \cite[Theorem~2.1]{BOS}\footnote{We point out the following misprints: in (A.5), $a=\frac{(1-\theta_\epsilon)(2-\theta_\epsilon)}{\epsilon^2}\ge\frac{1}{2\epsilon^2}$ (we assume $|u_\epsilon|\ge\mz$);
	in equations (A.11)--(A.12) some signs are wrong, but this does not affect the argument; most importantly, in (A.21) the right-hand side is just $C\|e_\epsilon(u_\epsilon)\|_{L^1(B_1)}^{1/2}$
	but, assuming (without loss of generality) $\frac{q}{q-2}\ge 2$, the last estimate on p.~347 still implies
	(A.23) with $\beta_q=(2-\alpha_0)\frac{q-2}{q}\in(0,1)$, as well as (A.25) with the same $\beta_q$ (by (A.20) with $q=2$).}).

On an unrelated note, we also record the following useful Lorentz estimate for a Riesz potential, which is used in the proof of \cref{slice.perp.small}. Recall that, for a function $f:\R^m\to\R$, its $L^{2,\infty}$-quasinorm is defined as
\begin{align*}
	&\|f\|_{L^{2,\infty}(\R^m)}=\sup_{\lambda>0}\lambda|\{|f|>\lambda\}|^{1/2}.
\end{align*}

\begin{proposition}\label{lorentz}
	If $f,g:\R^n\to\R$ satisfy
	\begin{align*}
		&|f|\le\frac{1}{|x|^{n-1}}*|g|,
	\end{align*}
	then for any $y\in\R^{n-2}$ we have
	\begin{align}\label{lor.bd}
		&\|f(y,\cdot)\|_{L^{2,\infty}(\R^2)}
		\le C(n)\sup_{r>0}\frac{1}{r^{n-2}}\int_{B_r^{n-2}(y)\times\R^2}|g|.
	\end{align}
\end{proposition}

Thus, the exponent $\frac{n}{n-1}$ in the classical Sobolev bound $\|f\|_{L^{n/(n-1),\infty}}\le C(n)\|g\|_{L^1}$ can be improved to $2$ (the exponent that we have on the plane), on a slice $\{y\}\times\R^2$, provided that we control the maximal function on the right-hand side of \cref{lor.bd}.

The proof is presented in \cite[Lemma~A.2]{LR} when $n=3$, but it is straightforward to adapt it to the case of general $n$.

Finally, we briefly show how one can obtain precise asymptotics for the (local) Green function of $\Delta_H$, the Hodge Laplacian on $k$-forms,
even when the metric is not Euclidean. Let $U\subset\R^n$ be a bounded smooth domain ($n\ge 3$), together with a smooth metric $g$ on $\bar U$.
Let us fix an orthonormal frame $(\omega_i)_{i\in I}$ for the bundle of $k$-forms on $\bar U$.

\begin{proposition}\label{green}
	Given a compact subset $K\subset U$,
	there exists $G_{i,p}\in\Omega^k(U\setminus\{p\})$ for every $p\in K$, satisfying
	$$\Delta_H G_{i,p} = \delta_p\cdot\omega_{i}(p)$$
	on $U$, in the distributional sense, and such that the difference $$H_{i,p}(q):=G_{i,p}(q)-\bar G(\dist(p,q))\,\omega_{i,p}(q)$$ obeys the bounds
	$$|H_{i,p}(q)|\le C\dist(p,q)^{3-n},\quad|\nabla H_{i,p}(q)|\le C\dist(p,q)^{2-n}$$
	for $q\in U$,
	for some constant $C=C(g,K,U)$, where $\bar G(r):=\frac{1}{n(n-2)\omega_{n}r^{n-2}}$ is the standard Green function on $\R^n$ and $\dist(p,q)$ is the geodesic distance induced by $g$ (the constant $C\to 0$ when $g$ converges to the Euclidean metric in the smooth topology).
\end{proposition}

It is clear from the proof that $G_{i,p}(q)$ and $\nabla G_{i,p}(q)$ depend continuously on the couple $(p,q)$, away from the diagonal $\{p=q\}$. With this proposition in hand, we can then easily invert the Hodge Laplacian (locally): given $\eta\in\Omega^k(\bar U)$, the convolution
$$\beta(q):={\textstyle\sum_{i\in I}}\int_K G_{i,p}(q)\ang{\eta(p),\omega_{i}(p)}\,d\operatorname{vol}_g(p)$$
then satisfies $\Delta_H\beta=\eta$ on the interior of $K$, and the previous bounds for $H_{i,p}$ imply that $\beta$ resembles the usual convolution with the Euclidean Green function (at small scales, or when $g$ is almost flat).

\begin{proof}
	For any (smooth) differential form $\omega\in\Omega^k(\bar U)$ we can find a unique $\alpha\in\Omega^k(\bar U)$ such that
	$\Delta_H\alpha=\omega$, with each component of $\alpha$ vanishing at $\de U$. Such $\alpha$ can be obtained by minimizing the energy
	$$\alpha\mapsto\int_U\left(\frac{|d\alpha|^2}{2}+\frac{|d^*\alpha|^2}{2}-\ang{\alpha,\omega}\right)$$
	in the space $W^{1,2}_0(U,\Lambda^k\R^n)$ (note that $\|\alpha\|_{W^{1,2}}\le C(g,U)(\|d\alpha\|_{L^2}+\|d^*\alpha\|_{L^2})$ for $\alpha$ in this space, by \cite[Theorem~4.8]{ISS} and a simple compactness and contradiction argument).
	We have $\|\alpha\|_{W^{1,2}(U)}\le C(g,U)\|\omega\|_{L^2(U)}$ and, by standard elliptic regularity for systems,
	\begin{align}\label{lp.plain}
		&\|\alpha\|_{L^s(U)}\le C(g,s,U)\|\omega\|_{L^t(U)}
	\end{align}
	for all $s,t\in(1,\infty)$ such that $\frac{1}{s}>\frac{1}{t}-\frac{2}{n}$.
	
	Fix a cutoff function $\chi\in C^\infty_c(U)$ with $\chi\equiv 1$ near $K$, and let
	$$\tilde G_{i,p}(q):=\chi(q)\bar G(\dist(p,q))\,\omega_{i,p}(q)$$
	for any fixed $p\in K$, where $\omega_{i,p}\in\Omega^k(\bar U)$ is such that $\omega_{i,p}(p)=\omega_i(p)$ and $\nabla\omega_{i,p}(p)=0$.
	Using normal coordinates centered at $p$, it is easy to check that
	$$|\Delta_H \tilde G_{i,p}|\le C(g,K,U)\dist(p,q)^{2-n};$$
	hence, $\Delta_H \tilde G_{i,p}$ coincides with a $k$-form $\varphi_{i,p}\in L^t(U)$ on $U\setminus\{p\}$,
	where $t\in(1,\frac{n}{n-2})$.
	
	On the other hand, an integration by parts shows that
	$$\Delta_H \tilde G_{i,p}=\delta_{p}\cdot\omega_{i,p}(p)+\varphi_{i,p}=\delta_p\cdot\omega_i(p)+\varphi_{i,p}$$
	on $U$, in the distributional sense.
	As explained above, by approximating $\varphi_{i,p}$ with smooth $k$-forms, we can then find $\alpha=\alpha_{i,p}$ such that $\Delta_H\alpha_{i,p}=\varphi_{i,p}$ and \cref{lp.plain} holds (with $\omega:=\varphi_{i,p}$).
	To conclude the proof, we show that $|\nabla\alpha(q)|\le C\dist(p,q)^{2-n}$ for some $C=C(g,K,U)$; the conclusion will follow by taking $G_{i,p}:=\tilde G_{i,p}-\alpha_{i,p}$.
	
	But indeed, considering the rescaled $k$-form $\alpha_r(x):=\alpha(p+rx)$, we see that
	$$\|\Delta_H\alpha_r\|_{L^\infty(A)}\le Cr^2\|\varphi_{i,p}\|_{L^\infty(B_r\setminus B_{r/2})}\le Cr^{4-n},\quad\|\alpha_r\|_{L^s(A)}\le C(s)r^{-n/s}$$
	whenever $\frac{1}{s}>1-\frac{4}{n}$, for the annular region $A:=B_1\setminus B_{1/2}$, provided that $r$ is small enough (with $C$ depending also on $g,K,U$).
	By standard elliptic regularity, we then obtain $|\nabla\alpha_r|\le C(s)(r^{4-n}+r^{-n/s})$ on $A$,
	which gives
	$$|\nabla\alpha(q)|\le C(s)(\dist(p,q)^{3-n}+\dist(p,q)^{-1-n/s}).$$
	Taking $s$ sufficiently close to $\frac{n}{n-4}$ (if $n>4$, or to $\infty$ if $n=4$) gives the claim for $n\ge 4$;
	when $n=3$, from $\Delta_H\alpha=\varphi_{i,p}$ we can immediately conclude that $|\alpha|\le C$, and we can take $s:=\infty$ in the previous bound to conclude that $|\nabla\alpha(q)|\le C\dist(p,q)^{-1}$.
\end{proof}


\end{document}